\documentclass[microtype]{gtpart}     
\usepackage{tikz}
\usepackage{color}
\usepackage{enumitem}
\usepackage{graphicx}
\usepackage{amsfonts}

\title{DR/DZ equivalence conjecture and tautological relations}

\author{Alexandr Buryak}
\givenname{Alexandr}
\surname{Buryak}
\address{School of Mathematics, University of Leeds, Leeds LS2 9JT, United Kingdom}
\email{a.buryak@leeds.ac.uk}
\urladdr{https://sites.google.com/site/alexandrburyakhomepage/home}

\author{J\'er\'emy Gu\'er\'e}
\givenname{J\'er\'emy}
\surname{Gu\'er\'e}
\address{Universit\'e Grenoble Alpes, CNRS, Institut Fourier, F-38000 Grenoble, France}
\email{jeremy.guere@gmail.com}
\urladdr{https://www-fourier.ujf-grenoble.fr/~guerej/index.html}

\author{Paolo Rossi}
\givenname{Paolo}
\surname{Rossi}
\address{Dipartimento di Matematica ``Tullio Levi-Civita'', Universit\`a degli Studi di Padova, Via Trieste~63, 35121 Padova, Italy}
\email{paolo.rossi@math.unipd.it}
\urladdr{http://www.math.unipd.it/~rossip/}



\arxivreference{1705.03287}

\newtheorem{theorem}{Theorem}[section]
\newtheorem{proposition}[theorem]{Proposition}
\newtheorem{lemma}[theorem]{Lemma}
\newtheorem{corollary}[theorem]{Corollary}
\newtheorem{conjecture}[theorem]{Conjecture}

\theoremstyle{definition}

\newtheorem{definition}[theorem]{Definition}

\newtheorem{remark}[theorem]{Remark}

\newcommand{\mbP}{\mathbb P}
\newcommand{\mbZ}{\mathbb Z}
\newcommand{\mbC}{\mathbb C}
\newcommand{\cP}{\mathcal P}
\newcommand{\oM}{\overline{\mathcal M}}

\newcommand{\tu}{{\widetilde u}}

\newcommand{\oh}{\overline h}
\newcommand{\hLambda}{\widehat\Lambda}

\def\cM{{\mathcal{M}}}
\def\oM{{\overline{\mathcal{M}}}}

\def\mbQ{{\mathbb Q}}
\def\d{{\partial}}

\newcommand{\<}{\left<}
\renewcommand{\>}{\right>}
\newcommand{\eps}{\varepsilon}

\newcommand{\hcA}{\widehat{\mathcal A}}
\newcommand{\DR}{\mathrm{DR}}
\newcommand{\DZ}{\mathrm{DZ}}

\newcommand{\even}{\mathrm{even}}
\newcommand{\ct}{\mathrm{ct}}
\newcommand{\cF}{\mathcal F}

\newcommand{\Coef}{\mathrm{Coef}}

\newcommand{\Desc}{\mathrm{Desc}}
\newcommand{\ST}{\mathrm{ST}}
\newcommand{\tv}{\widetilde v}
\renewcommand{\top}{\mathrm{top}}
\newcommand{\red}{\mathrm{red}}

\newcommand{\gl}{\mathrm{gl}}

\newcommand{\tA}{\widetilde{A}}

\newcommand{\tGamma}{\widetilde{\Gamma}}

\newcommand{\tl}{\widetilde{l}}
\newcommand{\sol}{\mathrm{sol}}
\newcommand{\cR}{\mathcal{R}}
\newcommand{\NN}{\mathbb{N}}
\newcommand{\ee}{\mathrm{e}}
\newcommand{\hA}{\widehat{A}}
\newcommand{\ta}{\widetilde{a}}
\newcommand{\st}{\mathrm{st}}

\renewcommand{\gg}[2]{\fill[color=white] (#2) circle(2mm) node {\color{black}$\substack{#1}$}; \draw (#2) circle (2mm)}
\newcommand{\lab}[4]{\draw (#1)++(#2:#3) node {$\substack{#4}$};}
\newcommand{\leg}[2]{\begin{scope}[shift={(#1)}] \draw (0:0) -- (#2:6.1mm);\end{scope}}
\newcommand{\legm}[3]{\begin{scope}[shift={(#1)}] \draw (0:0) -- (#2:7.8mm);\fill[color=white] (#2:7.8mm) circle(1.7mm) node {\color{black}$\substack{#3}$};\end{scope}}
\newcommand{\lp}[2]{\begin{scope}[shift={(#1)}] \draw (#2:0.3) circle(0.3);\end{scope}}
\newcommand{\doubleedge}{\draw plot [smooth,tension=1] coordinates {(A) (AB1) (B)}; \draw plot [smooth,tension=1.5] coordinates {(A) (AB2) (B)};}
\newcommand{\doubleedgeB}{\draw plot [smooth,tension=1] coordinates {(B) (BC1) (C)}; \draw plot [smooth,tension=1.5] coordinates {(B) (BC2) (C)}}
\renewcommand{\ggg}[2]{\fill[color=white] (#2) circle(4mm) node {\color{black}$#1$}; \draw (#2) circle (4mm)}
\newcommand{\legg}[3]{\begin{scope}[shift={(#1)}] \draw (0:0) -- (#2:#3);\end{scope}}
\newcommand{\legmm}[4]{\begin{scope}[shift={(#1)}]\draw (0:0) -- (#2:#3);\fill[color=white] (#2:#3) circle(2mm) node {\color{black}$#4$};\end{scope}}

\begin{document}

\tikz{\coordinate (AAA) at (-0.8,0);\coordinate (AA) at (0,0.8);\coordinate (A) at (0,0);\coordinate (B) at (0.8,0);\coordinate (C) at (1.6,0);\coordinate (D) at (2.4,0);\coordinate (AB1) at (0.4,0.18);\coordinate (AB2) at (0.4,-0.18);\coordinate (E) at (-0.69,0.4);\coordinate (F) at (-0.69,-0.4);\coordinate (BC1) at (1.2,0.18);\coordinate (BC2) at (1.2,-0.18);\coordinate (G) at (-1.38,0);\coordinate (H) at (-2.18,0);\coordinate (I) at (-2.98,0);}

\begin{abstract}    
In this paper we present a family of conjectural relations in the tautological ring of the moduli spaces of stable curves which implies the strong double ramification/Dubrovin-Zhang equivalence conjecture introduced in \cite{BDGR16a}. Our tautological relations have the form of an equality between two different families of tautological classes, only one of which involves the double ramification cycle. We prove that both families behave the same way upon pullback and pushforward with respect to forgetting a marked point. We also prove that our conjectural relations are true in genus $0$ and $1$ and also when first pushed forward from~$\oM_{g,n+m}$ to~$\oM_{g,n}$ and then restricted to $\cM_{g,n}$, for any $g,n,m\geq 0$. Finally we show that, for semisimple CohFTs, the DR/DZ equivalence only depends on a subset of our relations, finite in each genus, which we prove for $g\leq 2$. As an application we find a new formula for the class $\lambda_g$ as a linear combination of dual trees intersected with kappa and psi classes, and we check it for $g \leq 3$.
\end{abstract}

\maketitle


\section{Introduction}

A cohomological field theory (CohFT) $c_{g,n}$ is a family of cohomology classes on the moduli spaces  $\oM_{g,n}$ of genus $g$ stable curves with $n$ marked points (parameterized by $n$ tensor copies of a vector space) which satisfy certain compatibility axioms with respect to the natural morphisms among different moduli spaces. They were introduced by Kontsevich and Manin~\cite{KM94} to axiomatize the properties of Gromov-Witten classes for a given smooth projective variety, but have since then also proved to be a powerful probe for the cohomology and Chow rings of~$\oM_{g,n}$ itself, and their tautological subrings in particular \cite{PPZ15,Jan15,JPPZ16}.

Recall that the tautological rings $R^*(\oM_{g,n})$, for $g,n\geq 0$ satisfying $2g-2+n>0$, are the smallest $\mbQ$-subalgebras of $H^*(\oM_{g,n},\mbQ)$ closed under pushforward along the morphisms forgetting marked points and gluing two marked points together to form a node. $R^*(\oM_{g,n})$ is much smaller than the full cohomology ring, but still has a rich structure and contains most of the natural and geometrically interesting classes. The ring structure of $R^*(\oM_{g,n})$, however, is not yet completely under control. We know a system of additive generators, the so-called strata algebra, formed by basic classes which are represented by the closure of the loci of curves with fixed dual stable graph intersected with a given monomial in kappa and psi classes. The product of basic classes is explictly described, but the full system of relations is still unknown, although Pixton has found a large set of relations that is conjectured to be complete, see \cite{PPZ15}.

In this paper we present a new family of conjectural relations in the form of an equality between two families of tautological classes. We denote these classes in $R^*(\oM_{g,n})$ by $A^g_{d_1,\ldots,d_n}$ and $B^g_{d_1,\ldots,d_n}$, where the $n$ integer non-negative parameters $d_1,\ldots,d_n$ satisfy $2g-1\leq \sum d_i \leq 3g-3+n$. Their precise definition is given in Sections \ref{section:Aclasses} and~\ref{section:Bclasses} respectively, but here we stress that they can be described as two different linear combinations of stable trees with psi classes at the half-edges and, moreover, for the $A$-classes only, a double ramification cycle times the Hodge class $\lambda_g$ is attached at each vertex.

The motivation for this conjecture comes from the study of the double ramification (DR) hierarchy, an integrable system of Hamiltonian PDEs associated to a CohFT and involving the geometry of the DR cycle, introduced by the first author in~\cite{Bur15} and further studied in \cite{BR16a,BR16b,BG16,BDGR16a,BDGR16b} (see also \cite{Bur17,Ros17} for a review). In~\cite{BDGR16a}, sharpening a conjecture from \cite{Bur15}, it was conjectured that (the logarithm of) the tau-function of (a particular solution of) the DR hierarchy coincides with the reduced potential of the CohFT. The reduced potential is obtained from the full potential, i.e.~the generating series of the intersection numbers of the CohFT with monomials in the psi classes, by an explicit procedure, also described in \cite{BDGR16a}, which only depends on the potential itself and which ultimately forgets part of the information.

In case the CohFT is semisimple (a technical condition on its genus $0$ part), the conjecture translates into a statement about the relation between the DR hierarchy and the Dubrovin--Zhang hierarchy, another, more classical, construction associating an integrable system to a semisimple CohFT for which we have the Witten-type result that (the logarithm of) the tau-function of (a special solution of) the DZ hiearchy coincides with the potential of the CohFT.

In this case the strong DR/DZ equivalence conjecture states that the two hierarchies are related by a normal Miura transformation, i.e.~a change of coordinates preserving the tau-structure, and hence acting in particular on the tau-functions. This action on the tau-functions precisely corresponds to the reduction procedure described above for the potential of the CohFT.

As we have seen, the DR/DZ equivalence conjecture is about intersection numbers, not cohomology classes. However in Section~\ref{section:DR/DZ} we show how the coefficients of the two involved generating series, the (logarithm of the) DR tau-function and the reduced potential of the CohFT, are the intersection numbers of the CohFT with two different families of cohomology classes. These two families are precisely the $A$- and $B$-classes above. So the DR/DZ equivalence conjecture states that the intersection numbers of the $A$- and $B$-classes with any (possibly non tautological) CohFT are equal:
$$
\int_{\oM_{g,n}} A^g_{d_1,\ldots,d_n} c_{g,n} = \int_{\oM_{g,n}} B^g_{d_1,\ldots,d_n} c_{g,n}.
$$
This motivates us to conjecture that it is the $A$- and $B$-classes themselves to be equal:
$$A^g_{d_1,\ldots,d_n} = B^g_{d_1,\ldots,d_n}.$$

In the rest of the paper we work towards the proof of such conjecture. In Section~\ref{section:further structure} we prove the string and dilaton equations for both $A$- and $B$-classes, establishing that their behaviour upon pullback and pushforward along the morphism $\pi:\oM_{g,n+1}\to \oM_{g,n}$ that forgets the last marked point is the same.

The string equation allows us to prove that the conjecture is true if and only if it is true when all the parameters $d_1,\ldots,d_n$ are strictly positive. This in turn yields a full proof of the conjecture in genus $0$ and genus $1$.

The dilaton equation is used to show that the relations in $R^*(\cM_{g,n})$ obtained by pushing forward our conjectural relations from $R^*(\oM_{g,n+m})$ to $R^*(\oM_{g,n})$ and then restricting them to~$R^*(\cM_{g,n})$ are valid. This is what we mean by saying that the conjecture is valid on $\cM_{g,n}$.

We then show that our relations imply in particular a new expression for the top Chern class of the Hodge bundle $\lambda_g$ as a linear combination of basic classes whose dual graph is a tree (with psi and kappa classes). No expression of this type for $\lambda_g$ was known before. We check its validity for $g\leq 3$.

Finally, in Section \ref{section:restricted} we show that, for semisimple CohFTs, the DR/DZ equivalence conjecture actually depends on just a subset of our conjectural relations, namely the ones for which $\sum d_i = 2g$ and $d_i>0$. This means that the number of relations one needs to check is finite in each genus, and equal to the number of partitions of $2g$.

In the appendix we check this finite subset of relations for $g=2$ thereby proving the strong DR/DZ equivalence conjecture in genus $g\leq 2$ for any semisimple CohFT.

\subsection*{Acknowledgements}

We would like to thank Boris Dubrovin, Rahul Pandharipande, Sergey Shadrin and Dimitri Zvonkine for useful discussions. 

A.~B. received funding from the European Union's Horizon 2020 research and innovation programme under the Marie Sk\l odowska-Curie grant agreement No 797635, was supported by Grant ERC-2012-AdG-320368-MCSK in the group of Rahul Pandharipande at ETH Zurich and Grant RFFI-16-01-00409. J. G. was supported by the Einstein foundation. P.~R.~was partially supported by a Chaire CNRS/Enseignement superieur 2012-2017 grant.

Part of the work was completed during the visit of J.~G. and P.~R. of the Forschungsinstitut f\"ur Mathematik at ETH Z\"urich in 2017.


\section{Tautological relations}

In this section we present our conjectural tautological relations.

\subsection{Tautological ring of $\oM_{g,n}$}\label{subsection:tautological ring}

Here we fix notations concerning tautological cohomology classes on $\oM_{g,n}$. We will use the notations from~\cite[Sections 0.2 and 0.3]{PPZ15}.

Recall that for any stable graph $\Gamma$ we have the associated moduli space
$$
\oM_{\Gamma}:=\prod_{v\in V(\Gamma)}\oM_{g(v),n(v)}
$$
and the canonical morphism 
$$
\xi_\Gamma\colon\oM_{\Gamma}\to\oM_{g(\Gamma),|L(\Gamma)|}.
$$
Recall~\cite{PPZ15} that given numbers $x_i[v],y[h]\ge 0$, $i\ge 1$, $v\in V(\Gamma)$, $h\in H(\Gamma)$, we can define a basic cohomology class on $\oM_{\Gamma}$ by
\begin{gather}\label{eq:basic class}
\gamma=\prod_{v\in V(\Gamma)}\prod_{i\ge 1}\kappa_i[v]^{x_i[v]}\cdot\prod_{h\in H(\Gamma)}\psi_h^{y[h]}\in H^*(\oM_\Gamma,\mbQ),
\end{gather}
where $\kappa_i[v]$ is the $i$-th kappa class on $\oM_{g(v),n(v)}$ and $\psi_h$ is the psi class on $\oM_{g(v(h)),n(v(h))}$. A cohomology class on $\oM_{g,n}$ of the form $\xi_{\Gamma*}(\gamma)$, where $\Gamma$ is a stable graph of genus~$g$ with~$n$ legs and $\gamma$ is a basic class on $\oM_{\Gamma}$, will be called a basic tautological class. Denote by $R^*(\oM_{g,n})$ the subspace of $H^*(\oM_{g,n},\mbQ)$ spanned by all basic tautological classes. The subspace $R^*(\oM_{g,n})$ is closed under multiplication and is called the tautological ring of the moduli space of curves. Let 
$$
R^i(\oM_{g,n}):=R^*(\oM_{g,n})\cap H^{2i}(\oM_{g,n},\mbQ).
$$
Denote by $\cM^\ct_{g,n}\subset\oM_{g,n}$ the moduli space of curves of compact type and by $\cM_{g,n}\subset\oM_{g,n}$ the moduli space of smooth curves. We will use the notations
$$
R^i(\cM^\ct_{g,n}):=\left.R^i(\oM_{g,n})\right|_{\cM_{g,n}^\ct},\qquad R^i(\cM_{g,n}):=\left.R^i(\oM_{g,n})\right|_{\cM_{g,n}}.
$$
Linear relations between basic tautological classes are called tautological relations. The class $\xi_{\Gamma *}(1)\in R^{|E(\Gamma)|}(\oM_{g(\Gamma),|L(\Gamma)|})$ will be called a boundary stratum.

We will represent a basic tautological class $\xi_{\Gamma*}(\gamma)$ on $\oM_{g(\Gamma),|L(\Gamma)|}$ by a picture of the graph $\Gamma$ where we put the monomial $\prod_{i\ge 1}\kappa_i[v]^{x_i[v]}$ next to each vertex $v$ and the power of the psi class~$\psi_h^{y[h]}$ next to each half-edge $h$. For example, we have the following well-known formulas:
\begin{align*}
&\psi_1=\tikz[baseline=-1mm]{\draw (A)--(B);\legm{A}{150}{1};\legm{A}{-150}{2};\legm{B}{30}{3};\legm{B}{-30}{4};\gg{0}{A};\gg{0}{B};}\in R^1(\oM_{0,4}),\\
&\lambda_1=\frac{1}{24}\tikz[baseline=-1mm]{\lp{A}{0};\legm{A}{180}{1};\gg{0}{A};}\in R^1(\oM_{1,1}),
\end{align*}
where we denote by $\lambda_i\in H^{2i}(\oM_{g,n},\mbQ)$ the $i$-th Chern class of the Hodge vector bundle over~$\oM_{g,n}$. It is well-known that the class $\lambda_i$ is tautological (see e.g.~\cite{FP00}).

Suppose $\Gamma_1$ and $\Gamma_2$ are two stable graphs, both of genus $g$ and with $n$ legs. They are called isomorphic, if there exists a pair $f=(f_1,f_2)$ of set isomorphisms $f_1\colon V(\Gamma_1)\to V(\Gamma_2)$ and $f_2\colon H(\Gamma_1)\to H(\Gamma_2)$ that preserve all the additional structure of the stable graphs. Suppose $\gamma_1$ and $\gamma_2$ are two basic classes on the spaces $\oM_{\Gamma_1}$ and $\oM_{\Gamma_2}$ respectively:
\begin{gather*}
\gamma_1=\prod_{v\in V(\Gamma_1)}\prod_{i\ge 1}\kappa_i[v]^{x_{1,i}[v]}\cdot\prod_{h\in H(\Gamma_1)}\psi_h^{y_1[h]},\qquad
\gamma_2=\prod_{v\in V(\Gamma_2)}\prod_{i\ge 1}\kappa_i[v]^{x_{2,i}[v]}\cdot\prod_{h\in H(\Gamma_2)}\psi_h^{y_2[h]}.
\end{gather*}
We will say that the pairs~$(\Gamma_1,\gamma_1)$ and~$(\Gamma_2,\gamma_2)$ are combinatorially equivalent, if there exists a pair of maps $f=(f_1,f_2)$, $f_1\colon V(\Gamma_1)\to V(\Gamma_2)$, $f_2\colon H(\Gamma_1)\to H(\Gamma_2)$, that defines an isomorphism between the stable graphs~$\Gamma_1$ and~$\Gamma_2$ and also satisfies the properties
\begin{align*}
x_{1,i}[v]=&x_{2,i}[f_1(v)],&&\text{for any $i\ge 1$ and $v\in V(\Gamma_1)$},\\
y_1[h]=&y_2[f_2(h)],&&\text{for any $h\in H(\Gamma_1)$}.
\end{align*} 
Obviously, if the pairs~$(\Gamma_1,\gamma_1)$ and~$(\Gamma_2,\gamma_2)$ are combinatorially equivalent, then $\xi_{\Gamma_1*}(\gamma_1)=\xi_{\Gamma_2*}(\gamma_2)$.

Consider the set of stable graphs of genus~$g$ with~$n$ legs. Suppose $I$ is a subset of $\{1,2,\ldots,n\}$. The symmetric group~$S_{|I|}$ acts on our set of stable graphs by permutations of markings from the set~$I$. This gives an $S_{|I|}$-action on the set of pairs $(\Gamma,\gamma)$, where~$\Gamma$ is a stable graph and~$\gamma$ is a basic class on~$\oM_{\Gamma}$. Let us fix some stable graph~$\Gamma$ and a basic class~$\gamma$. The sum of the basic tautological classes corresponding to combinatorially non-equivalent pairs in the $S_{|I|}$-orbit of the pair~$(\Gamma,\gamma)$ will be represented by the picture corresponding to the class $\xi_{\Gamma*}(\gamma)$, where we erase the labels from the set~$I$. Let us give two examples in order to illustrate this rule:
\begin{align*}
\tikz[baseline=-1mm]{\draw (A)--(B)--(C);\leg{A}{-120};\leg{A}{-60};\leg{C}{-120};\leg{C}{-60};\gg{0}{A};\gg{1}{B};\gg{0}{C};}=&
\tikz[baseline=-1mm]{\draw (A)--(B)--(C);\legm{A}{-120}{1};\legm{A}{-60}{2};\legm{C}{-120}{3};\legm{C}{-60}{4};\gg{0}{A};\gg{1}{B};\gg{0}{C};}
+\tikz[baseline=-1mm]{\draw (A)--(B)--(C);\legm{A}{-120}{1};\legm{A}{-60}{3};\legm{C}{-120}{2};\legm{C}{-60}{4};\gg{0}{A};\gg{1}{B};\gg{0}{C};}
+\tikz[baseline=-1mm]{\draw (A)--(B)--(C);\legm{A}{-120}{1};\legm{A}{-60}{4};\legm{C}{-120}{2};\legm{C}{-60}{3};\gg{0}{A};\gg{1}{B};\gg{0}{C};},\\
\tikz[baseline=-1mm]{\draw (A)--(B);\leg{A}{-90};\legm{B}{-120}{1};\leg{B}{-60};\gg{1}{A};\gg{1}{B};\lab{B}{-38}{3.5mm}{\psi};}=&\,\,
\tikz[baseline=-1mm]{\draw (A)--(B);\legm{A}{-90}{3};\legm{B}{-120}{1};\legm{B}{-60}{2};\gg{1}{A};\gg{1}{B};\lab{B}{-38}{3.5mm}{\psi};}
+\tikz[baseline=-1mm]{\draw (A)--(B);\legm{A}{-90}{2};\legm{B}{-120}{1};\legm{B}{-60}{3};\gg{1}{A};\gg{1}{B};\lab{B}{-38}{3.5mm}{\psi};}.
\end{align*}
As another useful example, we can write the topological resursion relations in genus~$0$ and~$1$:
\begin{align}
&\psi_1=\sum_{\substack{i+j=n-3\\i\ge 1,\,j\ge 0}}\tikz[baseline=-1mm]{\draw (A)--(B);\legm{A}{180}{1};\legm{B}{30}{2};\legm{B}{-30}{3};\leg{A}{-60};\leg{A}{-120};\leg{B}{-60};\leg{B}{-120};\gg{0}{A};\gg{0}{B};\lab{A}{-90}{5.9mm}{...};\lab{A}{-90}{7mm}{\underbrace{\phantom{aaa}}_{\text{$i$ legs}}};\lab{B}{-90}{5.9mm}{...};\lab{B}{-90}{7mm}{\underbrace{\phantom{aaa}}_{\text{$j$ legs}}};}\in R^1(\oM_{0,n}),\quad n\ge 4,\label{eq:TRR0}\\
&\psi_1=\frac{1}{24}\tikz[baseline=-1mm]{\legm{A}{180}{1};\leg{A}{-60};\leg{A}{-120};\lp{A}{90};\gg{0}{A};\lab{A}{-90}{5.9mm}{...};\lab{A}{-90}{7mm}{\underbrace{\phantom{aaa}}_{\text{$n-1$ legs}}};}
+\sum_{\substack{i+j=n-1\\i\ge 1,\,j\ge 0}}\tikz[baseline=-1mm]{\draw (A)--(B);\legm{A}{180}{1};\leg{A}{-60};\leg{A}{-120};\leg{B}{-60};\leg{B}{-120};\gg{0}{A};\gg{1}{B};\lab{A}{-90}{5.9mm}{...};\lab{A}{-90}{7mm}{\underbrace{\phantom{aaa}}_{\text{$i$ legs}}};\lab{B}{-90}{5.9mm}{...};\lab{B}{-90}{7mm}{\underbrace{\phantom{aaa}}_{\text{$j$ legs}}};}\in R^1(\oM_{1,n}).\label{eq:TRR1}
\end{align}

By stable tree we mean a stable graph $\Gamma$ with the first Betti number $b_1(\Gamma)$ equal to zero. Suppose $\Gamma$ is a stable tree. Let 
$$
H^e(\Gamma):=H(\Gamma)\backslash L(\Gamma).
$$
A path in~$\Gamma$ is a sequence of pairwise distinct vertices $v_1,v_2,\ldots,v_k\in V(\Gamma)$, $v_i\ne v_j$ for $i\ne j$, such that for any $1\le i\le k-1$ the vertices $v_i$ and $v_{i+1}$ are connected by an edge. For a vertex $v\in V(\Gamma)$ define a number $r(v)$ by
\begin{gather*}
r(v):=2g(v)-2+n(v).
\end{gather*}
Denote by $\ST^m_{g,n}$ the set of stable trees of genus~$g$ with~$m$ vertices and with~$n$ legs marked by numbers $1,\ldots,n$. For a stable tree $\Gamma\in\ST^m_{g,n}$ denote by~$l_i(\Gamma)$ the leg in~$\Gamma$ that is marked by~$i$. For a leg $l\in L(\Gamma)$ denote by $1\le i(l)\le n$ its marking.

A stable rooted tree is a pair $(\Gamma,v_1)$, where $\Gamma$ is a stable tree and $v_1\in V(\Gamma)$. The vertex~$v_1$ is called the root. Denote by~$H_+(\Gamma)$ the set of half-edges of $\Gamma$ that are directed away from the root $v_1$. Clearly, $L(\Gamma)\subset H_+(\Gamma)$. Let 
$$
H^e_+(\Gamma):=H_+(\Gamma)\backslash L(\Gamma).
$$
A vertex $w$ is called a descendant of a vertex $v$, if $v$ is on the unique path from the root $v_1$ to $w$. Note that according to our definition the vertex $v$ is a descendant of itself. Denote by $\Desc[v]$ the set of all descendants of~$v$. A vertex $w$ is called a direct descendant of~$v$, if $w\in\Desc[v]$, $w\ne v$ and $w$ and $v$ are connected by an edge. In this case the vertex $v$ is called the mother of~$w$. 

\subsection{Double ramification cycle and the definition of the $A$-class}\label{section:Aclasses}

Consider integers $a_1,\ldots,a_n$ such that $a_1+\ldots+a_n=0$. The double ramification cycle $\DR_g(a_1,\ldots,a_n)$ is a cohomology class in $H^{2g}(\oM_{g,n},\mbQ)$. If not all of $a_i$'s are equal to zero, then the restriction~$\left.\DR_g(a_1,\ldots,a_n)\right|_{\cM_{g,n}}$ can be defined as the Poincar\'e dual to the locus of pointed smooth curves~$(C,p_1,\ldots,p_n)$ satisfying $\mathcal O_C\left(\sum_{i=1}^n a_ip_i\right)\cong\mathcal O_C$, and we refer the reader, for example, to~\cite{BSSZ15} for the definition of the double ramification cycle on the whole moduli space~$\oM_{g,n}$. We will often consider the Poincar\'e dual to the double ramification cycle~$\DR_g(a_1,\ldots,a_n)$. It is an element of $H_{2(2g-3+n)}(\oM_{g,n},\mbQ)$ and, abusing our notations a little bit, it will also be denoted by $\DR_g(a_1,\ldots,a_n)$. 

The double ramification cycle $\DR_g(a_1,\ldots,a_n)$ is a tautological class on $\oM_{g,n}$ \cite{FP05}. A simple explicit formula for the restriction $\left.\DR_g(a_1,\ldots,a_n)\right|_{\cM_{g,n}^\ct}$ was derived in~\cite{Hai13,MW13}:
\begin{gather}\label{eq:Hain's formula}
\left.\DR_g(a_1,\ldots,a_n)\right|_{\cM_{g,n}^\ct}=\frac{1}{g!}\left(\sum_{i=1}^n \frac{a_i^2 \psi_i}{2}-\frac{1}{2}\sum_{\substack{I \subset \{1,\ldots,n\}\\|I|\ge 2}}a_I^2 \delta_0^I - \frac{1}{4} \sum_{I \subset \{1,\ldots,n\}} \sum_{h=1}^{g-1} a_I^2 \delta_h^I\right)^g,
\end{gather}
where for a subset $I\subset\{1,2,\ldots,n\}$ and a number $0\le h\le g$ we use the following notations:
\begin{align*}
&a_I:=\sum_{i\in I}a_i,\\
&\delta_h^I:=\tikz[baseline=-1mm]{\draw (A)--(B);\leg{A}{-60};\leg{A}{-120};\leg{B}{-60};\leg{B}{-120};\gg{h}{A};\gg{h'}{B};\lab{A}{-90}{5.9mm}{...};\lab{A}{-90}{7mm}{\underbrace{\phantom{aaa}}_{\text{$I$}}};\lab{B}{-90}{5.9mm}{...};\lab{B}{-90}{7mm}{\underbrace{\phantom{aaa}}_{\text{$I^c$}}};}\in R^1(\oM_{g,n}),\quad I^c:=\{1,2,\ldots,n\}\backslash I,\quad h':=g-h.
\end{align*}
Formula~\eqref{eq:Hain's formula} is usually referred as Hain's formula. It implies that the class $\left.\DR_g(a_1,\ldots,a_n)\right|_{\cM_{g,n}^\ct}$ is a polynomial in the variables $a_1,\ldots,a_n$ homogeneous of degree~$2g$. Since $\lambda_g|_{\oM_{g,n}\backslash\cM^\ct_{g,n}}=0$, we obtain that the class $\lambda_g\DR_g(a_1,\ldots,a_n)\in R^{2g}(\oM_{g,n})$ is a polynomial in $a_1,\ldots,a_n$ homogeneous of degree~$2g$. The full double ramification cycle is also polynomial, but not necessarily homogeneous~\cite{JPPZ16}. 

The following properties of the double ramification cycle will be useful for us. Let $\pi_i\colon\oM_{g,n+1}\to\oM_{g,n}$ be the forgetful map that forgets the $i$-th marked point. Then 
$$
\DR_g(a_1,\ldots,a_n,0)=\pi_{n+1}^*\DR_g(a_1,\ldots,a_n).
$$
Let $\pi\colon\oM_{g,n+g}\to\oM_{g,n}$ be the forgetful map that forgets the last~$g$ marked points. Then we have~\cite[Example 3.7]{BSSZ15}
\begin{gather}\label{eq:DR and fundamental}
\pi_*\DR_g(a_1,\ldots,a_{n+g})=g!a_{n+1}^2\cdots a_{n+g}^2[\oM_{g,n}].
\end{gather}
It is also useful to remember that (see e.g.~\cite{JPPZ16})
$$
\DR_g(0,0,\ldots,0)=(-1)^g\lambda_g\in R^g(\oM_{g,n}).
$$
We will denote by $\DR_g(a_1,\ldots,\widetilde{a_i},\ldots,a_n)$ the class $\pi_{i*}\DR_g(a_1,\ldots,a_n)\in R^{g-1}(\oM_{g,n-1})$. Recall the following important divisibility property.
\begin{lemma}[\cite{BDGR16a}]\label{lemma:divisibility}
Let $g,n\ge 1$. Then the polynomial class 
\begin{gather*}
\left.\DR_g\left(-\sum a_i,a_1,a_2,\ldots,\widetilde{a_n}\right)\right|_{\cM_{g,n}^\ct}\in R^{g-1}(\cM^{\ct}_{g,n})
\end{gather*}
is divisible by $a_n^2$.
\end{lemma}

Consider a stable tree $\Gamma\in\ST^m_{g,n}$ and integers $a_1,\ldots,a_n$ such that $a_1+\ldots+a_n=0$. To each half-edge $h\in H(\Gamma)$ we assign an integer~$a(h)$ in such a way that the following conditions hold:
\begin{itemize}

\item[a)] If $h\in L(\Gamma)$, then $a(h)=a_{i(h)}$;

\item[b)] If $h\in H^e(\Gamma)$, then $a(h)+a(\iota(h))=0$;

\item[c)] For any vertex $v\in V(\Gamma)$, we have $\sum_{h\in H[v]}a(h)=0$.

\end{itemize}
Clearly, such a function $a\colon H(\Gamma)\to\mbZ$ exists and is uniquely determined by the numbers~$a_1,\ldots,a_n$. For each moduli space $\oM_{g(v),n(v)}$, $v\in V(\Gamma)$, the numbers $a(h)$, $h\in H[v]$, define the double ramification cycle
$$
\DR_{g(v)}\left(A_{H[v]}\right)\in R^{g(v)}(\oM_{g(v),n(v)}).
$$
Here $A_{H[v]}$ denotes the list $a(h_1),\ldots,a(h_{n(v)})$, where $\{h_1,\ldots,h_{n(v)}\}=H[v]$. If we multiply all these cycles, we get the class
$$
\prod_{v\in V(\Gamma)}\DR_{g(v)}\left(A_{H[v]}\right)\in H^{2g}(\oM_\Gamma,\mbQ).
$$
We define a class $\DR_\Gamma(a_1,\ldots,a_n)\in R^{g+m-1}(\oM_{g,n})$ by 
\begin{gather*}
\DR_\Gamma(a_1,\ldots,a_n):=\xi_{\Gamma*}\left(\prod_{v\in V(\Gamma)}\DR_{g(v)}\left(A_{H[v]}\right)\right).
\end{gather*}
Clearly, the class
$$
\lambda_g\DR_\Gamma\left(a_1,\ldots,a_n\right)\in R^{2g+m-1}(\oM_{g,n})
$$
is a polynomial in $a_1,\ldots,a_n$ homogeneous of degree $2g$. 

Suppose now that $a_1,\ldots,a_n$ are arbitrary integers and let $a_0:=-\sum_{i=1}^n a_i$. Consider the set of stable trees $\ST^m_{g,n+1}$. It would be convenient for us to assume that the legs of stable trees from~$\ST^m_{g,n+1}$ are marked by $0,1,\ldots,n$. Let $\Gamma\in\ST_{g,n+1}^m$ be an arbitrary stable tree. Consider it as a rooted tree with the root $v_1(\Gamma):=v(l_0(\Gamma))$. As above, the numbers $a_0,a_1,\ldots,a_n$ define a function $a\colon H(\Gamma)\to\mbZ$. Define a coefficient $a(\Gamma)$ by
$$
a(\Gamma):=\left(\prod_{h\in H^e_+(\Gamma)}a(h)\right)\left(\prod_{v\in V(\Gamma)}\frac{r(v)}{\sum_{\tv\in\Desc[v]}r(\tv)}\right).
$$
Let $\pi\colon\oM_{g,n+1}\to\oM_{g,n}$ be the forgetful map that forgets the first marked point. Define a class $\tA^{g,m}(a_1,\ldots,a_n)\in R^{2g+m-2}(\oM_{g,n})$ by
$$
\tA^{g,m}(a_1,\ldots,a_n):=\sum_{\Gamma\in\ST^m_{g,n+1}}a(\Gamma)\lambda_g\pi_*\DR_\Gamma(a_0,a_1,\ldots,a_n).
$$
We know that this class is a polynomial in $a_1,\ldots,a_n$ homogeneous of degree $2g+m-1$. Note that the expression for the class $\tA^{g,1}(a_1,\ldots,a_n)$ is actually very simple:
$$
\tA^{g,1}(a_1,\ldots,a_n)=\lambda_g\DR_g(\widetilde{a_0},a_1,\ldots,a_n).
$$
\begin{lemma}
The polynomial class $\tA^{g,m}(a_1,\ldots,a_n)$ is divisible by $\sum_{i=1}^n a_i$.
\end{lemma}
\begin{proof}
If $m=1$, then the lemma follows from Lemma~\ref{lemma:divisibility}. Suppose $m\ge 2$ and $a_0=-\sum_{i=1}^n a_i=0$. We have to prove that 
$$
\tA^{g,m}(a_1,\ldots,a_n)=0.
$$
Consider a stable tree $\Gamma\in\ST^m_{g,n+1}$. If $g(v_1(\Gamma))\ge 1$, then, again by Lemma~\ref{lemma:divisibility},
$$
\lambda_g\pi_*\DR_\Gamma\left(0,a_1,\ldots,a_n\right)=0.
$$
If $g(v_1(\Gamma))=0$, then $\pi_*\DR_\Gamma\left(0,a_1,\ldots,a_n\right)$, unless $n(v_1(\Gamma))=3$. We obtain
\begin{gather}\label{eq:divisibility of A-class,tmp1}
\tA^{g,m}(a_1,\ldots,a_n)=\sum_{\substack{\Gamma\in\ST^m_{g,n+1}\\g(v_1(\Gamma))=0\\n(v_1(\Gamma))=3}}a(\Gamma)\lambda_g\pi_*\DR_\Gamma\left(0,a_1,\ldots,a_n\right).
\end{gather}

Let us define certain maps
$$
\ST^{m-1}_{g,n}\to\{\Gamma\in\ST^m_{g,n+1}|g(v_1(\Gamma))=0,\,n(v_1(\Gamma))=3\}.
$$
Note that we mark the legs of stable trees from $\ST^{m-1}_{g,n}$ by $1,\ldots,n$ and the legs of stable trees from $\ST^m_{g,n+1}$ by $0,1,\ldots,n$. Let $\Gamma\in\ST^{m-1}_{g,n}$. Choose a leg $l\in L(\Gamma)$. Suppose that it is marked by number $1\le i\le n$. Let us attach to the leg~$l$ a new vertex of genus~$0$ with two legs marked by numbers~$0$ and~$i$. Denote the resulting stable tree by $\Phi_l(\Gamma)\in\ST^m_{g,n+1}$. Similarly, if we choose an edge $e\in E(\Gamma)$, then we can break this edge and insert a genus~$0$ vertex with one leg marked by~$0$. Denote the resulting stable tree by $\Phi_e(\Gamma)\in\ST^m_{g,n+1}$. Using these operations, we can rewrite formula~\eqref{eq:divisibility of A-class,tmp1} in the following way:
$$
\tA^{g,m}(a_1,\ldots,a_n)=\sum_{\Gamma\in\ST^{m-1}_{g,n}}\left(\sum_{l\in L(\Gamma)}a(\Phi_l(\Gamma))+\sum_{e\in E(\Gamma)}a(\Phi_e(\Gamma))\right)\lambda_g\DR_{\Gamma}(a_1,\ldots,a_n).
$$
We see that it is sufficient to prove that for any stable tree $\Gamma\in\ST^{m-1}_{g,n}$ we have the identity
\begin{gather}\label{eq:divisibility of A-class,tmp2}
\sum_{l\in L(\Gamma)}a(\Phi_l(\Gamma))+\sum_{e\in E(\Gamma)}a(\Phi_e(\Gamma))=0.
\end{gather}

We prove~\eqref{eq:divisibility of A-class,tmp2} by induction on~$m$. It will be convenient for us to assume that the genus $g(v)$ of a vertex $v\in V(\Gamma)$ can be a rational number such that $2g(v)-2+n(v)>0$. So the total genus $g=\sum_{v\in V(\Gamma)}g(v)$ can also be rational. If $m=2$, then 
$$
\sum_{l\in L(\Gamma)}a(\Phi_l(\Gamma))=\sum_{i=1}^n\frac{-a_i}{2g-1+n}=0.
$$
Suppose $m\ge 3$. Choose a vertex $v\in V(\Gamma)$ such that $|H[v]\backslash L[v]|=1$. Let $h$ be the unique half-edge from the set $H[v]\backslash L[v]$. Denote
$$
h':=\iota(h),\quad v':=v(h'),\quad r:=r(v),\quad r':=r(v'),\quad R:=2g-2+n.
$$
Denote by $e$ the edge of $\Gamma$ corresponding the pair of half-edges $(h,h')$. Let us erase the vertex~$v$ together with all half-edges incident to it. Then the half-edge $h'$ becomes a leg. Let us denote it by $l'$ and mark by~$n+1$. Finally, let us increase the genus of the vertex~$v'$ by~$\frac{r}{2}$. As a result, we get a stable tree from $\ST^{m-2}_{g+\frac{r}{2},n-|L[v]|+1}$ that we denote by~$\Gamma'$. Note that the legs of~$\Gamma'$ are marked by the numbers $i(l), l\in L(\Gamma)\backslash L[v]$, and $n+1$. We want to apply the induction assumption to the tree~$\Gamma'$. Naturally, we assign to a leg $l\in L(\Gamma')$ the number~$a_{i(l)}$, if $l\ne l'$, and the number $a(h')=\sum_{\tl\in L[v]}a(\tl)$, if $l=l'$. It is not hard to see that
\begin{align*}
\sum_{l\in L[v]}a(\Phi_l(\Gamma))=&(-a(h'))\frac{r r'}{(R-r)(r+r')}a(\Phi_{l'}(\Gamma')),\\
a(\Phi_e(\Gamma))=&a(h')\frac{r' R}{(R-r)(r+r')}a(\Phi_{l'}(\Gamma')).
\end{align*}
It is also easy to see that for any leg $l\in L(\Gamma')$, $l\ne l'$, and for any edge $e'\in E(\Gamma')$ we have
\begin{gather*}
a(\Phi_{l}(\Gamma))=\frac{r'}{r+r'}a(h')a(\Phi_{l}(\Gamma')),\qquad a(\Phi_{e'}(\Gamma))=\frac{r'}{r+r'}a(h')a(\Phi_{e'}(\Gamma')).
\end{gather*}
Therefore, we obtain
\begin{multline*}
\sum_{l\in L(\Gamma)}a(\Phi_l(\Gamma))+\sum_{e'\in E(\Gamma)}a(\Phi_{e'}(\Gamma))=\\
=\frac{r'}{r+r'}a(h')\left(\sum_{l\in L(\Gamma')}a(\Phi_{l}(\Gamma'))+\sum_{e'\in E(\Gamma')}a(\Phi_{e'}(\Gamma'))\right)=0,
\end{multline*}
where the last equality follows from the induction assumption. The lemma is proved.
\end{proof}
The lemma allows to define a class $A^{g,m}(a_1,\ldots,a_n)$ by
$$
A^{g,m}(a_1,\ldots,a_n):=\frac{1}{\sum a_i}\tA^{g,m}(a_1,\ldots,a_n)\in R^{2g+m-2}(\oM_{g,n}).
$$
It is a polynomial in $a_1,\ldots,a_n$ homogeneous of degree~$2g+m-2$.

\begin{definition}\label{A-class}
For any $d_1,\ldots,d_n\ge 0$ such that $\delta := \sum_{i=1}^n d_i \ge 2g-1$ we define
$$
A^g_{d_1,\ldots,d_n}:=\Coef_{a_1^{d_1}\cdots a_n^{d_n}}A^{g,\delta-2g+2}(a_1,\ldots,a_n) \in R^{\delta}(\oM_{g,n}).
$$	
\end{definition}

If $\sum d_i=2g-1$, then the formula for $A^g_{d_1,\ldots,d_n}$ becomes particularly simple:
$$
A^g_{d_1,\ldots,d_n}=\Coef_{a_1^{d_1}\cdots a_n^{d_n}}\left(\frac{1}{\sum a_i}\lambda_g\DR_g\left(\widetilde{-\sum a_i},a_1,\ldots,a_n\right)\right).
$$

\subsection{Definition of the $B$-class and the main conjecture}\label{section:Bclasses}

Let $T$ be a stable rooted tree with at least $n$ legs, where we split the set of legs in two subsets:
\begin{itemize}
	\item[-] the legs $\sigma_1, \dotsc, \sigma_n$ corresponding to the markings,
	\item[-] some extra legs, whose set is denoted by $F(T)$, corresponding to additional marked points that we will eventually forget.
\end{itemize}
We will never call marking an element of $F(T)$ and let 
$$
H^{em}_+(T):=H_+(T)\backslash F(T).
$$
There is a natural level function $l \colon V(T) \to \NN^*$ such that the root is of level $1$ and if a vertex~$v$ is the mother of a vertex $v'$, then $l(v')=l(v)+1$. The total number of levels in $T$ will be denoted by $\deg(T)$ and called the degree of~$T$. It is also convenient to extend the level function to $H^{em}_+(T)$ by taking $l(h):=k$ if the half-edge $h$ is attached to a vertex of level $k$. We say that~$T$ is complete if the following conditions are satisfied:
\begin{itemize}
	\item[-] every vertex has at least one of its descendants with level $\deg(T)$,
	\item[-] all the markings are attached to the vertices of level $\deg(T)$,
	\item[-] each vertex of level $\deg(T)$ is attached to at least one marking,
	\item[-] there are no extra legs attached to the root,
	\item[-] for every vertex except the root there is at least one extra leg attached to it.
\end{itemize}

For a complete tree $T$ define a power function
$$
q\colon H^e_+(T) \to \NN
$$
by requiring that for a half-edge $h\in H_+^e(T)$ there is exactly $q(h)+1$ extra legs attached to the vertex~$v$ which is the direct descendant of~$h$. We say that $T$ is stable if
\begin{itemize}
	\item[-] for every $1\le k\le\deg(T)$, there is at least one vertex $v \in V(T)$ of level $k$ such that $v$ remains stable once we forget all the extra legs,
	\item[-] every vertex of genus $0$ with exactly one half-edge $h \in H^{em}_+(T)$ attached to it has exactly $q(h)+1$ extra legs attached to it,
	\item[-] every vertex of genus $0$ with exactly two half-edges $h_1, h_2 \in H^{em}_+(T)$ attached to it has exactly $q(h_1)+q(h_2)$ extra legs attached to it.
\end{itemize}
We say that a stable complete tree $T$ is admissible if for every $1\leq k < \deg(T)$ we have the condition
\begin{gather}\label{admissibility condition}
\sum_{\substack{h \in H_+^e(T) \\ l(h)=k}} q(h) \leq 2 \sum_{\substack{v \in V(T) \\ l(v) \leq k}} g(v) -2.
\end{gather}
We denote by $\Omega^{B,g}_{d_1,\dotsc,d_n}$ the set of pairs $(T,q)$, where $T$ is an admissible stable complete tree with total genus $g$ and $n$ markings, and $q\colon H_+^{em}(T)\to\NN$ is the extension of the power function from above defined by $q(\sigma_i):=d_i$. We denote by
$$
[T,q] :=\xi_{T*}\left(\prod_{h \in H_+^{em}(T)} \psi_h^{q(h)}\right) \in R^*(\oM_{g,n+\#F(T)})
$$
and by 
$$
\ee \colon \oM_{g,n+\#F(T)} \to \oM_{g,n}
$$
the map forgetting all the extra legs.

\begin{definition}\label{B-class}
For any $d_1,\ldots,d_n\ge 0$ with $\delta:=d_1+\dotsb+d_n$, we define
\begin{equation}\label{eq:Bdef}
B^g_{d_1,\dotsc,d_n} = \sum_{(T,q) \in \Omega^{B,g}_{d_1,\dotsc,d_n}} (-1)^{\deg(T)-1} \ee_* [T,q] \in R^\delta(\oM_{g,n}).
\end{equation}
\end{definition}

\begin{conjecture}\label{conjecture}
Suppose $g\ge 0$, $n\ge 1$ and $2g-2+n>0$. Then for any $d_1,\ldots,d_n\ge 0$, such that $\sum d_i\ge 2g-1$, we have 
\begin{gather}\label{eq:main relations}
A^g_{d_1,\ldots,d_n}=B^g_{d_1,\ldots,d_n}.
\end{gather}
\end{conjecture}

\begin{remark}\label{Baltrem}

Let us show how to express the $B$-class in terms of basic tautological classes. Let~$T$ be a stable complete tree with $n$ markings. For a vertex $v\in V(T)$ denote by $F[v]$ the set of extra legs incident to $v$ and by $H^{em}_+[v]$ the set of half-edges $h\in H^{em}_+(T)$ incident to $v$. The vertex $v$ will be called strongly stable if it remains stable once we forget all the extra legs. Otherwise, we call it weakly stable. Clearly, the vertex $v$ is weakly stable if and only if $g(v)=0$ and $|H^{em}_+[v]|=1$. The set of all strongly stable vertices of $T$ will be denoted by $V^{ss}(T)$.

For a stable complete tree $T$ denote by $\st(T)$ the stable rooted tree obtained by forgetting all extra legs of $T$ and then contracting all unstable vertices. Clearly, we can identify $V(\st(T))=V^{ss}(T)$ and we also identify the set $H(\st(T))$ with the set of half-edges $h\in H(T)$ such that~$v(h)$ is strongly stable. 

Suppose $\pi\colon\oM_{g,n+m}\to\oM_{g,n}$ is the forgetful map that forgets the last $m$ marked points. Then for any numbers $c_1,\ldots,c_n\ge 0$ we have
$$
\pi_*(\psi_1^{c_1}\cdots\psi_n^{c_n})=\sum_{\substack{b_1,\ldots,b_n\ge 0\\b_i\le c_i\\\sum b_i+m=\sum c_i}}\frac{m!}{\prod(c_i-b_i)!}\prod_{i=1}^n\psi_i^{b_i}.
$$
Using this formula, it is easy to see that equation~\eqref{eq:Bdef} can be rewritten in the following way:
\begin{align*}
B^g_{d_1,\dotsc,d_n}=&\sum_{(T,q) \in \Omega^{B,g}_{d_1,\dotsc,d_n}} (-1)^{\deg(T)-1}\times\\
&\times \xi_{\st(T)*}\prod_{v\in V(\st(T))}\sum_{\substack{p\colon H_+[v]\to\mbZ_{\ge 0}\\p(h)\le q(h)\\\sum p(h)+|F[v]|=\sum q(h)}}\frac{|F[v]|!}{\prod(q(h)-p(h))!}\prod_{h\in H_+[v]}\psi_h^{p(h)}.
\end{align*}
 
\end{remark}

Let us immediately present some examples of relations~\eqref{eq:main relations}. Consider genus $0$. Then it is easy to see that for any $d_1,\ldots,d_n\ge 0$ we have
$$
B^0_{d_1,\ldots,d_n}=\psi_1^{d_1}\cdots\psi_n^{d_n}.
$$
On the other hand, let us compute, for example, $A^0_{1,0,0,0}$. We have
\begin{align*}
\tA^{0,3}(a_1,a_2,a_3,a_4)=&\pi_*\left(\sum_{\substack{\{i,j,k,l\}=\{1,2,3,4\}\\i<j}}\frac{(a_i+a_j)(a_i+a_j+a_k)}{6}
\tikz[baseline=-1mm]{\draw (A)--(B)--(C);\legm{A}{180}{0};\legm{A}{-45}{l};\legm{B}{-45}{k};\legm{C}{45}{i};\legm{C}{-45}{j};\gg{0}{A};\gg{0}{B};\gg{0}{C};}
\right.\\
&\hspace{0.7cm}+\left.\sum_{\substack{\{i,j,k,l\}=\{1,2,3,4\}\\i<j,\,k<l,\,i<k}}\frac{(a_i+a_j)(a_k+a_l)}{3}
\tikz[baseline=-1mm]{\draw (E)--(G)--(F);\legm{G}{180}{0};\legm{E}{18}{i};\legm{E}{-18}{j};\legm{F}{18}{k};\legm{F}{-18}{l};\gg{0}{E};\gg{0}{F};\gg{0}{G};}
\right)=\\
=&\frac{a^2}{3}\,\tikz[baseline=-1mm]{\draw (A)--(B);\leg{A}{150};\leg{A}{-150};\leg{B}{30};\leg{B}{-30};\gg{0}{A};\gg{0}{B};}
=a^2\,\tikz[baseline=-1mm]{\draw (A)--(B);\legm{A}{150}{1};\legm{A}{-150}{2};\legm{B}{30}{3};\legm{B}{-30}{4};\gg{0}{A};\gg{0}{B};},
\end{align*}
where $a:=\sum a_i$. This gives $A^0_{1,0,0,0}=\tikz[baseline=-1mm]{\draw (A)--(B);\legm{A}{150}{1};\legm{A}{-150}{2};\legm{B}{30}{3};\legm{B}{-30}{4};\gg{0}{A};\gg{0}{B};}$ that is indeed equal to $\psi_1=B^0_{1,0,0,0}$.

Consider genus~$1$ and the case $n=1$, $d_1=1$. Then we have
\begin{gather*}
A^1_1=\Coef_a\left(\frac{1}{a}\lambda_1\pi_*\DR_1(-a,a)\right)=\lambda_1=\psi_1,\qquad B^1_1=\psi_1.
\end{gather*}

Let us give one more example with $g=2$, $n=1$ and $d_1=3$. We compute
\begin{align}
&A^2_3=\Coef_{a^3}\left(\frac{1}{a}\lambda_2\DR_2(\widetilde{-a},a)\right),\notag\\
&B^2_3=\psi_1^3-\tikz[baseline=-1mm]{\draw (A)--(B);\leg{B}{0};\gg{1}{A};\gg{1}{B};\lab{B}{25}{4mm}{\psi^2};}.\label{eq:formula for B23}
\end{align}
Now the relation $A^2_3=B^2_3$ is not so trivial, and we will prove it in Section~\ref{section:genus 2}.

Below we will check that the conjecture is true in genus~$0$ and~$1$ for arbitrary $d_i$'s, and also in genus~$2$ in the case $\sum d_i\le 4$. 


\section{DR/DZ equivalence conjecture and the new tautological relations}\label{section:DR/DZ}

In this section we explain the relation between the above Conjecture \ref{conjecture} and the strong double ramification/Dubrovin-Zhang hierarchies equivalence conjecture from \cite{BDGR16a}. After recalling the main notions, we prove in particular how the first implies the second.

\subsection{Dubrovin-Zhang hierarchy}

Consider an arbitrary cohomological field theory (CohFT, see \cite{KM94}) $c_{g,n}\colon V^{\otimes n}\to H^{\even}(\oM_{g,n},\mbC)$, with $V$ its $N$-dimensional vector space, $e_1,\ldots,e_N$ a basis of $V$, $e_1$ the unit and $\eta$ its symmetric non-degenerate bilinear form. Let $F=F(t^*_*,\eps)$ denote its potential, i.e.~the generating series of its intersection numbers with monomials in the psi classes:
\begin{align*}
\<\tau_{d_1}(e_{\alpha_1})\cdots\tau_{d_n}(e_{\alpha_n})\>_g:=&\int_{\oM_{g,n}}c_{g,n}(\otimes_{i=1}^n e_{\alpha_i})\prod_{i=1}^n\psi_i^{d_i},\,\, 2g-2+n>0,\,\, 1\le\alpha_i\le N,\\
F(t^*_*,\eps):=&\sum_{g\ge 0}\eps^{2g}F_g(t^*_*),\quad\text{where}\\
F_g(t^*_*):=&\sum_{\substack{n\ge 0\\2g-2+n>0}}\frac{1}{n!}\sum_{d_1,\ldots,d_n\ge 0}\<\prod_{i=1}^n\tau_{d_i}(e_{\alpha_i})\>_g\prod_{i=1}^nt^{\alpha_i}_{d_i}.
\end{align*}
In case the cohomological field theory is semisimple, in \cite{DZ05,BPS12} the authors associate to it an integrable hierarchy of Hamiltonian PDEs.  Let $\hcA_w^{[d]}$ be the degree $d$ part of $\mbC[[w^*_*,\eps]]$, where~$w^\alpha_k$, $\alpha=1,\ldots,N$, $k=0,1,2\ldots$, are formal variables of degree $\deg w^\alpha_k = k$ and $\deg \eps = -1$. Let $\hLambda_w^{[d]}$ be its quotient with respect to constants and the image of the operator $\d_x=\sum_{k\geq 0} w^\alpha_{k+1} \frac{\d}{\d w^\alpha_{k}}$ (we perform sums over repeated Greek indices here and in what follows) and, if $f \in \hcA_w^{[d]}$, let $\overline{f}$ denote its equivalence class in $\hLambda_w^{[d]}$. The Dubrovin-Zhang (DZ) hierarchy consists in Hamiltonian densities
$$
h^\DZ_{\alpha,p} \in \hcA^{[0]}_w,\quad 1\le\alpha\le N,\quad p\ge -1,
$$
with $h^\DZ_{\alpha,-1} = \eta_{\alpha \mu} w^\mu$, and a Hamiltonian operator
$$(K^\DZ)^{\mu\nu} = \sum_{j\geq 0} (K^\DZ)_j^{\mu\nu} \partial_x^j,\qquad (K^\DZ)_j^{\mu\nu}\in \hcA_w^{[-j+1]},$$
such that
$$
\left\{\oh^\DZ_{\alpha,p},\oh^\DZ_{\beta,q}\right\}_{K^\DZ} :=\int\frac{\delta \oh^\DZ_{\alpha,p}}{\delta w^\mu} (K^\DZ)^{\mu\nu}\left(\frac{\delta \oh^\DZ_{\beta,q}}{\delta w^\nu}\right)dx = 0,\,\, 1\le\alpha,\beta\le N, \,\, p,q\ge -1,
$$
where we have used the variational derivative $\frac{\delta \overline{f}}{\delta w^\mu}:=\sum_{i\geq 0} (-\partial_x)^i \frac{\partial f}{\partial w^\mu_i}$.
This guarantees that solutions $w^\alpha_k (x,t^*_*,\eps) = \d_x^k w^\alpha(x,t^*_*,\eps) \in \mbC[[x,t^*_*,\eps]]$ exist for the system of Hamiltonian PDEs
$$
\frac{\d}{\d t^\beta_q}w^\alpha = (K^\DZ)^{\alpha\nu}\left(\frac{\delta \oh^\DZ_{\beta,q}}{\delta w^\nu}\right), \quad 1\le\alpha,\beta\le N, \quad q\ge 0.
$$
Notice how this Hamiltonian system in fact only depends on the Hamiltonian functionals $\oh^\DZ_{\alpha,p} \in \hLambda^{[0]}_w$ and not on the Hamiltonian densities $h^\DZ_{\alpha,p} \in \hcA^{[0]}_w$. Nonetheless, Dubrovin and Zhang's construction of specific Hamiltonian densities $h^\DZ_{\alpha,p} \in \hcA^{[0]}_w$ is important because it is a tau-structure (see \cite{BDGR16a} for details), which implies in particular that, for any solution $w^\alpha(x,t^*_*,\eps) \in \mbC[[x,t^*_*,\eps]]$, there exists a formal series, called (the logarithm of) the tau-function, $\cF(t^*_*,\eps) \in \mbC[[t^*_*,\eps]]$ such that
$$
\left.\frac{\d h^\DZ_{\alpha,p-1}}{\d t^\beta_q}\right|_{w^*_*=w^*_*(x,t^*_*,\eps)|_{x=0}} = \frac{\d^3 \cF}{\d t^1_0 \d t^\alpha_p \d t^\beta_q}, \quad 1\le\alpha,\beta\le N,\quad p,q\ge 0.
$$
An important property of the DZ hierarchy is that the so-called topological solution, i.e.~the solution with the initial condition $(w^\top)^\alpha(x,t^*_*,\eps)|_{t^*_*=0} = \delta^{\alpha,1} x$, has the potential $F(t^*_*,\eps)$ of the underlying semisimple CohFT as the logarithm of its tau-function,
$$
\left.\frac{\d h^\DZ_{\alpha,p-1}}{\d t^\beta_q}\right|_{w^*_*=(w^\top)^*_*(x,t^*_*,\eps)|_{x=0}} = \frac{\d^3 F}{\d t^1_0\d t^\alpha_p \d t^\beta_q}, \quad 1\le\alpha,\beta\le N,\quad p,q\ge 0.
$$

\subsection{Double ramification hierarchy}
The double ramification (DR) hierarchy, see \cite{Bur15,BDGR16a}, is another tau-symmetric hierarchy of Hamiltonian PDEs associated to an arbitrary CohFT, this time even without the requirement of semisimplicity. This time it is the Hamiltonians that are constructed as generating series of certain intersection numbers of the CohFT with psi classes, the $\lambda_g$ class and the double ramification cycle. Written in formal variables $\tu^\alpha_k$, $\alpha=1,\ldots,N$, $k=0,1,2\ldots$, it consists of differential polynomials
$$
h^\DR_{\alpha,p} \in \hcA^{[0]}_\tu, \quad 1\le\alpha\le N,\quad p\ge -1,
$$
with $h^\DR_{\alpha,-1} = \eta_{\alpha \mu} \tu^\mu$, and a Hamiltonian operator
$$(K^\DR)^{\mu\nu} = \sum_{j\geq 0} (K^\DR)_j^{\mu\nu} \partial_x^j,\qquad (K^\DR)_j^{\mu\nu}\in \hcA_\tu^{[-j+1]},$$
such that
$$
\left\{\oh^\DR_{\alpha,p},\oh^\DR_{\beta,q}\right\}_{K^\DR} := \int \frac{\delta \oh^\DR_{\alpha,p}}{\delta \tu^\mu} (K^\DR)^{\mu\nu}\left(\frac{\delta \oh^\DR_{\beta,q}}{\delta \tu^\nu}\right)dx = 0,\,\, 1\le\alpha,\beta\le N, \,\, p,q\ge -1.
$$
Like for the DZ hierarchy, the DR Hamiltonian densities $h^\DR_{\alpha,p} \in \hcA^{[0]}_\tu$ form a tau-structure and we can define the DR potential as (the logarithm of) the tau-function of the topological solution, $(\tu^\top)^\alpha \in \mbC[[x,t^*_*,\eps]]$ with $(\tu^\top)^\alpha(x,t^*_*,\eps)|_{t^*_*=0} = \delta^{\alpha,1} x$, i.e.~$F^\DR(t^*_*,\eps) \in \mbC[[t^*_*,\eps]]$ satisfies
$$
\left.\frac{\d h^\DR_{\alpha,p-1}}{\d t^\beta_q}\right|_{\tu^*_*=(\tu^\top)^*_*(x,t^*_*,\eps)|_{x=0}} = \frac{\d^3 F^\DR}{\d t^1_0 \d t^\alpha_p \d t^\beta_q}, \quad 1\le\alpha,\beta\le N,\quad p,q\ge 0.
$$
Clearly, this equation doesn't determine the function $F^\DR$ uniquely, but we can additionaly require that~$F^\DR$ should satisfy the string and the dilaton equations. Then this fixes the potential~$F^\DR$ completely. We define the DR correlators as the coefficients of the power series~$F^\DR(t^*_*,\eps)$,
\begin{align*}
F^\DR(t^*_*,\eps)=:&\sum_{g\ge 0}\eps^{2g}F^\DR_g(t^*_*),\quad\text{where}\\
F^\DR_g(t^*_*)=:&\sum_{\substack{n\ge 0\\2g-2+n>0}}\frac{1}{n!}\sum_{d_1,\ldots,d_n\ge 0}\<\prod_{i=1}^n\tau_{d_i}(e_{\alpha_i})\>^\DR_g\prod_{i=1}^nt^{\alpha_i}_{d_i}.
\end{align*}

\subsection{Strong DR/DZ equivalence conjecture}
In the effort of understanding the relation between the DR and DZ hierarchies associated to the same semisimple CohFT, in \cite{BDGR16a} it was conjectured that a change of coordinates $w^\alpha \mapsto \tu^\alpha$ existed, transforming one hierarchy into the other and preserving the given tau-structures. A natural family of such changes of coordinates (called normal Miura transformations) has the form
\begin{equation}\label{eq:normal}
\tu^\alpha(w)=w^\alpha+\eta^{\alpha\mu}\d_x\left\{\cP,\oh^\DZ_{\mu,0}\right\}_{K^\DZ},
\end{equation}
where $\cP\in\hcA^{[-2]}_w$ is an arbitrary differential polynomial and 
$$
\left\{\cP,\oh^\DZ_{\mu,0}\right\}_{K^\DZ} = \sum_{k\geq 0} \frac{\d \cP}{\d w^\mu_k} \d_x^k\left((K^\DZ)^{\mu \nu} \left(\frac{\delta \oh^\DZ_{\mu,0}}{\delta w^\nu}\right)\right).
$$
The effect of such a transformation on the topological tau-function of the DZ hierarchy is the following:
$$
F \mapsto F+\left.\cP(w^*_*,\eps)\right|_{w^*_* = (w^\top)^*_*(x,t^*_*,\eps)|_{x=0}}.
$$

In \cite{BDGR16a} the following results were proved.
\begin{proposition}[\cite{BDGR16a}]\label{proposition:vanishing}
Let $g,m\ge 0$ such that $2g-2+m>0$. Then
$$\<\tau_{d_1}(e_{\alpha_1})\cdots\tau_{d_m}(e_{\alpha_m})\>^{\DR}_g=0,\quad\text{if}\quad \sum_{i=1}^m d_i< 2g-1.$$
\end{proposition}
\begin{proposition}[\cite{BDGR16a}]\label{proposition:Fred}
There exists a unique differential polynomial $\cP\in\hcA^{[-2]}_w$ such that for the power series $F^\red\in\mbC[[t^*_*,\eps]]$, defined by 
\begin{gather}\label{eq:definition of Fred}
F^\red:=F+\left.\cP(w^*_*,\eps)\right|_{w^*_* = (w^\top)^*_*(x,t^*_*,\eps)|_{x=0}},
\end{gather}
the correlators
$$
\<\tau_{d_1}(e_{\alpha_1})\cdots\tau_{d_n}(e_{\alpha_n})\>^\red_g :=\Coef_{\eps^{2g}}\left.\frac{\d^n F^\red}{\d t^{\alpha_1}_{d_1}\cdots\d t^{\alpha_n}_{d_n}}\right|_{t^*_*=0}
$$
satisfy the following vanishing property:
\begin{gather}\label{eq:property of Fred}
\<\tau_{d_1}(e_{\alpha_1})\cdots\tau_{d_n}(e_{\alpha_n})\>^\red_g=0,\quad\text{if}\quad \sum_{i=1}^n d_i< 2g-1.
\end{gather}
\end{proposition}

In light of these two results the following conjecture was formulated in \cite{BDGR16a}.
\begin{conjecture}[Strong DR/DZ equivalence]\label{conjecture:strong}
Consider a semisimple cohomological field theory and the associated DZ and DR hierarchies. Then the normal Miura transformation (\ref{eq:normal}) defined by the differential polynomial~$\cP$ of Proposition \ref{proposition:Fred} maps the DZ hierarchy to the DR hierarchy respecting their tau-structures.
\end{conjecture}

As proved in \cite{BDGR16a} this conjecture is equivalent to saying that $F^\red = F^\DR$. This last form of the conjecture can be generalized to arbitrary CohFTs, forgetting about the DZ hierarchy and concentrating on the reduced and DR potentials.
\begin{conjecture}[Generalized strong DR/DZ equivalence]\label{conjecture:generalization of strong}
For an arbitrary cohomological field theory we have $F^\DR=F^\red$.
\end{conjecture}

\subsection{From intersection numbers to cohomology classes} The following result makes the relation between Conjecture \ref{conjecture} and Conjecture \ref{conjecture:generalization of strong} explicit, showing in particular how the first implies the second.

\begin{proposition}
Consider an arbitrary cohomological field theory $c_{g,n}\colon V^{\otimes n}\to H^{\even}(\oM_{g,n},\mbC)$. Then for any $g,n\ge 0$, $2g-2+n>0$, and numbers $d_1,\ldots,d_n\ge 0$ such that $\sum d_i\ge 2g-1$ we have
\begin{align}
&\<\tau_{d_1}(e_{\alpha_1})\cdots\tau_{d_n}(e_{\alpha_n})\>^\DR_g=\int_{\oM_{g,n}}A^g_{d_1,\ldots,d_n}c_{g,n}(e_{\alpha_1}\otimes\cdots\otimes e_{\alpha_n}),\label{eq:DR/DZ and relations,eq1}\\
&\<\tau_{d_1}(e_{\alpha_1})\cdots\tau_{d_n}(e_{\alpha_n})\>^\red_g=\int_{\oM_{g,n}}B^g_{d_1,\ldots,d_n}c_{g,n}(e_{\alpha_1}\otimes\cdots\otimes e_{\alpha_n}).\label{eq:DR/DZ and relations,eq2}
\end{align}
\end{proposition}
\begin{proof}
In \cite{BDGR16b} the authors proved that for any $d\ge 2g-1$ we have
\begin{align*}
&\sum_{\substack{d_1,\ldots,d_n\ge 0\\\sum d_i=d}}\<\tau_{d_1}(e_{\alpha_1})\cdots\tau_{d_n}(e_{\alpha_n})\>^{\DR}_ga_1^{d_1}\cdots a_n^{d_n}=\\
=&\frac{1}{\sum a_i}\sum_{\Gamma\in\ST^{d-2g+2}_{g,n+1}}a(\Gamma)\int_{\oM_{g,n+1}}\DR_\Gamma\left(-\sum a_i,a_1,\ldots,a_n\right)\lambda_g c_{g,n+1}\left(e_1\otimes\otimes_{i=1}^n e_{\alpha_i}\right)=\\
=&\int_{\oM_{g,n}}A^{g,d-2g+2}(a_1,\ldots,a_n)c_{g,n}\left(\otimes_{i=1}^n e_{\alpha_i}\right).
\end{align*}
Equation~\eqref{eq:DR/DZ and relations,eq1} is proved.

Let us prove equation~\eqref{eq:DR/DZ and relations,eq2}. The reduced potential~$F^\red$ can be constructed in the following way. Let us recursively construct a sequence of power series
$$
F^{(0,-2)}=F,F^{(1,0)},F^{(2,0)},F^{(2,1)},F^{(2,2)},\ldots,F^{(j,0)},F^{(j,1)},\ldots,F^{(j,2j-2)},\ldots\in\mbC[[t^*_*,\eps]].
$$
Suppose that a series $F^{(j,k)}$ is already defined. Introduce correlators $\<\tau_{d_1}(e_{\alpha_1})\cdots\tau_{d_n}(e_{\alpha_n})\>^{(j,k)}_g$ by 
$$
\<\tau_{d_1}(e_{\alpha_1})\cdots\tau_{d_n}(e_{\alpha_n})\>^{(j,k)}_g:=\left.\Coef_{\eps^{2g}}\frac{\d^n F^{(j,k)}}{\d t^{\alpha_1}_{d_1}\cdots\d t^{\alpha_n}_{d_n}}\right|_{t^*_*=0}.
$$
If $k<2j-2$, then we define the series $F^{(j,k+1)}$ by 
\begin{gather}\label{eq:definition of Fjk+1}
F^{(j,k+1)}:=F^{(j,k)}-\sum_{n\ge 0}\sum_{\substack{d_1,\ldots,d_n\ge 0\\\sum d_i=k+1}}\frac{\eps^{2j}}{n!}\<\prod\tau_{d_i}(e_{\alpha_i})\>^{(j,k)}_j\prod((w^\top)_{d_i}^{\alpha_i}-\delta^{\alpha_i,1}\delta_{d_i,1})|_{x=0}.
\end{gather}
If $k=2j-2$, then we define the series $F^{(j+1,0)}$ by an analogous formula
\begin{gather*}
F^{(j+1,0)}:=F^{(j,2j-2)}-\sum_{n\ge 0}\frac{\eps^{2j+2}}{n!}\<\prod_{i=1}^n\tau_0(e_{\alpha_i})\>^{(j,2j-2)}_{j+1}\prod(w^\top)^{\alpha_i}|_{x=0}.
\end{gather*}

Recall that
$$
(w^\top)^\alpha=\left.\eta^{\alpha\mu}\frac{\d^2 F}{\d t^\mu_0\d t^1_0}\right|_{t^1_0\mapsto t^1_0+x}.
$$
The string equation for the potential $F$,
$$
\frac{\d F}{\d t^1_0}=\sum_{n\ge 0}t^\alpha_{n+1}\frac{\d F}{\d t^\alpha_n}+\frac{1}{2}\eta_{\alpha\beta}t^\alpha_0t^\beta_0+\eps^2\<\tau_0(e_1)\>_1,
$$
implies that the function $(w^\top)^\alpha_n|_{x=0}$ has the form
$$
(w^\top)^\alpha_n|_{x=0}=\delta^{\alpha,1}\delta_{n,1}+t^\alpha_n+r^\alpha_n+O(\eps^2),
$$ 
where the power series $r^\alpha_n\in\mbC[[t^*_*]]$ doesn't contain monomials $t^{\beta_1}_{b_1}\cdots t^{\beta_m}_{b_m}$ with $\sum b_i\le n$. Clearly, if $g\le j$, then we have the property
$$
\<\tau_{d_1}(e_{\alpha_1})\cdots\tau_{d_n}(e_{\alpha_n})\>^{(j,k)}_g=0,\quad\text{if}\quad
\sum d_i\le
\begin{cases}
2g-2,&\text{if $g<j$},\\
k,&\text{if $g=j$}.
\end{cases}
$$
Define a series $F'$ by $F':=\lim_{j\to\infty}F^{(j,2j-2)}$. The series $F'$ has the form
$$
F'=F+\left.\cP'(w^\top,w^\top_x,\ldots,\eps)\right|_{x=0}
$$
for some non-homogeneous differential polynomial $\cP'=\sum_{i\le -2}\cP'_i$, $\cP'_i\in\hcA^{[i]}_w$. Moreover, we have the property
$$
\left.\Coef_{\eps^{2g}}\frac{\d^n F'}{\d t^{\alpha_1}_{d_1}\cdots\d t^{\alpha_n}_{d_n}}\right|_{t^*_*=0}=0,\quad\text{if}\quad\sum d_i\le 2g-2.
$$
One can notice that the recursive construction, described above, is slightly different from the recursive construction of the reduced potential~$F^\red$, presented in the proof of Proposition~7.2 in~\cite{BDGR16a}. However, using the uniqueness argument given there we can see that $F'=F^\red$ and that actually $\cP'\in\hcA^{[-2]}_w$.

For a stable complete tree $T$ and $1\le m\le\deg(T)$ let 
$$
g_m(T):=\sum_{\substack{v\in V(T)\\l(v)\le m}}g(v).
$$
Before we proceed, let us prove the following simple lemma.
\begin{lemma}\label{simple lemma}
Let $d_1,\ldots,d_n\ge 0$, $(T,q)\in\Omega^{B,g}_{d_1,\ldots,d_n}$ and $1\le m<\deg(T)$. Suppose that $g_{m+1}(T)=g_m(T)$ and $\ee_*[T,q]\ne 0$. Then $\sum_{\substack{h\in H^e_+(T)\\l(h)=m+1}}q(h)>\sum_{\substack{h\in H^e_+(T)\\l(h)=m}}q(h)$.
\end{lemma}
\begin{proof}
Consider a half-edge $h\in H^e_+(T)$ with $l(h)=m$ and let $v:=v(\iota(h))$. We have $g(v)=0$ and the map $\ee$ forgets all $q(h)+1$ extra legs incident to $v$. Therefore, if $v$ is strongly stable, then $\sum_{h'\in H^e_+[v]}q(h')>q(h)$. If $v$ is weakly stable, then $|H^e_+[v]|=1$ and $q(h')=q(h)$, where $h'\in H^e_+[v]$. Since at least one vertex of level $m+1$ is strongly stable, the lemma is true.
\end{proof}

A stable complete tree $T$ will be called $(j,k)$-admissible, if for any $1\le m<\deg(T)$ we have $g_m(T)\le j$ and
$$
\sum_{\substack{h\in H^e_+(T)\\l(h)=m}}q(h)\le
\begin{cases}
2g_m(T)-2,&\text{if $g_m(T)<j$},\\
k,&\text{if $g_m(T)=j$}.
\end{cases}
$$
Let $\Omega^{B,g,(j,k)}_{d_1,\ldots,d_n}:=\{(T,q)\in\Omega^{B,g}_{d_1,\ldots,d_n}|\text{$T$ is $(j,k)$-admissible}\}$. Define a class $B^{g,(j,k)}_{d_1,\ldots,d_n}$ by
$$
B^{g,(j,k)}_{d_1,\ldots,d_n}:=\sum_{(T,q)\in\Omega^{B,g,(j,k)}_{d_1,\ldots,d_n}}(-1)^{\deg(T)-1}e_*[T,q]\in R^{\sum d_i}(\oM_{g,n}).
$$ 
Clearly, $B^{g,(j,k)}_{d_1,\ldots,d_n}=B^g_{d_1,\ldots,d_n}$, if $j>g$. 

In order to prove equation~\eqref{eq:DR/DZ and relations,eq2}, it is sufficient to prove that for any pair $(j,k)$ from the sequence
\begin{gather}\label{sequence of pairs}
(0,-2),(1,0),(2,0),(2,1),(2,2),\ldots,(j,0),(j,1),\ldots,(j,2j-2),\ldots
\end{gather}
we have
\begin{gather}\label{eq:DR/DZ and relations,tmp2}
\<\tau_{d_1}(e_{\alpha_1})\cdots\tau_{d_n}(e_{\alpha_n})\>^{(j,k)}_g=\int_{\oM_{g,n}}B^{g,(j,k)}_{d_1,\ldots,d_n}c_{g,n}(e_{\alpha_1}\otimes\cdots\otimes e_{\alpha_n}),
\end{gather}
if $g>j$, or $g\le j$ and 
$$
\sum d_i>
\begin{cases}
2g-2,&\text{if $g<j$},\\
k,&\text{if $g=j$}.
\end{cases}
$$
We proceed by induction. Obviously, equation~\eqref{eq:DR/DZ and relations,tmp2} is true for $(j,k)=(0,-2)$. Suppose that equation~\eqref{eq:DR/DZ and relations,tmp2} is true for a pair $(j,k)$ from the sequence~\eqref{sequence of pairs}. Let us check it for the next pair. 

Suppose $k<2j-2$. For any $d_1,\ldots,d_n\ge 0$ we have $\Omega^{B,g,(j,k)}_{d_1,\ldots,d_n}\subset\Omega^{B,g,(j,k+1)}_{d_1,\ldots,d_n}$. Using the induction assumption and formula~\eqref{eq:definition of Fjk+1}, we see that it remains to check that
\begin{align}
&\sum_{g,n\ge 0}\frac{\eps^{2g}}{n!}\sum_{d_1,\ldots,d_n\ge 0}\sum_{(T,q)\in\Omega^{B,g,(j,k+1)}_{d_1,\ldots,d_n}\backslash\Omega^{B,g,(j,k)}_{d_1,\ldots,d_n}}(-1)^{\deg(T)}\times\label{eq:B-class and correlators,case1}\\
&\hspace{5cm}\times\left(\int_{\oM_{g,n}}\ee_*[T,q]c_{g,n}(\otimes_{i=1}^n e_{\alpha_i})\right)\prod t^{\alpha_i}_{d_i}=\notag\\
=&\sum_{n\ge 0}\sum_{\substack{d_1,\ldots,d_n\ge 0\\\sum d_i=k+1}}\frac{\eps^{2j}}{n!}\<\prod\tau_{d_i}(e_{\alpha_i})\>^{(j,k)}_j\left(\prod((w^\top)_{d_i}^{\alpha_i}-\delta^{\alpha_i,1}\delta_{d_i,1})|_{x=0}-\prod t^{\alpha_i}_{d_i}\right).\notag
\end{align}
Consider a pair $(T,q)\in\Omega^{B,g,(j,k+1)}_{d_1,\ldots,d_n}\backslash\Omega^{B,g,(j,k)}_{d_1,\ldots,d_n}$ such that $\ee_*[T,q]\ne 0$. Then there exists $1\le m<\deg(T)$ such that $g_m(T)=j$ and $\sum_{\substack{h\in H^e_+(T)\\l(h)=m}}q(h)=k+1$. By Lemma~\ref{simple lemma}, $m=\deg(T)-1$. Denote by $T'$ the stable rooted tree obtained by erasing all vertices in $T$ of level $m+1$ together with half-edges incident to them. Half-edges $h\in H^e_+(T)$ with $l(h)=m$ become marked legs of $T'$. Clearly,~$T'$ is a stable complete tree. By Lemma~\ref{simple lemma}, the tree~$T'$ is $(j,k)$-admissible. Using the induction assumption, we conclude that equation~\eqref{eq:B-class and correlators,case1} is true. This completes the induction step in the case $k<2j-2$. The case $k=2j-2$ is analagous. The proposition is proved.
\end{proof}


\section{Further structure of the relations}\label{section:further structure}

In this section we discuss the structure of the conjectural relations~\eqref{eq:main relations} in more details. In Section~\ref{subsection:formulas with DR cycles} we recall the formulas for the intersections of the double ramification cycle with a $\psi$-class and with a boundary divisor on $\oM_{g,n}$. In Section~\ref{subsection:1-point} we show that for a fixed $g\ge 1$ all relations $A^g_d=B^g_d$, $d\ge 2g-1$, follow from the relation $A^g_{2g-1}=B^g_{2g-1}$. In Section~\ref{subsection:string equation} we prove that the $A$- and the $B$-class behave in the same way upon the pullback along the forgetful map. We then use this result in Section~\ref{subsection:reduction} in order to show that Conjecture~\ref{conjecture} is true if and only if it is true when all $d_i$'s are positive. In Section~\ref{subsection:dilaton equation} we prove that the classes $A^g_{d_1,\ldots,d_n,1}$ and~$B^g_{d_1,\ldots,d_n,1}$ behave in the same way upon the pushforward along the map forgetting the last marked point. Using this result, in Section~\ref{subsection:Mgn} we show that Conjecture~\ref{conjecture} is valid on $\cM_{g,n}$. In Section~\ref{subsection:lambdag} we show that the conjectural relations~\eqref{eq:main relations} give a new formula for the class $\lambda_g\in R^g(\oM_g)$ and check the resulting formula for $g\le 3$. 

\subsection{Formulas with the double ramification cycles}\label{subsection:formulas with DR cycles}

First of all, let us recall the formula from~\cite{BSSZ15} for the product of the double ramification cycle with a $\psi$-class. Denote by 
$$
\gl_k\colon\oM_{g_1,n_1+k}\times\oM_{g_2,n_2+k}\to\oM_{g_1+g_2+k-1,n_1+n_2}
$$
the gluing map that corresponds to gluing a curve from~$\oM_{g_1,n_1+k}$ to a curve from~$\oM_{g_2,n_2+k}$ along the last $k$ marked points on the first curve and the last $k$ marked points on the second curve. Introduce the notation
\begin{align*}
&\DR_{g_1}(a_1,\ldots,a_n)\boxtimes_k\DR_{g_2}(b_1,\ldots,b_m):=\\
=&\gl_{k*}\left(\DR_{g_1}(a_1,\ldots,a_n)\times\DR_{g_2}(b_1,\ldots,b_m)\right)\in R^{g_1+g_2+k}(\oM_{g_1+g_2+k-1,n+m-2k}).
\end{align*}
Let $a_1,\ldots,a_n$ be a list of integers with vanishing sum. For a subset $I=\{i_1,\ldots,i_k\}\subset\{1,\ldots,n\}$, $i_1<i_2<\ldots<i_k$, let us denote by~$A_I$ the list $a_{i_1},\ldots,a_{i_k}$ and by~$a_I$ the sum $\sum_{i\in I}a_i$. Assume that $a_s \ne 0$ for some $1\le s\le n$. Then we have~\cite[Theorem 4]{BSSZ15}
\begin{align}
a_s\psi_s \DR_g(a_1, \dots, a_n)=&\sum_{\substack{I\sqcup J=\{1,\ldots,n\}\\a_I>0}}\sum_{p\ge 1}\sum_{\substack{g_1,g_2\ge 0\\g_1+g_2+p-1=g}}\sum_{\substack{k_1,\ldots,k_p\ge 1\\\sum k_j=a_I}}\frac{\rho}{r}\frac{\prod_{i=1}^p k_i}{p!}\times\label{eq:DR times psi}\\
&\times\DR_{g_1}(A_I,-k_1,\ldots,-k_p)\boxtimes_p\DR_{g_2}(A_J,k_1,\ldots,k_p),\notag
\end{align}
where $r:=2g-2+n$ and
$$
\rho:=
\begin{cases} 
2g_2-2+|J|+p,    &\text{if $s\in I$};\\
-(2g_1-2+|I|+p), &\text{if $s\in J$}.
\end{cases}
$$

Let us also recall the formula for the intersection of the double ramification cycle with a boundary divisor on~$\oM_{g,n}$. For $0\le h\le g$ and a subset $I\subset\{1,\ldots,n\}$ we have~\cite{BSSZ15}
$$
\delta_h^I\cdot\DR_g(a_1,\ldots,a_n)=\DR_h\left(A_I,-a_I\right)\boxtimes_1\DR_{g-h}\left(A_{I^c},a_I\right),
$$
where $I^c:=\{1,2,\ldots,n\}\backslash I$.

\subsection{One-point case}\label{subsection:1-point}

\begin{lemma}\label{lemma:A-class for n=1}
Let $g\ge 1$. Then for any $k\ge 0$ we have $A^g_{2g-1+k}=\psi_1^k A^g_{2g-1}$.
\end{lemma}
\begin{proof}
Let $\pi\colon\oM_{g,2}\to\oM_{g,1}$ be the forgetful map that forgets the second marked point. We compute
\begin{align*}
a\psi_1 A^{g,1}(a)=&\psi_1\lambda_g\DR_g(a,\widetilde{-a})=\pi_*\left(\pi^*\psi_1\cdot\lambda_g\DR_g(a,-a)\right)=\\
=&\pi_*\left(\psi_1\cdot\lambda_g\DR_g(a,-a)\right)-\pi_*\left(\delta_0^{\{1,2\}}\cdot\lambda_g\DR_g(a,-a)\right)=\\
=&\sum_{\substack{g_1,g_2\ge 1\\g_1+g_2=g}}\frac{g_2}{g}\pi_*\left(\lambda_g\DR_{g_1}(a,-a)\boxtimes_1\DR_{g_2}(-a,a)\right)\\
&-\pi_*\left(\lambda_g\DR_g(0)\boxtimes_1\DR_0(a,-a,0)\right)=\\
=&A^{g,2}(a),
\end{align*}
where we used that $\lambda_g\DR_g(0)=(-1)^g\lambda_g^2=0$. If $k=1$, then the lemma is proved. If $k\ge 2$, then we write the equation $(a_1\psi_1)^k A^{g,1}(a)=(a_1\psi_1)^{k-1}A^{g,2}(a)$ and apply formula~\eqref{eq:DR times psi} to the right-hand side of it~$k-1$ times. The lemma is proved. 
\end{proof}

On the other hand, it is not hard to get an explicit expression for the class $B^g_d$. Let $g_1,g_2,\ldots,g_k\ge 1$ and $d_1,\ldots,d_k\ge 0$. Introduce a class $C_{d_1,\ldots,d_k}^{g_1,\ldots,g_k}\in R^{\sum d_i+k-1}(\oM_{\sum g_i,1})$ by 
$$
C_{d_1,\ldots,d_k}^{g_1,\ldots,g_k}:=\tikz[baseline=-1mm]{\draw (A)--(1,0);\leg{1,0}{0};\leg{2.8,0}{180};\leg{2.8,0}{0};\gg{g_1}{A};\gg{g_2}{1,0};\gg{g_k}{2.8,0};\lab{A}{23}{4.8mm}{\psi^{d_1}};\lab{1,0}{23}{4.8mm}{\psi^{d_2}};\lab{1.9,0}{-90}{0.06}{...};\lab{2.8,0}{23}{4.8mm}{\psi^{d_k}};}.
$$
Then it is easy to see that for $g\ge 1$ and $d\ge 2g-1$ we have
$$
B^g_d=\sum_{k=1}^g\sum_{\substack{g_1,\ldots,g_k\ge 1\\\sum g_i=g}}\sum_{d_1,\ldots,d_k}(-1)^{k-1}C^{g_1,\ldots,g_k}_{d_1,\ldots,d_k},
$$
where the last sum is taken over all non-negative integers $d_1,\ldots,d_k$ satisfying
\begin{align*}
&\sum_{i=1}^l d_i+l-1\le 2\sum_{i=1}^l g_i-2,\quad\text{if $1\le l\le g-1$},\\
&\sum_{i=1}^k d_i+k-1=d.
\end{align*}
We see that $B^g_d=\psi_1^{d-2g+1}B^g_{2g-1}$. Thus, for $n=1$ Conjecture~\ref{conjecture} is equivalent to the sequence of relations
$$
A^g_{2g-1}=B^g_{2g-1},\quad g\ge 1.
$$

\subsection{String equation}\label{subsection:string equation}

In this section we prove that the $A$- and the $B$-class behave in the same way upon the pullback along the forgetful map $\pi\colon\oM_{g,n+1}\to\oM_{g,n}$.

\begin{proposition}\label{proposition:pullback of A-class}
Denote by $\pi\colon\oM_{g,n+1}\to\oM_{g,n}$ the forgetful map that forgets the last marked point. Then we have
\begin{gather}\label{eq:pullback of A-class}
A^g_{d_1,\ldots,d_n,0}=
\begin{cases}
\pi^* A^g_{d_1,\ldots,d_n},&\text{if $\sum d_i=2g-1$},\\
\pi^* A^g_{d_1,\ldots,d_n}+\sum_{\substack{1\le i\le n\\d_i\ge 1}}\delta_0^{\{i,n+1\}}\pi^* A^g_{d_1,\ldots,d_i-1,\ldots,d_n},&\text{if $\sum d_i\ge 2g$}.
\end{cases}
\end{gather}
\end{proposition}
\begin{proof}
Let $m:=\sum d_i-2g+2$. The proposition is equivalent to the equation
\begin{multline}\label{eq:pullback of tA}
\tA^{g,m}(a_1,\ldots,a_n,0)=\\
=
\begin{cases}
\pi^*\tA^{g,m}(a_1,\ldots,a_n),&\text{if $m=1$},\\
\pi^*\tA^{g,m}(a_1,\ldots,a_n)+\sum_{i=1}^n a_i\delta_0^{\{i,n+1\}}\pi^*\tA^{g,m-1}(a_1,\ldots,a_n),&\text{if $m\ge 2$},
\end{cases}
\end{multline}
where $a_1,\ldots,a_n$ are arbitrary integers. Let $a_0:=-\sum_{i=1}^n a_i$. Introduce a class $\hA^{g,m}(a_0,a_1,\ldots,a_n)$ by
$$
\hA^{g,m}(a_0,a_1,\ldots,a_n):=\sum_{\Gamma\in\ST^m_{g,n+1}}a(\Gamma)\lambda_g\DR_\Gamma\left(a_0,a_1,\ldots,a_n\right).
$$
Formula~\eqref{eq:pullback of tA} follows from the equation
\begin{multline}\label{eq:string for A-class,reformulation}
\hA^{g,m}(a_0,\ldots,a_n,0)=\\
=
\begin{cases}
\pi^*\hA^{g,m}(a_0,\ldots,a_n),&\text{if $m=1$},\\
\pi^*\hA^{g,m}(a_0,\ldots,a_n)+\sum_{i=1}^na_i\delta^{\{i,n+1\}}_0\pi^*\hA^{g,m-1}(a_0,\ldots,a_n),&\text{if $m\ge 2$},
\end{cases}
\end{multline}
where the map $\pi\colon\oM_{g,n+2}\to\oM_{g,n+1}$ forgets the last marked point. 

For $m=1$ equation~\eqref{eq:string for A-class,reformulation} is clear. Suppose that $m\ge 2$. Consider a stable tree $\Gamma\in\ST^m_{g,n+2}$. Recall that we denote by $l_i(\Gamma)$ the leg of $\Gamma$ marked by $0\le i\le n+1$. We will call a vertex $v\in V(\Gamma)$ exceptional, if $g(v)=0$, $n(v)=3$ and the leg $l_{n+1}(\Gamma)$ is incident to~$v$. An exceptional vertex $v\in V(\Gamma)$ will be called bad, if it is not incident to any leg $l_i(\Gamma)$, where $1\le i\le n$. We will call the tree $\Gamma$ bad, if it has a bad vertex. Otherwise, it will be called good. For a vertex $v\in V(\Gamma)$ let
$$
r'(v):=
\begin{cases}
2g(v)+n(v)-2,&\text{if $l_{n+1}(\Gamma)$ is not incident to $v$},\\
2g(v)+n(v)-3,&\text{if $l_{n+1}(\Gamma)$ is incident to $v$}.
\end{cases}
$$    
For a good stable tree $\Gamma\in\ST^m_{g,n+2}$ introduce a constant $a'(\Gamma)$ by
$$
a'(\Gamma):=\left(\prod_{h\in H^e_+(\Gamma)}a(h)\right)\prod_{\substack{v\in V(\Gamma)\\\text{$v$ is not exceptional}}}\frac{r'(v)}{\sum_{\tv\in\Desc[v]}r'(\tv)}.
$$
Using these notations, we can rewrite the right-hand side of~\eqref{eq:string for A-class,reformulation} as follows:
\begin{multline*}
\pi^*\hA^{g,m}(a_0,\ldots,a_n)+\sum_{i=1}^na_i\delta^{\{i,n+1\}}_0\pi^*\hA^{g,m-1}(a_0,\ldots,a_n)=\\
=\sum_{\substack{\Gamma\in\ST^m_{g,n+2}\\\text{$\Gamma$ is good}}}a'(\Gamma)\lambda_g\DR_\Gamma\left(a_0,a_1,\ldots,a_n,0\right).
\end{multline*}
On the other hand, by definition,
$$
\hA^{g,m}(a_0,\ldots,a_n,0)=\sum_{\Gamma\in\ST^m_{g,n+2}}a(\Gamma)\lambda_g\DR_\Gamma\left(a_0,a_1,\ldots,a_n,0\right).
$$
We see that we have to prove the equation
\begin{gather}\label{eq:string for A-class,good and bad}
\sum_{\Gamma\in\ST^m_{g,n+2}}a(\Gamma)\lambda_g\DR_\Gamma\left(a_0,a_1,\ldots,a_n,0\right)=\sum_{\substack{\Gamma\in\ST^m_{g,n+2}\\\text{$\Gamma$ is good}}}a'(\Gamma)\lambda_g\DR_\Gamma\left(a_0,a_1,\ldots,a_n,0\right).
\end{gather}

Let us prove equation~\eqref{eq:string for A-class,good and bad}. Suppose $\Gamma$ is a bad stable tree. Let us show how to express the class $a(\Gamma)\lambda_g\DR_\Gamma(a_0,a_1,\ldots,a_n,0)$ as a linear combination of the classes $\lambda_g\DR_{\tGamma}(a_0,a_1,\ldots,a_n,0)$, where the stable trees $\tGamma$ are good. Suppose that $s\ge 2$ and $b_1,\ldots,b_s$ are integers with vanishing sum. We have the following relation in the cohomology of~$\oM_{g,s+2}$ (see e.g.~\cite[eq.~(5.2)]{Bur15}):
\begin{gather}\label{eq:string for A-class,WDVV}
\lambda_g\sum_{\substack{I\sqcup J=\{1,\ldots,s\}\\I,J\ne\emptyset}}\sum_{g_1+g_2=g}b_I\DR_{g_1}(0,B_I,-b_I)\boxtimes_1\DR_{g_2}(0,B_J,-b_J)=0.
\end{gather}
Suppose that the point with the zero multiplicity in the second double ramification cycle is marked by $s+2$. Let us multiply relation~\eqref{eq:string for A-class,WDVV} by $\psi_{s+2}$ and push it forward to $\oM_{g,s+1}$ by forgetting the last marked point:
\begin{gather}\label{eq:string for A-class,WDVV2}
\lambda_g\sum_{\substack{I\sqcup J=\{1,\ldots,s\}\\I,J\ne\emptyset}}\sum_{\substack{g_1+g_2=g\\2g_2+|J|-1>0}}b_I(2g_2+|J|-1)\DR_{g_1}(0,B_I,-b_I)\boxtimes_1\DR_{g_2}(B_J,-b_J)=0.
\end{gather}
Suppose that the level of the bad vertex in our bad stable tree $\Gamma$ is equal to~$k$. Then relation~\eqref{eq:string for A-class,WDVV2} allows to express the class $a(\Gamma)\lambda_g\DR_\Gamma(a_0,\ldots,a_n,0)$ in terms of the classes $\lambda_g\DR_{\tGamma}(a_0,\ldots,a_n,0)$, where the tree $\tGamma$ is good or bad with the bad vertex of level $k+1$. Therefore, applying relation~\eqref{eq:string for A-class,WDVV2} sufficiently many times, we come to a decomposition
$$
a(\Gamma)\lambda_g\DR_\Gamma(a_0,\ldots,a_n,0)=\sum_{\substack{\tGamma\in\ST^m_{g,n+2}\\\text{$\tGamma$ is good}}}a(\Gamma,\tGamma)\lambda_g\DR_{\tGamma}(a_0,\ldots,a_n,0),
$$ 
where $a(\Gamma,\tGamma)$ are certain coefficients. We see that for any good graph $\Gamma$ we have to prove the identity
\begin{gather}\label{eq:string for A-class,identity}
a(\Gamma)+\sum_{\substack{\tGamma\in\ST^m_{g,n+2}\\\text{$\tGamma$ is bad}}}a(\tGamma,\Gamma)=a'(\Gamma).
\end{gather}

Let us prove~\eqref{eq:string for A-class,identity}. Suppose that the leg $l_{n+1}=l_{n+1}(\Gamma)$ is incident to a vertex of level $k$. Denote it by $v_k$. Denote by~$v_1$ the root of $\Gamma$. Let $v_1,v_2,\ldots,v_k$ be the unique path connecting~$v_1$ and~$v_k$. Denote by $v^1_{k+1},\ldots,v^l_{k+1}$, $l\ge 0$, the direct descendants of $v_k$. Let $L':=L[v_k]\backslash\{l_{n+1}\}$. In Fig.~\ref{fig:string for A-class,1} we draw our tree $\Gamma$. 
\begin{figure}[t]
\begin{tikzpicture}[scale=1.1]
\legmm{9.5,0}{-90}{1.1}{l_{n+1}};\legg{5.5,0}{0}{0.8};\legg{8,0}{180}{0.8};
\legg{5.5,0}{-110}{1.1};\legg{5.5,0}{-70}{1.1};
\legg{8,0}{-110}{1.1};\legg{8,0}{-70}{1.1};\legg{11,1}{20}{1.1};\legg{11,1}{-20}{1.1};\legg{11,-1}{20}{1.1};\legg{11,-1}{-20}{1.1};
\legg{9.5,0}{110}{1};\legg{9.5,0}{70}{1};
\draw (8,0) -- (9.5,0);\draw (9.5,0) -- (11,1);\draw (9.5,0) -- (11,-1);
\ggg{r_1}{5.5,0};\ggg{r_{k-1}}{8,0};\ggg{r_k}{9.5,0};\ggg{r^1_{k+1}}{11,1};
\ggg{r^l_{k+1}}{11,-1};
\coordinate [label=center:$\ldots$] () at (6.75,0);
\coordinate [label=center:$\vdots$] () at (11,0.1);
\draw (5.5,-1.2) ellipse (0.7 and 0.5);
\coordinate [label=center:$A_1$] () at (5.5,-1.2);
\draw (8,-1.2) ellipse (0.7 and 0.5);
\coordinate [label=center:$A_{k-1}$] () at (8,-1.2);
\draw (12.2,1) ellipse (0.5 and 0.7);
\coordinate [label=center:$A^1_{k+1}$] () at (12.2,1);
\draw (12.2,-1) ellipse (0.5 and 0.7);
\coordinate [label=center:$A^l_{k+1}$] () at (12.2,-1);
\coordinate [label=center:$\ldots$] () at (9.5,0.9);
\coordinate [label=center:$\overbrace{\phantom{aaaaa}}^{L'}$] () at (9.5,1.15);
\end{tikzpicture}
\caption{Stable tree $\Gamma$}
\label{fig:string for A-class,1}
\end{figure}
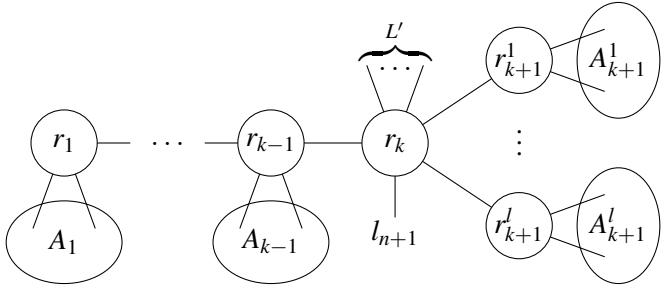
Note that each vertex $v$ in the picture is decorated by the number~$r(v)$, instead of the genus. This is more convenient for the computations. We use the notations $r_i:=r(v_i)$, $1\le i\le k$, and $r^j_{k+1}:=r(v^j_{k+1})$, $1\le j\le l$. The symbols $A_i$ and~$A^j_{k+1}$ indicate the pieces of the tree~$\Gamma$ that don't contain the vertices $v_i$ and $v^j_{k+1}$. Let us also introduce the following notations:
\begin{align*}
R_i:=&\sum_{v\in\Desc[v_i]}r(v),\quad 1\le i\le k,\\
R^j_{k+1}:=&\sum_{v\in\Desc[v^j_{k+1}]}r(v),\quad 1\le j\le l,
\end{align*}
\begin{gather*}
\ta:=a(\Gamma)\left/\left(\prod_{i=1}^k\frac{r_i}{R_i}\prod_{j=1}^l\frac{r^j_{k+1}}{R^j_{k+1}}\right)\right..
\end{gather*}
There are two cases: the vertex $v_k$ can be exceptional or not. 

Suppose that $v_k$ is not exceptional. Then 
$$
a'(\Gamma)=\ta\frac{r_1\cdots r_{k-1}(r_k-1)}{(R_1-1)\cdots (R_k-1)}\prod_{j=1}^l\frac{r^j_{k+1}}{R^j_{k+1}}.
$$
It is not hard to understand the structure of bad stable trees $\tGamma$ such that $a(\tGamma,\Gamma)\ne 0$. These trees are of two types. A bad tree of of the first type is shown in Fig.~\ref{fig:string for A-class,2},
\begin{figure}[t]
\begin{tikzpicture}[scale=0.95]
\legg{0,0}{0}{0.8};\legg{2.5,0}{180}{0.8};\legmm{4,0}{-90}{1.1}{l_{n+1}};\legg{5.5,0}{0}{0.8};\legg{8,0}{180}{0.8};
\legg{0,0}{-110}{1.1};\legg{0,0}{-70}{1.1};\legg{2.5,0}{-110}{1.1};\legg{2.5,0}{-70}{1.1};\legg{5.5,0}{-110}{1.1};\legg{5.5,0}{-70}{1.1};
\legg{8,0}{-110}{1.1};\legg{8,0}{-70}{1.1};\legg{9.5,0}{-110}{1.1};\legg{9.5,0}{-70}{1.1};\legg{11,1}{20}{1.1};\legg{11,1}{-20}{1.1};\legg{11,-1}{20}{1.1};\legg{11,-1}{-20}{1.1};
\legg{9.5,0}{110}{1};\legg{9.5,0}{70}{1};
\draw (2.5,0) -- (4,0);\draw (4,0) -- (5.5,0);\draw (8,0) -- (9.5,0);\draw (9.5,0) -- (11,1);\draw (9.5,0) -- (11,-1);
\ggg{r_1}{0,0};\ggg{r_{i-1}}{2.5,0};\ggg{1}{4,0};\ggg{1}{4,0};\ggg{r_i}{5.5,0};\ggg{r_{k-2}}{8,0};\ggg{\widetilde{r}}{9.5,0};\ggg{r^1_{k+1}}{11,1};
\ggg{r^l_{k+1}}{11,-1};
\coordinate [label=center:$\ldots$] () at (1.25,0);
\coordinate [label=center:$\ldots$] () at (6.75,0);
\coordinate [label=center:$\vdots$] () at (11,0.1);
\draw (0,-1.2) ellipse (0.7 and 0.5);
\coordinate [label=center:$A_1$] () at (0,-1.2);
\draw (2.5,-1.2) ellipse (0.7 and 0.5);
\coordinate [label=center:$A_{i-1}$] () at (2.5,-1.2);
\draw (5.5,-1.2) ellipse (0.7 and 0.5);
\coordinate [label=center:$A_i$] () at (5.5,-1.2);
\draw (8,-1.2) ellipse (0.7 and 0.5);
\coordinate [label=center:$A_{k-2}$] () at (8,-1.2);
\draw (9.5,-1.2) ellipse (0.7 and 0.5);
\coordinate [label=center:$A_{k-1}$] () at (9.5,-1.2);
\draw (12.2,1) ellipse (0.5 and 0.7);
\coordinate [label=center:$A^1_{k+1}$] () at (12.2,1);
\draw (12.2,-1) ellipse (0.5 and 0.7);
\coordinate [label=center:$A^l_{k+1}$] () at (12.2,-1);
\coordinate [label=center:$\ldots$] () at (9.5,0.9);
\coordinate [label=center:$\overbrace{\phantom{aaaaa}}^{L'}$] () at (9.5,1.15);
\end{tikzpicture}
\caption{Bad stable tree of the first type}
\label{fig:string for A-class,2}
\end{figure}
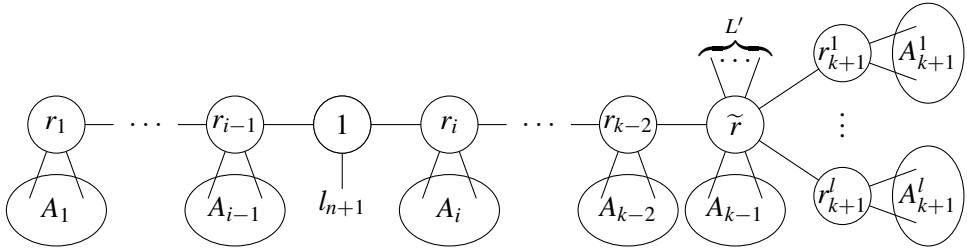 
where $1\le i\le k-1$ and $\widetilde{r}=r_{k-1}+r_k-1$. A bad tree of the second type is shown in Fig.~\ref{fig:string for A-class,3}, 
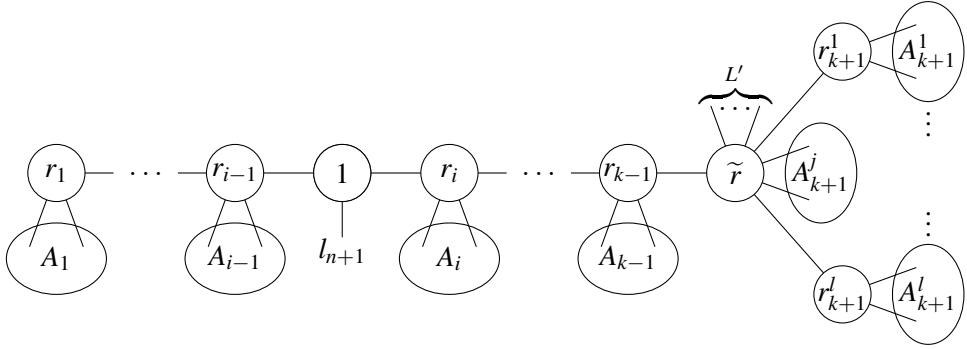
\begin{figure}[t]
\begin{tikzpicture}[scale=0.95]
\legg{0,0}{0}{0.8};\legg{2.5,0}{180}{0.8};\legmm{4,0}{-90}{1.1}{l_{n+1}};\legg{5.5,0}{0}{0.8};\legg{8,0}{180}{0.8};
\legg{0,0}{-110}{1.1};\legg{0,0}{-70}{1.1};\legg{2.5,0}{-110}{1.1};\legg{2.5,0}{-70}{1.1};\legg{5.5,0}{-110}{1.1};\legg{5.5,0}{-70}{1.1};
\legg{8,0}{-110}{1.1};\legg{8,0}{-70}{1.1};\legg{11,1.7}{20}{1.1};\legg{11,1.7}{-20}{1.1};\legg{11,-1.7}{20}{1.1};\legg{11,-1.7}{-20}{1.1};
\legg{9.5,0}{110}{1};\legg{9.5,0}{70}{1};\legg{9.5,0}{20}{1.1};\legg{9.5,0}{-20}{1.1};
\draw (2.5,0) -- (4,0);\draw (4,0) -- (5.5,0);\draw (8,0) -- (9.5,0);\draw (9.5,0) -- (11,1.7);\draw (9.5,0) -- (11,-1.7);
\ggg{r_1}{0,0};\ggg{r_{i-1}}{2.5,0};\ggg{1}{4,0};\ggg{1}{4,0};\ggg{r_i}{5.5,0};\ggg{r_{k-1}}{8,0};\ggg{\widetilde{r}}{9.5,0};
\ggg{r^1_{k+1}}{11,1.7};\ggg{r^l_{k+1}}{11,-1.7};
\coordinate [label=center:$\ldots$] () at (1.25,0);
\coordinate [label=center:$\ldots$] () at (6.75,0);
\coordinate [label=center:$\vdots$] () at (12.2,0.8);
\coordinate [label=center:$\vdots$] () at (12.2,-0.6);
\draw (0,-1.2) ellipse (0.7 and 0.5);
\coordinate [label=center:$A_1$] () at (0,-1.2);
\draw (2.5,-1.2) ellipse (0.7 and 0.5);
\coordinate [label=center:$A_{i-1}$] () at (2.5,-1.2);
\draw (5.5,-1.2) ellipse (0.7 and 0.5);
\coordinate [label=center:$A_i$] () at (5.5,-1.2);
\draw (8,-1.2) ellipse (0.7 and 0.5);
\coordinate [label=center:$A_{k-1}$] () at (8,-1.2);
\draw (12.2,1.7) ellipse (0.5 and 0.7);
\coordinate [label=center:$A^1_{k+1}$] () at (12.2,1.7);
\draw (12.2,-1.7) ellipse (0.5 and 0.7);
\coordinate [label=center:$A^l_{k+1}$] () at (12.2,-1.7);
\coordinate [label=center:$\ldots$] () at (9.5,0.9);
\draw (10.7,0) ellipse (0.5 and 0.7);
\coordinate [label=center:$A^j_{k+1}$] () at (10.7,0);
\coordinate [label=center:$\overbrace{\phantom{aaaaa}}^{L'}$] () at (9.5,1.15);
\end{tikzpicture}
\caption{Bad stable tree of the second type}
\label{fig:string for A-class,3}
\end{figure}
where $1\le i\le k$, $1\le j\le l$ and $\widetilde{r}=r_k+r^j_{k+1}-1$. It is not hard to see that
\begin{gather*}
a(\tGamma,\Gamma)=
\begin{cases}
\ta\frac{r_1\cdots r_{k-1}}{R_1\cdots R_i(R_i-1)\cdots(R_{k-1}-1)}\prod_{j=1}^l\frac{r^j_{k+1}}{R^j_{k+1}},&\text{if $\tGamma$ is of the first type},\\
-\ta\frac{r_1\cdots r_{k-1}}{R_1\cdots R_i(R_i-1)\cdots(R_k-1)}R^j_{k+1}\prod_{m=1}^l\frac{r^m_{k+1}}{R^m_{k+1}},&\text{if $\tGamma$ is of the second type}.
\end{cases}
\end{gather*}
Therefore, equation~\eqref{eq:string for A-class,identity} follows from the identity
\begin{align*}
&\prod_{i=1}^k\frac{r_i}{R_i}+\sum_{i=1}^{k-1}\frac{r_1\cdots r_{k-1}}{R_1\cdots R_i(R_i-1)\cdots(R_{k-1}-1)}\\
&-\sum_{i=1}^k\sum_{j=1}^l\frac{r_1\cdots r_{k-1}}{R_1\cdots R_i(R_i-1)\cdots(R_k-1)}R^j_{k+1}=\\
=&\frac{r_1\cdots r_{k-1}(r_k-1)}{(R_1-1)\cdots (R_k-1)},
\end{align*}
or, equivalently,
\begin{align}
&\frac{r_k}{R_1\cdots R_k}+\sum_{i=1}^{k-1}\frac{1}{R_1\cdots R_i(R_i-1)\cdots(R_{k-1}-1)}\label{string for A,tmp1}\\
&-\sum_{i=1}^k\frac{R_k-r_k}{R_1\cdots R_i(R_i-1)\cdots(R_k-1)}=\notag\\
=&\frac{r_k-1}{(R_1-1)\cdots (R_k-1)}.\notag
\end{align}
Note that
\begin{align*}
&\frac{r_k}{R_1\cdots R_k}-\frac{R_k-r_k}{R_1\cdots R_k(R_k-1)}=\frac{r_k-1}{R_1\cdots R_{k-1}(R_k-1)},\\
&\frac{1}{R_1\cdots R_i(R_i-1)\cdots(R_{k-1}-1)}-\frac{R_k-r_k}{R_1\cdots R_i(R_i-1)\cdots(R_k-1)}=\\
=&\frac{r_k-1}{R_1\cdots R_i(R_i-1)\cdots(R_k-1)},
\end{align*}
where $1\le i\le k-1$. Therefore, equation~\eqref{string for A,tmp1} is equivalent to the equation
\begin{gather}\label{eq:string for A-class,elementary identity}
\frac{1}{R_1\cdots R_{k-1}}+\sum_{i=1}^{k-1}\frac{1}{R_1\cdots R_i(R_i-1)\cdots(R_{k-1}-1)}=\frac{1}{(R_1-1)\cdots (R_{k-1}-1)},
\end{gather}
which can be easily proved by induction on~$k$. 

Suppose that $v_k$ is exceptional. Then $l=0$, the set $L'$ consists of only one leg and $r_k=R_k=1$. We have
$$
a'(\Gamma)=\ta\frac{r_1\cdots r_{k-1}}{(R_1-1)\cdots (R_{k-1}-1)}.
$$
A bad stable tree $\tGamma$ with $a(\tGamma,\Gamma)\ne 0$ should necessarily be of the first type (see Fig.~\ref{fig:string for A-class,2}) and then we have
\begin{gather*}
a(\tGamma,\Gamma)=\ta\frac{r_1\cdots r_{k-1}}{R_1\cdots R_i(R_i-1)\cdots(R_{k-1}-1)}.
\end{gather*}
We immediately see that again equation~\eqref{eq:string for A-class,identity} follows from the elementary identity~\eqref{eq:string for A-class,elementary identity}. The proposition is proved.
\end{proof}

\begin{proposition}\label{proposition:pullback of B-class}
Denote by $\pi\colon\oM_{g,n+1}\to\oM_{g,n}$ the forgetful map that forgets the last marked point. Then we have
\begin{gather}\label{eq:pullback of B-class}
B^g_{d_1,\ldots,d_n,0}=
\begin{cases}
\pi^* B^g_{d_1,\ldots,d_n},&\text{if $\sum d_i=2g-1$},\\
\pi^* B^g_{d_1,\ldots,d_n}+\sum_{\substack{1\le i\le n\\d_i\ge 1}}\delta_0^{\{i,n+1\}}\pi^* B^g_{d_1,\ldots,d_i-1,\ldots,d_n},&\text{if $\sum d_i\ge 2g$}.
\end{cases}
\end{gather}
\end{proposition}

\begin{proof}
Let $(T,q) \in \Omega^{B,g}_{d_1,\dotsc,d_n}$ be an admissible and stable complete tree with a power function
$$q \colon H^{em}_+(T) \to \NN,$$
as in Definition \ref{B-class}.
We denote by $\deg(T)$ its number of levels.
In particular, there are extra legs at every vertex (except the root) that we will eventually forget when computing the $B$-class.

Choose a vertex $v \in V(T)$.
Let $C=(e^C_1,v^C_1,\dotsc,e^C_{\deg(T)-l(v)},v^C_{\deg(T)-l(v)},\sigma_{n+1})$ be a chain of weakly stable vertices with a new marking $\sigma_{n+1}$.
Precisely, the edge $e^C_1$ is attached to the vertex $v^C_1$, the edge $e^C_k$ links the vertex $v^C_{k-1}$ to $v^C_k$, and the leg $\sigma_{n+1}$ is attached to the vertex~$v^C_{\deg(T)-l(v)}$.
Moreover, every vertex is of genus $0$ and contains an extra leg.
We construct a tree $T_v$, obtained from $T$ by gluing the edge $e^C_1$ (and thus the chain $C$) to the vertex $v$.
We have $H^{em}_+(T) \subset H^{em}_+(T_v)$ and we extend the power function $q$ into a function
$q_v \colon H^{em}_+(T_v) \to \NN$
by taking $$q_v(h^C_k):=0 \quad \textrm{and} \quad q_v(\sigma_{n+1}):=0,$$
where $h^C_k$ is the half-edge in $H^{em}_+(T_v)$ contained in the edge $e^C_k$.
It is easy to see that we get
$$(T_v,q_v) \in \Omega^{B,g}_{d_1,\dotsc,d_n,0}.$$

Choose a half-edge $h \in H^{em}_+(T)$ attached to the vertex $v$ and such that $q(h)>0$.
We construct a tree $T_{(v,h)}$, obtained from $T$ by adding an extra level between the levels~$l(v)$ and~$l(v)+1$ of~$T$ as follows:
\begin{itemize}
	\item[-] denote by $h_0,\dotsc,h_m \in H^{em}_+(T)$ the half-edges of level $l(v)$, with $h_0:=h$, 
	\item[-] insert a pair $(e_k,v_k)$ between the half-edge $h_k$ and the vertex it is attached to, where $e_k=(h'_k,h''_k)$ is an edge and $v_k$ is a vertex of genus $0$,
	\item[-] glue the half-edge $h^C_1$ from the chain $C$ to the vertex $v_0$,
	\item[-] add $q(h)$ extra legs to the vertex $v_0$ and $q(h_k)+1$ extra leg to the vertex $v_k$, for $1 \leq k \leq m$.
\end{itemize}
Therefore, the number of levels of the tree $T_{(v,h)}$ is $\deg(T)+1$, the vertex $v_0 \in V(T_{(v,h)})$ is the only strongly stable vertex at its level, and we have a natural inclusion $H^{em}_+(T_v) \subset H^{em}_+(T_{(v,h)})$.
Then, we extend the power function $q_v$ into a function
$q_{(v,h)} \colon H^{em}_+(T_{(v,h)}) \to \NN$
by taking
$$q_{(v,h)}(h'_k) := \left\lbrace \begin{array}{ll}
q(h_k)-1, & \textrm{if $k=0$,} \\
q(h_k), & \textrm{if $k \neq 0$.} \\
\end{array}\right.$$
The complete tree $T_{(v,h)}$ is obviously stable, but not necessarily admissible.
We get
$$(T_{(v,h)},q_{(v,h)}) \in \Omega^{B,g}_{d_1,\dotsc,d_n,0} \iff l(v) \neq \deg(T) ~~ \textrm{or} ~~ \sum_{i=1}^n d_i = 2g-1.$$
Furthermore, observe that when $l(v)=\deg(T)$, then the half-edge $h$ corresponds to a marking~$\sigma_i$ and we get
$$\ee_* [T_{(v,h)},q_{(v,h)}] = {\sigma_i}_* \ee_* [T,q_i] = \delta_0^{\{i,n+1\}} \cdot \pi^* \ee_* [T,q_i] \in R^*(\overline{\cM}_{g,n+1}),$$
where the morphism $\sigma_i$ denotes here the section of the $i$-th marking in the universal curve $\mathcal{C}_{g,n} \simeq \overline{\cM}_{g,n+1}$, and where $q_i \colon H^{em}_+(T) \to \NN$ is defined by
$$q_{i}(h) := \left\lbrace \begin{array}{ll}
d_i-1, & \textrm{if $h=\sigma_i$,} \\
q(h), & \textrm{otherwise.} \\
\end{array}\right.$$

Conversely, let $(T',q') \in \Omega^{B,g}_{d_1,\dotsc,d_n,0}$ and denote by $v$ the first strongly stable ancestor of the marking $\sigma_{n+1}$.
In particular, the marking $\sigma_{n+1}$ is attached to the vertex $v$ via a chain $C$ of weakly stable vertices and we denote by $h^{(n+1)} \in H^{em}_+(T')$ the half-edge from $C$ attached to $v$.
We have two possibilities:
\begin{itemize}
	\item[$(1)$] $v$ is a vertex of genus $0$ with exactly two half-edges $h,h^{(n+1)} \in H^{em}_+(T')$ attached to it and $v$ is the only strongly stable vertex of level $l(v)$,
	\item[$(2)$] $v$ is another type of vertex.
\end{itemize}
Denote by $T$ the tree obtained from $T'$ by forgetting the chain $C$ containing the marking $\sigma_{n+1}$, and contracting the level $l(v)$ in case $(1)$.
In particular, the power function $q'$ restricts to a function $q$ and we get
$$(T,q) \in \Omega^{B,g}_{d_1,\dotsc,d_n} \quad \textrm{and} \quad (T',q') = \left\lbrace \begin{array}{ll}
(T_{(v,h)},q_{(v,h)}) & \textrm{in case $(1)$,} \\
(T_v,q_v) & \textrm{in case $(2)$}.
\end{array}\right.$$

Furthermore, from the formula
\begin{gather}\label{pullbackpsi}
\pi^* \tikz[baseline=-1mm]{\draw (A);\leg{A}{90};\leg{A}{-60};\leg{A}{-120};\gg{v}{A};\lab{A}{-90}{7mm}{q_1 \dotsm q_r};} =  \tikz[baseline=-1mm]{\draw (A);\leg{A}{90};\legm{A}{-30}{n+1};\leg{A}{-60};\leg{A}{-120};\gg{v}{A};\lab{A}{-90}{7mm}{q_1 \dotsm q_r};} - \sum_{\substack{1 \leq i \leq r \\ q_i>0}} \tikz[baseline=-1mm]{\draw (A)--(B)--(C);\leg{A}{90};\leg{A}{-60};\leg{A}{-120};\gg{v}{A};\lab{A}{-90}{7mm}{q_1 \dotsm \hat{q_i} \dotsm q_r};\legm{C}{-90}{i};\legm{C}{-30}{n+1};\gg{0}{C};\lab{A}{15}{5.5mm}{q_i-1};}
\end{gather}
expressing the pullback of $\psi$-classes via the map $\pi$, we obtain
$$e_*\pi^*([T,q]) = \sum_{v \in T} e_* \left( [T_v,q_v] - \sum_{\substack{h \in H^{em}_+(T) \\ h \to v, q(h)>0}} [T_{(v,h)},q_{(v,h)}] \right),$$
for every $(T,q) \in \Omega^{B,g}_{d_1,\dotsc,d_n}$ and where $h \to v$ means that the half-edge $h$ is incident to the vertex~$v$.
As a consequence, when $d_1+\dotsb+d_n \geq 2g$, we obtain
\begin{align*}
\pi^*(B^g_{d_1,\dotsc,d_n}) = & \sum_{(T,q) \in \Omega^{B,g}_{d_1,\dotsc,d_n}} (-1)^{\deg(T)-1} \pi^* \ee_* [T,q]= \\
 = & \sum_{(T,q) \in \Omega^{B,g}_{d_1,\dotsc,d_n}} (-1)^{\deg(T)-1} \ee_* \pi^* [T,q], \\
\end{align*}
where the second equality comes from the general fact that $\overline{\cM}_{g,n+2}$ is birational to the fiber product $\overline{\cM}_{g,n+1} \times_{\overline{\cM}_{g,n}} \overline{\cM}_{g,n+1}$, and then
\begin{align*}
\pi^*(B^g_{d_1,\dotsc,d_n}) = & \sum_{(T,q) \in \Omega^{B,g}_{d_1,\dotsc,d_n}} (-1)^{\deg(T)-1} \sum_{v \in T} \left( \ee_* [T_v,q_v] - \sum_{\substack{h \in H^{em}_+(T) \\ h \to v, q(h)>0}} \ee_* [T_{(v,h)},q_{(v,h)}] \right)= \\
 = & \sum_{(T',q') \in \Omega^{B,g}_{d_1,\dotsc,d_n,0}} (-1)^{\deg(T')-1} \ee_* [T',q'] \\
 & - \sum_{(T,q) \in \Omega^{B,g}_{d_1,\dotsc,d_n}} (-1)^{\deg(T)-1} \sum_{\substack{v \in T\\ l(v)=\deg(T)}} \sum_{\substack{h \in H^{em}_+(T) \\ h \to v, q(h)>0}} \ee_* [T_{(v,h)},q_{(v,h)}]= \\
 = & B^g_{d_1,\dotsc,d_n,0} - \sum_{(T,q) \in \Omega^{B,g}_{d_1,\dotsc,d_n}} (-1)^{\deg(T)-1} \sum_{\substack{1 \leq i \leq n \\ d_i > 0}} \delta_0^{\{i,n+1\}} \cdot \pi^*\ee_*[T,q_i]= \\
 = & B^g_{d_1,\dotsc,d_n,0} - \sum_{\substack{1 \leq i \leq n \\ d_i > 0}} \delta_0^{\{i,n+1\}} \cdot \pi^*B^g_{d_1,\dotsc, d_i-1, \dotsc,d_n}.
\end{align*}

When $d_1+\dotsb+d_n=2g-1$, then we have seen that $(T_{(v,h)},q_{(v,h)})$ is always admissible, so that the first three equalities are the same, but there is no second term in the last three equalities.
Hence we get
$$\pi^*B^g_{d_1,\dotsc,d_n} = B^g_{d_1,\dotsc,d_n,0}.$$
\end{proof}

\subsection{Reduction of the conjecture}\label{subsection:reduction}

\begin{proposition}
Conjecture~\ref{conjecture} is true if and only if it is true when all $d_i$'s are positive. Furthermore, Conjecture~\ref{conjecture} is true in genus $0$ and in genus $1$.
\end{proposition}

\begin{proof}
The first statement follows immediately from Propositions~\ref{proposition:pullback of A-class} and~\ref{proposition:pullback of B-class}.

Assume $g=0$.
Since $\dim\oM_{0,n}=n-3$, the classes $A^0_{d_1,\ldots,d_n}$ and $B^0_{d_1,\ldots,d_n}$ are non-trivial only if $\sum d_i\le n-3$.
Therefore, we can always apply formulas~\eqref{eq:pullback of A-class} and~\eqref{eq:pullback of B-class} to them, unless $n=3$ and $d_1=d_2=d_3=0$, where we get $$A^0_{0,0,0}=B^0_{0,0,0}=1\in H^0(\oM_{0,3},\mbQ).$$

Assume $g=1$.
Since $\dim\oM_{1,n}=n$, the classes $A^1_{d_1,\ldots,d_n}$ and $B^1_{d_1,\ldots,d_n}$ are non-trivial only if $\sum d_i\le n$.
Therefore, we can always apply formulas~\eqref{eq:pullback of A-class} and~\eqref{eq:pullback of B-class} to them, unless $d_1=d_2=\ldots=d_n=1$.
In order to prove that $A^1_{1,1,\ldots,1}=B^1_{1,1,\ldots,1}$, it is sufficient to check that $\int_{\oM_{1,n}}A^1_{1,1,\ldots,1}=\int_{\oM_{1,n}}B^1_{1,1,\ldots,1}$.
Note that these two integrals are equal to $\<\tau_1(e_1)^n\>_1^\DR$ and~$\<\tau_1(e_1)^n\>^\red_1$, respectively, for the trivial cohomological field theory. The equality $F^\DR=F^\red$ for the trivial cohomological field theory was checked in~\cite{BDGR16a}. Therefore, Conjecture~\ref{conjecture} is true in genus~$1$.
\end{proof}

\subsection{Dilaton equation}\label{subsection:dilaton equation}

Here we prove that the classes $A^g_{d_1,\ldots,d_n,1}$ and~$B^g_{d_1,\ldots,d_n,1}$ behave in the same way upon the pushforward along the map forgetting the last marked point.

\begin{proposition}\label{proposition:dilatonA}
Denote by $\pi\colon\oM_{g,n+1}\to\oM_{g,n}$ the forgetful map that forgets the last marked point. Then we have
\begin{gather}\label{eq:dilaton of A-class}
\pi_* (A^g_{d_1,\ldots,d_n,1})=
\begin{cases}
(2g-2+n)A^g_{d_1,\ldots,d_n},&\text{if $\sum d_i>2g-2$},\\
0,&\text{if $\sum d_i=2g-2$}.
\end{cases}
\end{gather}
\end{proposition}
Before proving the proposition let us formulate three auxiliary statements. Recall that for a stable tree $\Gamma\in\ST^m_{g,n+1}$ we denote by $v_1(\Gamma)$ the root of $\Gamma$ and by $l_i(\Gamma)$, $0\le i\le n$, the leg of $\Gamma$ marked by $i$.
\begin{lemma}\label{lemma:psi times hA}
Let $a_0,\ldots,a_n$, $n\ge 1$, be integers with vanishing sum and $m\ge 2$. Then we have
\begin{multline*}
\hA^{g,m}(a_0,\ldots,a_n)-a_1\psi_1\hA^{g,m-1}(a_0,\ldots,a_n)=\\
=\sum_{\substack{\Gamma\in\ST^m_{g,n+1}\\v(l_1(\Gamma))=v_1(\Gamma)}}\frac{2g-1+n}{r(v_1(\Gamma))}a(\Gamma)\lambda_g\DR_\Gamma(a_0,\ldots,a_n).
\end{multline*}
\end{lemma}
\begin{proof}
Using formula~\eqref{eq:DR times psi}, for an arbitrary stable tree $\Gamma\in\ST^{m-1}_{g,n+1}$ we can write a decomposition
$$
a_1\psi_1\cdot a(\Gamma)\lambda_g\DR_\Gamma(a_0,\ldots,a_n)=\sum_{\tGamma\in\ST^m_{g,n+1}}a(\Gamma,\tGamma)\lambda_g\DR_{\tGamma}(a_0,\ldots,a_n),
$$
where $a(\Gamma,\tGamma)$ are certain coefficients. Let $\Gamma\in\ST^m_{g,n+1}$. The statement of the lemma is equivalent to the following equation:
\begin{gather}\label{eq:dilaton for A-class,identity}
a(\Gamma)-\sum_{\tGamma\in\ST^{m-1}_{g,n+1}}a(\tGamma,\Gamma)=
\begin{cases}
\frac{2g-1+n}{r(v_1(\Gamma))},&\text{if $l_1(\Gamma)$ is incident to $v_1(\Gamma)$},\\
0,&\text{otherwise}.
\end{cases}
\end{gather}
Let $v\in V(\Gamma)$ be the vertex incident to $l_1=l_1(\Gamma)$. Denote by $v''_1,\ldots,v''_l$, $l\ge 0$, the direct descendants of~$v$. Let $L':=L[v]\backslash\{l_1\}$, $r:=r(v)$, $r''_i:=r(v''_i)$, $R:=\sum_{\tv\in\Desc[v]}r(\tv)$ and $R''_i:=\sum_{\tv\in\Desc[v''_i]}r(\tv)$.

Suppose that $v\ne v_1(\Gamma)$. Denote by $v'\in V(\Gamma)$ the mother of $v$ and let $r':=r(v')$ and $R':=\sum_{\tv\in\Desc[v']}r(\tv)$. We draw the stable tree $\Gamma$ in Fig.~\ref{fig:dilaton for A-class,1}. 
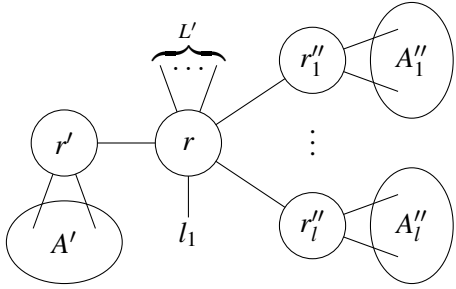
\begin{figure}[t]
\begin{tikzpicture}[scale=1.1]
\legmm{9.5,0}{-90}{1.1}{l_1};
\legg{8,0}{-110}{1.1};\legg{8,0}{-70}{1.1};\legg{11,1}{20}{1.1};\legg{11,1}{-20}{1.1};\legg{11,-1}{20}{1.1};\legg{11,-1}{-20}{1.1};
\legg{9.5,0}{110}{1};\legg{9.5,0}{70}{1};
\draw (8,0) -- (9.5,0);\draw (9.5,0) -- (11,1);\draw (9.5,0) -- (11,-1);
\ggg{r'}{8,0};\ggg{r}{9.5,0};\ggg{r''_1}{11,1};\ggg{r''_l}{11,-1};
\coordinate [label=center:$\vdots$] () at (11,0.1);
\draw (8,-1.2) ellipse (0.7 and 0.5);
\coordinate [label=center:$A'$] () at (8,-1.2);
\draw (12.2,1) ellipse (0.5 and 0.7);
\coordinate [label=center:$A''_1$] () at (12.2,1);
\draw (12.2,-1) ellipse (0.5 and 0.7);
\coordinate [label=center:$A''_l$] () at (12.2,-1);
\coordinate [label=center:$\ldots$] () at (9.5,0.9);
\coordinate [label=center:$\overbrace{\phantom{aaaaa}}^{L'}$] () at (9.5,1.15);
\end{tikzpicture}
\caption{Stable tree $\Gamma$}
\label{fig:dilaton for A-class,1}
\end{figure}
Similarly to the figures in the proof of Proposition~\ref{proposition:pullback of A-class}, we decorate a vertex $w$ of~$\Gamma$ by number $r(w)$. It is not hard to see that there are exactly $l+1$ stable trees $\tGamma\in\ST^{m-1}_{g,n+1}$ such that $a(\tGamma,\Gamma)\ne 0$. The first one is shown on the left-hand side of Fig.~\ref{fig:dilaton for A-class,2},
\begin{figure}[t]
\begin{tikzpicture}[scale=1.1]
\legmm{9.5,0}{180}{1.1}{l_1};
\legg{9.5,0}{-110}{1.1};\legg{9.5,0}{-70}{1.1};\legg{11,1}{20}{1.1};\legg{11,1}{-20}{1.1};\legg{11,-1}{20}{1.1};\legg{11,-1}{-20}{1.1};
\legg{9.5,0}{110}{1};\legg{9.5,0}{70}{1};
\draw (9.5,0) -- (11,1);\draw (9.5,0) -- (11,-1);
\ggg{\substack{r'+r}}{9.5,0};\ggg{r''_1}{11,1};\ggg{r''_l}{11,-1};
\coordinate [label=center:$\vdots$] () at (11,0.1);
\draw (9.5,-1.2) ellipse (0.7 and 0.5);
\coordinate [label=center:$A'$] () at (9.5,-1.2);
\draw (12.2,1) ellipse (0.5 and 0.7);
\coordinate [label=center:$A''_1$] () at (12.2,1);
\draw (12.2,-1) ellipse (0.5 and 0.7);
\coordinate [label=center:$A''_l$] () at (12.2,-1);
\coordinate [label=center:$\ldots$] () at (9.5,0.9);
\coordinate [label=center:$\overbrace{\phantom{aaaaa}}^{L'}$] () at (9.5,1.15);

\begin{scope}[shift={(7,0)}]
\legmm{9.5,0}{-90}{1.1}{l_1};
\legg{8,0}{-110}{1.1};\legg{8,0}{-70}{1.1};\legg{11,1.7}{20}{1.1};\legg{11,1.7}{-20}{1.1};\legg{11,-1.7}{20}{1.1};
\legg{11,-1.7}{-20}{1.1};\legg{9.5,0}{110}{1};\legg{9.5,0}{70}{1};\legg{9.5,0}{20}{1.1};\legg{9.5,0}{-20}{1.1};
\draw (8,0) -- (9.5,0);\draw (9.5,0) -- (11,1.7);\draw (9.5,0) -- (11,-1.7);
\ggg{r'}{8,0};\ggg{\substack{r+r''_j}}{9.5,0};
\ggg{r''_1}{11,1.7};\ggg{r''_l}{11,-1.7};
\coordinate [label=center:$\vdots$] () at (12.2,0.8);
\coordinate [label=center:$\vdots$] () at (12.2,-0.6);
\draw (8,-1.2) ellipse (0.7 and 0.5);
\coordinate [label=center:$A'$] () at (8,-1.2);
\draw (12.2,1.7) ellipse (0.5 and 0.7);
\coordinate [label=center:$A''_1$] () at (12.2,1.7);
\draw (12.2,-1.7) ellipse (0.5 and 0.7);
\coordinate [label=center:$A''_l$] () at (12.2,-1.7);
\coordinate [label=center:$\ldots$] () at (9.5,0.9);
\draw (10.7,0) ellipse (0.5 and 0.7);
\coordinate [label=center:$A''_j$] () at (10.7,0);
\coordinate [label=center:$\overbrace{\phantom{aaaaa}}^{L'}$] () at (9.5,1.15);
\end{scope}
\end{tikzpicture}
\caption{Stable trees $\tGamma$ such that $a(\tGamma,\Gamma)\ne 0$}
\label{fig:dilaton for A-class,2}
\end{figure}
and the other $l$ trees are on the right-hand side, where $1\le j\le l$. Let 
$$
\ta:=a(\Gamma)\left/\left(\frac{r'}{R'}\frac{r}{R}\prod_{j=1}^l\frac{r''_j}{R''_j}\right)\right..
$$
The coefficient $a(\tGamma,\Gamma)$ for the left tree in Fig.~\ref{fig:dilaton for A-class,2} is equal to $\ta\frac{r'}{R'}\prod_{k=1}^l\frac{r''_k}{R''_k}$ and for the right tree in Fig.~\ref{fig:dilaton for A-class,2} it is equal to $-\ta\frac{r'R''_j}{R'R}\prod_{k=1}^l\frac{r''_k}{R''_k}$. We compute
$$
\sum_{\tGamma\in\ST^{m-1}_{g,n+1}}a(\tGamma,\Gamma)=\ta\left(\frac{r'}{R'}-\sum_{j=1}^l\frac{r'R''_j}{R'R}\right)\prod_{k=1}^l\frac{r''_k}{R''_k}=\ta\frac{r'r}{R'R}\prod_{k=1}^l\frac{r''_k}{R''_k}=a(\Gamma).
$$ 
Therefore, formula~\eqref{eq:dilaton for A-class,identity} is proved in the case when $l_1$ is not incident to $v_1(\Gamma)$.

Suppose that $v=v_1(\Gamma)$. The tree $\Gamma$ and stable trees $\tGamma$ such that $a(\tGamma,\Gamma)\ne 0$ are shown in Fig.~\ref{fig:dilaton for A-class,3}.
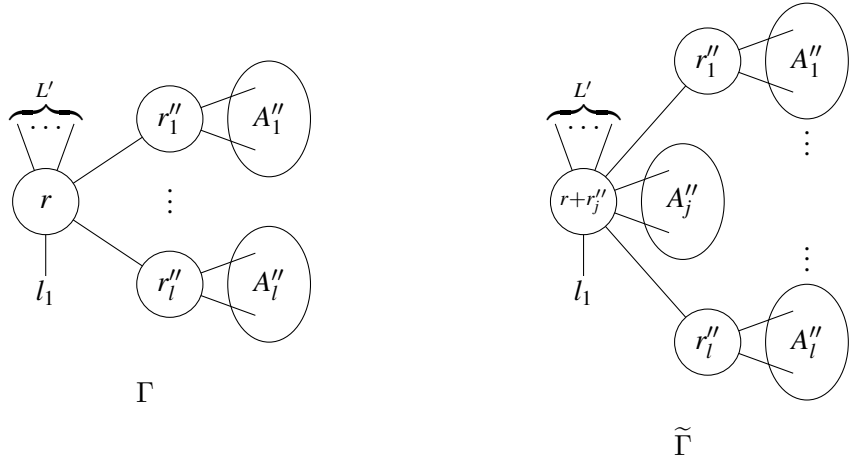
\begin{figure}[t]
\begin{tikzpicture}[scale=1.1]
\legmm{9.5,0}{-90}{1.1}{l_1};
\legg{11,1}{20}{1.1};\legg{11,1}{-20}{1.1};\legg{11,-1}{20}{1.1};\legg{11,-1}{-20}{1.1};
\legg{9.5,0}{110}{1};\legg{9.5,0}{70}{1};
\draw (9.5,0) -- (11,1);\draw (9.5,0) -- (11,-1);
\ggg{r}{9.5,0};\ggg{r''_1}{11,1};\ggg{r''_l}{11,-1};
\coordinate [label=center:$\vdots$] () at (11,0.1);
\draw (12.2,1) ellipse (0.5 and 0.7);
\coordinate [label=center:$A''_1$] () at (12.2,1);
\draw (12.2,-1) ellipse (0.5 and 0.7);
\coordinate [label=center:$A''_l$] () at (12.2,-1);
\coordinate [label=center:$\ldots$] () at (9.5,0.9);
\coordinate [label=center:$\overbrace{\phantom{aaaaa}}^{L'}$] () at (9.5,1.15);
\coordinate [label=center:$\Gamma$] () at (10.7,-2.3);

\begin{scope}[shift={(6.5,0)}]
\legmm{9.5,0}{-90}{1.1}{l_1};
\legg{11,1.7}{20}{1.1};\legg{11,1.7}{-20}{1.1};\legg{11,-1.7}{20}{1.1};
\legg{11,-1.7}{-20}{1.1};\legg{9.5,0}{110}{1};\legg{9.5,0}{70}{1};\legg{9.5,0}{20}{1.1};\legg{9.5,0}{-20}{1.1};
\draw (9.5,0) -- (11,1.7);\draw (9.5,0) -- (11,-1.7);
\ggg{\substack{r+r''_j}}{9.5,0};
\ggg{r''_1}{11,1.7};\ggg{r''_l}{11,-1.7};
\coordinate [label=center:$\vdots$] () at (12.2,0.8);
\coordinate [label=center:$\vdots$] () at (12.2,-0.6);
\draw (12.2,1.7) ellipse (0.5 and 0.7);
\coordinate [label=center:$A''_1$] () at (12.2,1.7);
\draw (12.2,-1.7) ellipse (0.5 and 0.7);
\coordinate [label=center:$A''_l$] () at (12.2,-1.7);
\coordinate [label=center:$\ldots$] () at (9.5,0.9);
\draw (10.7,0) ellipse (0.5 and 0.7);
\coordinate [label=center:$A''_j$] () at (10.7,0);
\coordinate [label=center:$\overbrace{\phantom{aaaaa}}^{L'}$] () at (9.5,1.15);
\coordinate [label=center:$\widetilde{\Gamma}$] () at (10.7,-2.9);

\end{scope}

\end{tikzpicture}
\caption{Stable tree $\Gamma$ and stable trees $\tGamma$ such that $a(\tGamma,\Gamma)\ne 0$}
\label{fig:dilaton for A-class,3}
\end{figure}
Let 
$$
\ta:=a(\Gamma)\left/\left(\frac{r}{R}\prod_{j=1}^l\frac{r''_j}{R''_j}\right)\right..
$$
The coefficient $a(\tGamma,\Gamma)$ for the right tree in Fig.~\ref{fig:dilaton for A-class,3} is equal to $-\ta\frac{R''_j}{R}\prod_{k=1}^l\frac{r''_k}{R''_k}$. So we compute
$$
a(\Gamma)-\sum_{\tGamma\in\ST^{m-1}_{g,n+1}}a(\tGamma,\Gamma)=\ta\left(\frac{r}{R}+\sum_{j=1}^l\frac{R''_j}{R}\right)\prod_{k=1}^l\frac{r''_k}{R''_k}=\ta\prod_{k=1}^l\frac{r''_k}{R''_k}=\frac{R}{r}a(\Gamma).
$$ 
The lemma is proved.
\end{proof}

\begin{lemma}\label{lemma:psi0 times hA}
Let $a_0,\ldots,a_n$, $n\ge 1$, be integers with vanishing sum and $m\ge 2$. Then we have
\begin{multline*}
\hA^{g,m}(a_0,\ldots,a_n)-a_0\psi_0\hA^{g,m-1}(a_0,\ldots,a_n)=\\
=\sum_{\Gamma\in\ST^m_{g,n+1}}\frac{2g-1+n}{r(v_1(\Gamma))}a(\Gamma)\lambda_g\DR_\Gamma(a_0,\ldots,a_n).
\end{multline*}
\end{lemma}
\begin{proof}
The proof is analogous to the proof of the previous lemma.
\end{proof}

\begin{corollary}\label{corollary:psi times Apol}
Let $a_1,\ldots,a_n$, $n\ge 1$, be arbitrary integers and $m\ge 2$. Denote by $\pi\colon\oM_{g,n+1}\to\oM_{g,n}$ the forgetful map that forgets the first marked point. Then we have 
\begin{align}
&A^{g,m}(a_1,\ldots,a_n)-a_1\psi_1 A^{g,m-1}(a_1,\ldots,a_n)=\label{eq:psi times Apol}\\
=&\sum_{\substack{\Gamma\in\ST^m_{g,n+1}\\v(l_1(\Gamma))=v_1(\Gamma)\\g(v_1(\Gamma))\ge 1}}\frac{2g-1+n}{r(v_1(\Gamma))}\frac{a(\Gamma)}{\sum a_i}\lambda_g\pi_*\DR_\Gamma\left(-\sum a_i,a_1,\ldots,a_n\right)\notag\\
&+\hA^{g,m-1}\left(-\sum_{i=2}^n a_i,a_2,\ldots,a_n\right).\notag
\end{align}
\end{corollary}
\begin{proof}
The corollary is an elementary exercise that uses two previous lemmas and the fact that
$$
A^{g,m}(a_1,\ldots,a_n)=\frac{1}{\sum a_i}\pi_*\hA^{g,m}\left(-\sum a_i,a_1,\ldots,a_n\right).
$$
\end{proof}

\noindent{\it Proof of Proposition~\ref{proposition:dilatonA}}
Let $m:=\sum d_i-2g+3$. Let us prove that
\begin{gather}\label{eq:dilaton for A-class,derivative}
\left.\frac{\d}{\d a_{n+1}}\pi_*A^{g,m}(a_1,\ldots,a_{n+1})\right|_{a_{n+1}=0}=
\begin{cases}
0,&\text{if $m=1$},\\
\pi_*\left(\psi_{n+1} A^{g,m-1}(a_1,\ldots,a_n,0)\right),&\text{if $m\ge 2$}.
\end{cases}
\end{gather}
For $m=1$ this equation immediately follows from Lemma~\ref{lemma:divisibility}. Suppose $m\ge 2$. Let us rewrite equation~\eqref{eq:psi times Apol} in the way that is more suitable for us:
\begin{align}
&A^{g,m}(a_1,\ldots,a_{n+1})-a_{n+1}\psi_{n+1} A^{g,m-1}(a_1,\ldots,a_{n+1})=\label{eq:psi times Apol,2}\\
=&\sum_{\substack{\Gamma\in\ST^m_{g,n+2}\\v(l_{n+1}(\Gamma))=v_1(\Gamma)\\g(v_1(\Gamma))\ge 1}}\frac{2g+n}{r(v_1(\Gamma))}\frac{a(\Gamma)}{\sum a_i}\lambda_g\pi_{0*}\DR_\Gamma\left(-\sum a_i,a_1,\ldots,a_{n+1}\right)\notag\\
&+\hA^{g,m-1}\left(-\sum_{i=1}^n a_i,a_1,\ldots,a_n\right),\notag
\end{align}
where the map $\pi_0\colon\oM_{g,n+2}\to\oM_{g,n+1}$ forgets the first marked point. The last term on the right-hand side of this equation doesn't depend on $a_{n+1}$. Note also that, by Lemma~\ref{lemma:divisibility}, after applying the pushforward $\pi_*$ each term in the sum on the right-hand side of~\eqref{eq:psi times Apol,2} becomes divisible by $a_{n+1}^2$. This proves equation~\eqref{eq:dilaton for A-class,derivative}. 

Equation~\eqref{eq:dilaton for A-class,derivative} immediately implies the statement of the proposition for $m=1$. In the case $m\ge 2$ equation~\eqref{eq:dilaton for A-class,derivative} yields
\begin{multline*}
\pi_*A^g_{d_1,\ldots,d_n,1}=\pi_*\left(\psi_{n+1} A^g_{d_1,\ldots,d_n,0}\right)\stackrel{\text{by Prop.~\ref{proposition:pullback of A-class}}}{=}\pi_*\left(\psi_{n+1}\pi^* A^g_{d_1,\ldots,d_n}\right)=\\
=(2g-2+n)A^g_{d_1,\ldots,d_n}.
\end{multline*}
The proposition is proved.
\qed

\begin{proposition}\label{proposition:dilatonB}
	Denote by $\pi\colon\oM_{g,n+1}\to\oM_{g,n}$ the forgetful map that forgets the last marked point. Then we have
	\begin{gather}\label{eq:dilaton of B-class}
	\pi_* (B^g_{d_1,\ldots,d_n,1})=
	\begin{cases}
	(2g-2+n) ~ B^g_{d_1,\ldots,d_n},&\text{if $\sum d_i>2g-2$},\\
	0,&\text{if $\sum d_i = 2g-2$}.
	\end{cases}
	\end{gather}
\end{proposition}
\begin{proof}
Let $(T,q) \in \Omega^{B,g}_{d_1,\dotsc,d_n}$ be an admissible and stable complete tree with a power function
$$q \colon H^{em}_+(T) \to \NN,$$
as in Definition \ref{B-class}. We denote by $\epsilon \colon \{1, \dotsc,\deg(T)-1\} \to \NN$ the function
$$
\epsilon (k) := 2 \sum_{\substack{v \in V(T) \\ l(v) \leq k}} g(v)-\sum_{\substack{h \in H^{em}_+(T) \\ l(h)=k}} q(h)-2 
$$
measuring the distance to non-admissibility at the level $k$. As in the proof of Proposition~\ref{proposition:pullback of B-class}, we have two possible ways to add a new marking labelled by $n+1$.

First, choose a vertex $v \in V(T)$.
Let $C=(e^C_1,v^C_1,\dotsc,e^C_{\deg(T)-l(v)},v^C_{\deg(T)-l(v)},\sigma_{n+1})$ be a chain of weakly stable vertices with a new marking $\sigma_{n+1}$.
Precisely, the edge $e^C_1$ is attached to the vertex $v^C_1$, the edge $e^C_k$ links the vertex $v^C_{k-1}$ to $v^C_k$, and the leg $\sigma_{n+1}$ is attached to the vertex~$v^C_{\deg(T)-l(v)}$.
Moreover, every vertex is of genus $0$ and contains two extra legs.
We construct a tree $T_v$, obtained from $T$ by gluing the edge $e^C_1$ (and thus the chain $C$) to the vertex $v$.
We have $H^{em}_+(T) \subset H^{em}_+(T_v)$ and we extend the power function $q$ into a function
$q_v \colon H^{em}_+(T_v) \to \NN$
by taking 
$$
q_v(h^C_k):=1 \quad \textrm{and} \quad q_v(\sigma_{n+1}):=1,
$$
where $h^C_k$ is the half-edge in $H^{em}_+(T_v)$ contained in the edge $e^C_k$.
It is easy to see that we get
$$(T_v,q_v) \in \Omega^{B,g}_{d_1,\dotsc,d_n,1} \iff \forall k \in [l(v),\deg(T)-1], ~~ \epsilon(k) \geq 1.$$
In particular, when the vertex $v$ is at the maximal level $\deg(T)$, then the tree $T_v$ is always admissible.

Second, choose a half-edge $h \in H^{em}_+(T)$ attached to the vertex $v$.
We construct a tree $T_{(v,h)}$, obtained from $T$ by adding an extra level between the levels $l(v)$ and $l(v)+1$ of $T$ as follows:
\begin{itemize}
	\item[-] denote by $h_0,\dotsc,h_m \in H^{em}_+(T)$ the half-edges of level $l(v)$, with $h_0:=h$, 
	\item[-] insert a pair $(e_k,v_k)$ between the half-edge $h_k$ and the vertex it is attached to, where $e_k=(h'_k,h''_k)$ is an edge and $v_k$ is a vertex of genus $0$,
	\item[-] glue the half-edge $h^C_1$ from the chain $C$ to the vertex $v_0$,
	\item[-] add $q(h_k)+1$ extra legs to the vertex $v_k$, for $0 \leq k \leq m$.
\end{itemize}
Therefore, the number of levels of the tree $T_{(v,h)}$ is $\deg(T)+1$, the vertex $v_0 \in V(T_{(v,h)})$ is the only strongly stable vertex at its level, and we have a natural inclusion $H^{em}_+(T_v) \subset H^{em}_+(T_{(v,h)})$.
Then, we extend the power function $q_v$ into a function
$q_{(v,h)} \colon H^{em}_+(T_{(v,h)}) \to \NN$
by taking
$$q_{(v,h)}(h'_k) := q(h_k).$$
We obtain
$$(T_{(v,h)},q_{(v,h)}) \in \Omega^{B,g}_{d_1,\dotsc,d_n,1} \iff \left\lbrace \begin{array}{l}
\forall k \in [l(v),\deg(T)-1], ~~ \epsilon(k) \geq 1,  \quad \textrm{and} \\
\left( l(v) \neq \deg(T) ~~ \textrm{or} ~~ \sum_{i=1}^n d_i = 2g-2 \right).
\end{array}
\right.$$
In particular, when the vertex $v$ is at the maximal level $\deg(T)$, the tree $T_{(v,h)}$ is admissible if and only if $d_1+\dotsb+d_n = 2g-2$.

Let $l_T \in [1,\deg(T)]$ be the smallest integer such that
$$\forall k \in [l_T,\deg(T)-1], ~~ \epsilon(k) \geq 1.$$
When $d_1+\dotsb+d_n > 2g-2$ (resp.~when $d_1+\dotsb+d_n=2g-2$), the two constructions
$$(T,q,v) \mapsto (T_v,q_v) \quad \textrm{and} \quad (T,q,v,h) \mapsto (T_{(v,h)},q_{(v,h)})$$
give a bijection from the set
$$\bigsqcup_{(T,q) \in \Omega^{B,g}_{d_1,\dotsc,d_n}} \{v \in V(T) | l(v) \geq l_T \} \sqcup \{(v,h) \in V(T) \times H^{em}_+(T) | h \to v,l_T \leq l(v) < \deg(T) \}$$
(resp.~the same set with the inequality $l_T \leq l(v) \leq \deg(T)$)
to the set $\Omega^{B,g}_{d_1,\dotsc,d_n,1}$.
Furthermore, we get the contributions
\begin{align}
e_* \pi_* ([T_v,q_v])=& (2g(v)-2+n(v)+q(v)+1) e_* [T,q], \label{eqdil1}\\
e_* \pi_* ([T_{(v,h)},q_{(v,h)}])=& (q(h)+1) e_* [T,q],\label{eqdil2}
\end{align}
where $q(v)$ denotes the value of the power function $q \colon H^{em}_+(T) \to \NN$ at the (half-)edge linking the mother of the vertex $v$ to the vertex $v$, and $n(v)$ denotes the number of half-edges attached to the vertex $v$, without counting the extra legs. Thus, the total number of half-edges attached to the vertex $v$ is indeed $n(v)+q(v)+1$.

Finally, when $d_1+\dotsb+d_n>2g-2$, we get
\begin{align*}
\pi_* (B^g_{d_1,\ldots,d_n,1})=&\sum_{(T,q) \in \Omega^{B,g}_{d_1,\dotsc,d_n,1}} (-1)^{\deg(T)-1} \pi_* \ee_* [T,q]= \\
=& \sum_{(T,q) \in \Omega^{B,g}_{d_1,\dotsc,d_n}}  (-1)^{\deg(T)-1}\left( \sum_{\substack{v \in V(T) \\ l(v)=\deg(T)}}
\pi_* \ee_* [T_v,q_v]\right. \\
& + \left.\sum_{\substack{v \in V(T) \\ l_T \leq l(v) < \deg(T)}}
\left(\pi_* \ee_* [T_v,q_v] - \sum_{\substack{h \in H^{em}_+(T) \\ h \to v }} \pi_* \ee_* [T_{(v,h)}, q_{(v,h)}] \right)\right)=\\
=&\sum_{(T,q) \in \Omega^{B,g}_{d_1,\dotsc,d_n}}  (-1)^{\deg(T)-1}\left( \sum_{\substack{v \in V(T) \\ l(v)=\deg(T)}}
\ee_* \pi_* [T_v,q_v] \right. \\
& + \left.\sum_{\substack{v \in V(T) \\ l_T \leq l(v) < \deg(T)}}
\left(\ee_* \pi_* [T_v,q_v] - \sum_{\substack{h \in H^{em}_+(T) \\ h \to v }} \ee_* \pi_* [T_{(v,h)}, q_{(v,h)}] \right)\right), 
\end{align*}
where the minus sign in the second line of the second equality comes from the fact that the number of levels in the tree $T_{(v,h)}$ is $\deg(T)+1$, the third equality comes from the relation $\ee \circ \pi = \pi \circ \ee$ among the forgetful maps. Using equations \eqref{eqdil1} and \eqref{eqdil2}, we get
\begin{align*}
&\pi_*(B^g_{d_1,\ldots,d_n,1})=\\
= & \sum_{(T,q) \in \Omega^{B,g}_{d_1,\dotsc,d_n}} (-1)^{\deg(T)-1} \ee_* [T,q] \cdot
\left(\sum_{\substack{v \in V(T) \\ l(v)=\deg(T)}}
(2g(v)-2+n(v)+q(v)+1)\right. \\
& + \left. \sum_{\substack{v \in V(T) \\ l_T \leq l(v) < \deg(T)}}
\left(2g(v)-1+n(v)+q(v) - \sum_{\substack{h \in H^{em}_+(T) \\ h \to v }} (q(h)+1) \right)\right)=
\end{align*}
\begin{align*}
= & \sum_{(T,q) \in \Omega^{B,g}_{d_1,\dotsc,d_n}} (-1)^{\deg(T)-1} \ee_* [T,q] \cdot
\left(\sum_{\substack{v \in V(T) \\ l(v)=\deg(T)}}
(2g(v)+n(v)-1+q(v))\right. \\
& + \left.\sum_{\substack{v \in V(T) \\ l_T \leq l(v) < \deg(T)}}
\left( 2g(v)+q(v) - \sum_{\substack{h \in H^{em}_+(T) \\ h \to v }} q(h) \right)\right)= \\
= & \sum_{(T,q) \in \Omega^{B,g}_{d_1,\dotsc,d_n}} (-1)^{\deg(T)-1} \ee_* [T,q] \cdot
\left(2 \sum_{\substack{v \in V(T) \\ l(v) \geq l_T}} g(v) + n \right. \\
& \left. + \left(\sum_{\substack{v \in V(T) \\ l(v) \geq l_T}} q(v)
- \sum_{\substack{v \in V(T) \\ l_T \leq l(v) < \deg(T)}} \sum_{\substack{h \in H^{em}_+(T) \\ h \to v }} q(h) \right)\right)= \\
= & \sum_{(T,q) \in \Omega^{B,g}_{d_1,\dotsc,d_n}} (-1)^{\deg(T)-1} \ee_* [T,q] \cdot
\left(2 \sum_{\substack{v \in V(T) \\ l(v) \geq l_T}} g(v) + n + \sum_{\substack{v \in V(T) \\ l(v) = l_T}} q(v)
\right).
\end{align*}
We conclude using the equality $\epsilon(l_T-1)=0$:
\begin{align*}
&\pi_* (B^g_{d_1,\ldots,d_n,1})=\\
=&\sum_{(T,q) \in \Omega^{B,g}_{d_1,\dotsc,d_n}} (-1)^{\deg(T)-1} \ee_* [T,q] \cdot
\left(2 \sum_{\substack{v \in V(T) \\ l(v) \geq l_T}} g(v) + n + 2\sum_{\substack{v \in V(T) \\ l(v) < l_T}} g(v) -2\right)=\\
=&(2g-2+n) \sum_{(T,q) \in \Omega^{B,g}_{d_1,\dotsc,d_n}} (-1)^{\deg(T)-1} \ee_* [T,q]=(2g-2+n)B^g_{d_1,\ldots,d_n}.
\end{align*}
When $d_1+\dotsb+d_n=2g-2$, we have the same sequence of equalities with the additional term
\begin{align*}
&- \sum_{(T,q) \in \Omega^{B,g}_{d_1,\dotsc,d_n}} \sum_{\substack{v \in V(T) \\ l(v)=\mathrm{deg}(T)}}\sum_{\substack{h \in H^{em}_+(T) \\ h \to v }} (-1)^{\deg(T)-1} \ee_* \pi_* [T_{(v,h)},q_{(v,h)}] =\\ 	
=&-\sum_{(T,q) \in \Omega^{B,g}_{d_1,\dotsc,d_n}} (-1)^{\deg(T)-1} \ee_* [T,q] \cdot (n+d_1+\dotsb+d_n)= \\
 = & - (2g-2+n)B^g_{d_1,\ldots,d_n},
\end{align*}
coming from the fact that $(T_{(v,h)},q_{(v,h)}) \in \Omega^{B,g}_{d_1,\dotsc,d_n,1}$ when $l(v)=\deg(T)$.
\end{proof}

\subsection{Validity of the conjecture on $\cM_{g,n}$}\label{subsection:Mgn}

Let $g, n, m \geq 0$ such that $2g-2+n>0$ and denote by $\pi^{(m)} \colon \overline{\cM}_{g,n+m} \to \overline{\cM}_{g,n}$ the map forgetting the last $m$ markings. By definition, the restriction of Conjecture \ref{conjecture} to~$\cM_{g,n}$ is the following statement.

\begin{proposition}\label{proposition:Mgn}
The restriction of Conjecture \ref{conjecture} to $\cM_{g,n}$ is true.
Precisely, for every integers $d_1, \dotsc, d_{n+m} \geq 1$ such that
$$d_1+\dotsb+d_{n+m} > 2g-2,$$
we have
\begin{equation*}
\left.\left(\pi^{(m)}_*(A^g_{d_1,\dotsc,d_{n+m}})\right)\right|_{\cM_{g,n}} = \left.\left(\pi^{(m)}_*(B^g_{d_1,\dotsc,d_{n+m}})\right)\right|_{\cM_{g,n}} \in R^*(\cM_{g,n}).
\end{equation*}
\end{proposition}

\begin{proof}
Using Propositions \ref{proposition:dilatonA} and \ref{proposition:dilatonB}, we can assume that $d_{n+1}, \dotsc, d_{n+m} \geq 2$.
Furthermore, the Chow degree of the two classes in the statement is
$$\delta:=d_1+\dotsb+d_n+(d_{n+1}-1)+\dotsb+(d_{n+m}-1).$$
We get
$$\delta > 2g-2-m \quad \textrm{and} \quad \delta \geq n+m.$$
Summing these two inequalities yields
$$\delta > g+\frac{n}{2}-1 \geq g-1.$$
We conclude with the following result from \cite{Ionel}:
$$R^p(\cM_{g,n})=0, ~~ \textrm{for all } p > g-1.$$
\end{proof}


\subsection{New expression for $\lambda_g$}\label{subsection:lambdag}

Let us show that our conjectural relations~\eqref{eq:main relations} give a new formula for the class $\lambda_g\in R^g(\oM_g)$.

Let $g\ge 2$ and consider the class
$$
A^{g,1}(a_1,\ldots,a_{g-1})=\lambda_g\frac{1}{\sum a_i}\DR_g\left(\widetilde{-\sum a_i},a_1,\ldots,a_{g-1}\right).
$$
Let $\pi\colon\oM_{g,g-1}\to\oM_g$ be the forgetful map that forgets all marked points. Then, by~\eqref{eq:DR and fundamental},
$$
\pi_* A^{g,1}(a_1,\ldots,a_{g-1})=g!\lambda_g a_1^2\cdots a_{g-1}^2\sum a_i.
$$
Thus,
$$
\lambda_g=\frac{1}{g!}\pi_* A^g_{3,2,\ldots,2}\in R^g(\oM_g).
$$
So, Conjecture~\ref{conjecture} implies that
\begin{gather}\label{eq:conjecture for lambdag}
\lambda_g=\frac{1}{g!}\pi_* B^g_{3,2,\ldots,2}\in R^g(\oM_g).
\end{gather}
We can easily see that the expression on the right-hand side of this equation is a linear combination of basic tautological classes $\xi_{\Gamma*}(\gamma)$, where $\Gamma$ is a tree. No such expressions for the class~$\lambda_g$ were known before. Let us write explicitly and prove the resulting formulas in genus~$2$ and~$3$.

\subsubsection{Genus $2$}

We already wrote the expression for $B^2_3$ in~\eqref{eq:formula for B23}. Pushing it forward to $\oM_2$ and dividing by $2$, we get that Conjecture~\ref{conjecture} implies
\begin{gather}\label{eq:lambda2}
 \lambda_2 =\frac{1}{2}\kappa_2-\frac{1}{2}\tikz[baseline=-1mm]{\draw (A)--(B); \gg{1}{A}; \gg{1}{B};\lab{B}{90}{3mm}{\kappa_1};}.
\end{gather}
The relation $A^2_3=B^2_3$ is proved in Section~\ref{section:genus 2}, so formula~\eqref{eq:lambda2} is true.

\subsubsection{Genus $3$}

We compute
\begin{align*}
B^3_{3,2}=&\psi_1^3\psi_2^2
-10\,\tikz[baseline=-1mm]{\draw (A)--(B);\legm{B}{30}{1};\legm{B}{-30}{2};\gg{3}{A};\gg{0}{B};\lab{A}{25}{4mm}{\psi^4};}
-\tikz[baseline=-1mm]{\draw (A)--(B);\legm{B}{30}{1};\legm{B}{-30}{2};\gg{2}{A};\gg{1}{B};\lab{A}{28}{3.5mm}{\psi};\lab{B}{53}{3.5mm}{\psi};\lab{B}{-69}{3.8mm}{\psi^2};}
-2\,\tikz[baseline=-1mm]{\draw (A)--(B);\legm{B}{30}{1};\legm{B}{-30}{2};\gg{2}{A};\gg{1}{B};\lab{A}{28}{3.5mm}{\psi};\lab{B}{47}{4.1mm}{\psi^2};\lab{B}{-60}{3.5mm}{\psi};}
-\tikz[baseline=-1mm]{\draw (A)--(B);\legm{B}{30}{1};\legm{B}{-30}{2};\gg{2}{A};\gg{1}{B};\lab{A}{28}{3.5mm}{\psi};\lab{B}{47}{4.1mm}{\psi^3};}\\
&-\tikz[baseline=-1mm]{\draw (A)--(B);\legm{B}{30}{1};\legm{B}{-30}{2};\gg{2}{A};\gg{1}{B};\lab{A}{25}{4mm}{\psi^2};\lab{B}{-69}{3.8mm}{\psi^2};}
-3\,\tikz[baseline=-1mm]{\draw (A)--(B);\legm{B}{30}{1};\legm{B}{-30}{2};\gg{2}{A};\gg{1}{B};\lab{A}{25}{4mm}{\psi^2};\lab{B}{-60}{3.5mm}{\psi};\lab{B}{53}{3.5mm}{\psi};}
-3\,\tikz[baseline=-1mm]{\draw (A)--(B);\legm{B}{30}{1};\legm{B}{-30}{2};\gg{2}{A};\gg{1}{B};\lab{A}{25}{4mm}{\psi^2};\lab{B}{47}{4.1mm}{\psi^2};}
-\tikz[baseline=-1mm]{\draw (A)--(B);\legm{B}{0}{1};\legm{A}{-45}{2};\gg{2}{A};\gg{1}{B};\lab{B}{25}{4mm}{\psi^2};\lab{A}{-84}{3.8mm}{\psi^2};}
-\tikz[baseline=-1mm]{\draw (A)--(B);\legm{B}{30}{1};\legm{B}{-30}{2};\gg{1}{A};\gg{2}{B};\lab{B}{47}{4.1mm}{\psi^3};\lab{B}{-60}{3.5mm}{\psi};}\\
&-\tikz[baseline=-1mm]{\draw (A)--(B);\legm{B}{30}{1};\legm{B}{-30}{2};\gg{1}{A};\gg{2}{B};\lab{B}{47}{4.1mm}{\psi^2};\lab{B}{-69}{3.8mm}{\psi^2};}
+10\,\tikz[baseline=-1mm]{\draw (A)--(B)--(C);\legm{C}{30}{1};\legm{C}{-30}{2};\gg{2}{A};\gg{1}{B};\gg{0}{C};\lab{A}{28}{3.5mm}{\psi};\lab{B}{25}{4mm}{\psi^2};}
+10\,\tikz[baseline=-1mm]{\draw (A)--(B)--(C);\legm{C}{30}{1};\legm{C}{-30}{2};\gg{2}{A};\gg{1}{B};\gg{0}{C};\lab{B}{28}{3.5mm}{\psi};\lab{A}{25}{4mm}{\psi^2};}
+\tikz[baseline=-1mm]{\draw (A)--(B)--(C);\legm{C}{0}{1};\legm{B}{-45}{2};\gg{2}{A};\gg{0}{B};\gg{1}{C};\lab{A}{28}{3.5mm}{\psi};\lab{C}{25}{4mm}{\psi^2};}
\end{align*}
\begin{align*}
&+\tikz[baseline=-1mm]{\draw (A)--(B)--(C);\legm{C}{30}{1};\legm{C}{-30}{2};\gg{1}{A};\gg{1}{B};\gg{1}{C};\lab{C}{53}{3.5mm}{\psi};\lab{C}{-69}{3.8mm}{\psi^2};}
+2\,\tikz[baseline=-1mm]{\draw (A)--(B)--(C);\legm{C}{30}{1};\legm{C}{-30}{2};\gg{1}{A};\gg{1}{B};\gg{1}{C};\lab{C}{47}{4.1mm}{\psi^2};\lab{C}{-60}{3.5mm}{\psi};}
+\tikz[baseline=-1mm]{\draw (A)--(B)--(C);\legm{C}{30}{1};\legm{C}{-30}{2};\gg{1}{A};\gg{1}{B};\gg{1}{C};\lab{C}{47}{4.1mm}{\psi^3};}
+\tikz[baseline=-1mm]{\draw (A)--(B)--(C);\legm{C}{30}{1};\legm{C}{-30}{2};\gg{1}{A};\gg{1}{B};\gg{1}{C};\lab{B}{28}{3.5mm}{\psi};\lab{C}{-69}{3.8mm}{\psi^2};}\\
&+3\,\tikz[baseline=-1mm]{\draw (A)--(B)--(C);\legm{C}{30}{1};\legm{C}{-30}{2};\gg{1}{A};\gg{1}{B};\gg{1}{C};\lab{B}{28}{3.5mm}{\psi};\lab{C}{-60}{3.5mm}{\psi};\lab{C}{53}{3.5mm}{\psi};}
+3\,\tikz[baseline=-1mm]{\draw (A)--(B)--(C);\legm{C}{30}{1};\legm{C}{-30}{2};\gg{1}{A};\gg{1}{B};\gg{1}{C};\lab{B}{28}{3.5mm}{\psi};\lab{C}{47}{4.1mm}{\psi^2};}
+\tikz[baseline=-1mm]{\draw (A)--(B)--(C);\legm{C}{0}{1};\legm{B}{-45}{2};\gg{1}{A};\gg{1}{B};\gg{1}{C};\lab{C}{25}{4mm}{\psi^2};\lab{B}{-75}{3.5mm}{\psi};}
+10\,\tikz[baseline=-1mm]{\draw (A)--(B)--(C);\legm{C}{30}{1};\legm{C}{-30}{2};\gg{1}{A};\gg{2}{B};\gg{0}{C};\lab{B}{25}{4mm}{\psi^3};}\\
&-10\,\tikz[baseline=-1mm]{\draw (A)--(B)--(C)--(D);\legm{D}{30}{1};\legm{D}{-30}{2};\gg{1}{A};\gg{1}{B};\gg{1}{C};\gg{0}{D};\lab{C}{25}{4mm}{\psi^2};}
-10\,\tikz[baseline=-1mm]{\draw (A)--(B)--(C)--(D);\legm{D}{30}{1};\legm{D}{-30}{2};\gg{1}{A};\gg{1}{B};\gg{1}{C};\gg{0}{D};\lab{B}{28}{3.5mm}{\psi};\lab{C}{28}{3.5mm}{\psi};}
-\tikz[baseline=-1mm]{\draw (A)--(B)--(C)--(D);\legm{D}{0}{1};\legm{C}{-45}{2};\gg{1}{A};\gg{1}{B};\gg{0}{C};\gg{1}{D};\lab{D}{25}{4mm}{\psi^2};}.
\end{align*}
Pushing forward this expression to~$\oM_3$ and dividing it by $6$, we get that Conjecture~\ref{conjecture} implies
\begin{align}
 \lambda_3=&-\frac{3}{2}\kappa_3+\frac{1}{6}\kappa_1\kappa_2+
 \frac{2}{3}\,\tikz[baseline=-1mm]{\draw (A)--(B); \gg{2}{A}; \gg{1}{B};\lab{A}{28}{3.5mm}{\psi};\lab{B}{90}{3mm}{\kappa_1};}
 -\frac{1}{6}\tikz[baseline=-1mm]{\draw (A)--(B); \gg{2}{A}; \gg{1}{B};\lab{A}{90}{3mm}{\kappa_1};\lab{B}{90}{3mm}{\kappa_1};}
 +\frac{5}{6}\tikz[baseline=-1mm]{\draw (A)--(B); \gg{2}{A}; \gg{1}{B};\lab{A}{90}{3mm}{\kappa_2};}
 -\frac{1}{6}\tikz[baseline=-1mm]{\draw (A)--(B); \gg{2}{A}; \gg{1}{B};\lab{A}{90}{3.85mm}{\kappa_1^2};}\label{eq:lambda3}\\
 &-\frac{1}{3}\,\tikz[baseline=-1mm]{\draw (A)--(B); \draw (B)--(C); \gg{1}{A};\gg{1}{B};\gg{1}{C};\lab{C}{90}{3mm}{\kappa_1};}.\notag
\end{align}
Let us prove this equation. In~\cite{Fab90} C.~Faber proved that the whole cohomology ring of~$\oM_3$ is tautological and that it is generated by the classes
$$
\delta_0:=\tikz[baseline=-1mm]{\lp{A}{0};\gg{2}{A};}\qquad \delta_1:=\tikz[baseline=-1mm]{\draw (A)--(B);\gg{2}{A};\gg{1}{B};}\qquad\lambda_1\qquad\kappa_2.
$$
There are $13$ monomials of cohomological degree~$6$ in these classes. C.~Faber proved that $\dim R^3(\oM_3)=10$ and found $3$ relations between the 13 monomials (see \cite[page 407]{Fab90}). These relations easily imply that the following $10$ classes form a basis in $R^3(\oM_3)$:
$$
\delta_0^2\lambda_1,\,\delta_0\delta_1\lambda_1,\,\delta_0\lambda_1^2,\,\delta_0\kappa_2,\,\delta_1^3,\,\delta_1^2\lambda_1,\,\delta_1\lambda_1^2,\,\lambda_1^3,\,\delta_1\kappa_2,\,\lambda_1\kappa_2.
$$  
It is not hard to check that each of these $10$ classes has the same intersection numbers with both sides of equation~\eqref{eq:lambda3}. So, formula~\eqref{eq:lambda3} is true.


\section{Restricted set of relations}\label{section:restricted}

In this section we show that the strong DR/DZ equivalence conjecture for semisimple cohomological field theories follows from the restricted set of relations~\eqref{eq:main relations}, where $\sum d_i=2g$ and $d_i\ge 1$.

Consider an arbitrary cohomological field theory in genus $0$, $c_{0,n}\colon V^{\otimes n}\to H^*(\oM_{0,n},\mbC)$. Let~$F_0(t^*_*)$ be its potential. Suppose we have a deformation~$F(t^*_*,\eps)$ of~$F_0$ the form
$$
F=F_0+\sum_{g\ge 1}\eps^{2g}F_g,\quad F_g\in \mbC[[t^*_*]].
$$
Introduce formal power series $(w^\sol)^\alpha(x,t^*_*,\eps)$ by $(w^\sol)^\alpha:=\left.\eta^{\alpha\mu}\frac{\d^2 F}{\d t^\mu_0\d t^1_0}\right|_{t^1_0\mapsto t^1_0+x}$, and let $(w^\sol)^\alpha_n:=\d_x^n(w^\sol)^\alpha$. We will use the following notation:
$$
\<\tau_{d_1}(e_{\alpha_1})\cdots\tau_{d_n}(e_{\alpha_n})\>_g:=\left.\frac{\d^n F_g}{\d t^{\alpha_1}_{d_1}\cdots\d t^{\alpha_n}_{d_n}}\right|_{t^*_*=0}.
$$
A correlator $\<\tau_{d_1}(e_{\alpha_1})\cdots\tau_{d_n}(e_{\alpha_n})\>_g$ will be called admissible, if $\sum d_i\le 2g$.
\begin{lemma}\label{lemma:technical lemma for reconstruction}
Suppose that the following conditions are satisfied:
\begin{itemize}

\item we have the vanishing property 
\begin{gather}\label{eq:vanishing for technical}
\<\tau_{d_1}(e_{\alpha_1})\cdots\tau_{d_n}(e_{\alpha_n})\>_g=0,\quad\text{if}\quad \sum d_i\le 2g-2;
\end{gather}
\item the string and the dilaton equations hold:
\begin{align*}
&\frac{\d F}{\d t^1_0}=\sum_{n\ge 0}t^\alpha_{n+1}\frac{\d F}{\d t^\alpha_n}+\frac{1}{2}\eta_{\alpha\beta}t^\alpha_0t^\beta_0,\\
&\frac{\d F}{\d t^1_1}=\eps\frac{\d F}{\d\eps}+\sum_{n\ge 0}t^\alpha_n\frac{\d F}{\d t^\alpha_n}-2F+\eps^2\frac{N}{24};
\end{align*} 

\item for each $\mu$ there exists a differential polynomial $\Omega_{1,1;\mu,0}\in\hcA^{[0]}_w$ such that
\begin{gather}\label{eq:differential polynomial for technical}
\left.\Omega_{1,1;\mu,0}\right|_{w^\gamma_n=(w^\sol)^\gamma_n}=\left.\frac{\d^2 F}{\d t^1_1\d t^\mu_0}\right|_{t^1_0\mapsto t^1_0+x}.
\end{gather}
\end{itemize}
Then all correlators $\<\tau_{d_1}(e_{\alpha_1})\cdots\tau_{d_n}(e_{\alpha_n})\>_g$ are uniquely determined by the admissible correlators.
\end{lemma}
\begin{proof}
The topological recursion relation in genus zero implies that the primary correlators $\<\tau_0(e_{\alpha_1})\cdots\tau_0(e_{\alpha_n})\>_0$ determine all correlators in genus zero. Denote by $\cR_d$ the subspace of~$\mbC[[t^*_*]]$ defined by
$$
\cR_d:=\left\{\sum c_{\alpha_1,\ldots,\alpha_n}^{d_1,\ldots,d_n}\prod t^{\alpha_i}_{d_i}\in\mbC[[t^*_*]]\left|c_{\alpha_1,\ldots,\alpha_n}^{d_1,\ldots,d_n}=0,\text{ if $\sum d_i\le d-1$}\right.\right\}.
$$ 
From the string equation and the vanishing property~\eqref{eq:vanishing for technical} it follows that the function $(w^\sol)^\alpha_n|_{x=0}$ has the form
\begin{gather}\label{eq:property of wsol}
(w^\sol)^\alpha_n|_{x=0}=\delta^{\alpha,1}\delta_{n,1}+t^\alpha_n+r^\alpha_n+\sum_{g\ge 1}q^\alpha_{g,n}\eps^{2g},\quad r^\alpha_n\in\cR_{n+1},\quad q^\alpha_{g,n}\in\cR_{2g+n}. 
\end{gather}
Introduce a grading in the ring $\mbC[[t^*_*]]$ by $\deg t^\alpha_d:=d$ and consider the expansion
$$
q^\alpha_{g,n}=\sum_{k\ge 0}q^\alpha_{g,n,k},\quad \deg q^\alpha_{g,n,k}=2g+n+k.
$$
Note that the functions $q^\alpha_{g,n,0}$ and $q^\alpha_{g,n,1}$ are determined by the admissible correlators.

Let us show that $\frac{\d^2\Omega_{1,1;\mu,0}}{\d(w^1_x)^2}=0$. Consider a monomial $f$ of the form
\begin{gather}\label{eq:monomial f}
f=\eps^{2h}(w^1_x)^l w^{\alpha_1}_{d_1}\cdots w^{\alpha_n}_{d_n},\quad l+\sum d_i=2h,\quad (\alpha_i,d_i)\ne (1,1).
\end{gather}
Then property~\eqref{eq:property of wsol} implies that
\begin{gather}\label{eq:main property for technical}
(f|_{w^\alpha_n=(w^\sol)^\alpha_n})|_{x=0}=\eps^{2h}(t^{\alpha_1}_{d_1}\cdots t^{\alpha_n}_{d_n}+h_0)+\sum_{\substack{k\ge 1\\m\ge 0}}\eps^{2h+2k}h_{k,m},
\end{gather}
where $h_0\in\cR_{2h-l+1}$, $\deg h_{k,m}=2h+2k-l+m$ and the functions $h_{k,0}$ and $h_{k,1}$ are completely determined by the admissible correlators. Suppose now that $\frac{\d^2\Omega_{1,1;\mu,0}}{\d(w^1_x)^2}\ne 0$. Consider monomials~$f$ \eqref{eq:monomial f} with the minimal $h$ such that $l\ge 2$ and the coefficient of $f$ in the differential polynomial $\Omega_{1,1;\mu,0}$ is non-zero. Let us choose such a monomial with as big $l$ as possible. Then using equation~\eqref{eq:differential polynomial for technical} we can see that
$$
\<\tau_0(e_\mu)\prod\tau_{d_i}(e_{\alpha_i})\>_h=\frac{1}{2h-1+n}\<\tau_1(e_1)\tau_0(e_\mu)\prod\tau_{d_i}(e_{\alpha_i})\>_h\ne 0.
$$
This contradicts the vanishing property~\eqref{eq:vanishing for technical}, because $\sum d_i=2h-l\le 2h-2$. We conclude that $\frac{\d^2\Omega_{1,1;\mu,0}}{\d(w^1_x)^2}=0$.

Let us now prove that the differential polynomial $\Omega_{1,1;\mu,0}$ is completely determined by the admissible correlators. Let
$$
c_{g;\alpha_1,\ldots,\alpha_n}^{d_1,\ldots,d_n}:=\Coef_{\eps^{2g}}\frac{\d^n\Omega_{1,1;\mu,0}}{\d w^{\alpha_1}_{d_1}\cdots\d w^{\alpha_n}_{d_n}},\quad\sum d_i=2g.
$$
Let us prove by induction on~$g$ that all coefficients $c_{g;\alpha_1,\ldots,\alpha_n}^{d_1,\ldots,d_n}$ are uniquely determined by the admissible correlators. We already know it for $g=0$. Suppose $g\ge 1$. Using property~\eqref{eq:main property for technical} we see that if $(\beta_i,q_i)\ne (1,1)$ and $\sum q_i=2g-1$, then the difference 
$$
\<\tau_1(e_1)\tau_0(e_\mu)\prod\tau_{q_i}(e_{\beta_i})\>_g-c_{g;1,\beta_1,\ldots,\beta_m}^{1,q_1,\ldots,q_m}
$$ 
can be expressed in terms of the admissible correlators and the coefficients $c_{h;\gamma_1,\ldots,\gamma_i}^{r_1,\ldots,r_i}$ with $h<g$. Similarly, if $(\beta_i,q_i)\ne (1,1)$ and $\sum q_i=2g$, then the difference 
$$
\<\tau_1(e_1)\tau_0(e_\mu)\prod\tau_{q_i}(e_{\beta_i})\>_g-c_{g;\beta_1,\ldots,\beta_m}^{q_1,\ldots,q_m}
$$ 
can be expressed in terms of the admissible correlators, the coefficients $c_{h;\gamma_1,\ldots,\gamma_i}^{r_1,\ldots,r_i}$ with $h<g$ and the coefficients $c_{g;1,\rho_1,\ldots,\rho_j}^{1,s_1,\ldots,s_j}$. We conclude that the differential polynomial $\Omega_{1,1;\mu,0}$ is completely determined by the admissible correlators.

We see that the functions $(w^\sol)^\alpha$ are solutions of the following system of partial differential equations:
$$
\frac{\d w^\alpha}{\d t^1_1}=\eta^{\alpha\mu}\d_x\Omega_{1,1;\mu,0},\quad 1\le\alpha\le N.
$$
The argument from the proof of Proposition~5.2 in~\cite{BG16} shows that using this system together with the string and the dilaton equations for $F$ one can uniquely reconstruct the whole solution~$(w^\sol)^\alpha$ starting from the dispersionless part $(w^\sol)^\alpha|_{\eps=0}$. After that using the string and the dilaton equations it is easy to reconstruct the whole function $F$. The lemma is proved.
\end{proof}

\begin{proposition}
Suppose that all relations~\eqref{eq:main relations} with $\sum d_i=2g$ and $d_i\ge 1$ are true. Then the strong DR/DZ equivalence conjecture is true for any semisimple cohomological field theory.
\end{proposition}
\begin{proof}
Consider an arbitrary semisimple cohomological field theory. Propositions~\ref{proposition:pullback of A-class},~\ref{proposition:pullback of B-class}, \ref{proposition:dilatonA} and~\ref{proposition:dilatonB} imply that all relations~\eqref{eq:main relations} with $\sum d_i\le 2g$ are true. Therefore, $\<\prod\tau_{d_i}(e_{\alpha_i})\>^\DR_g=\<\prod\tau_{d_i}(e_{\alpha_i})\>_g^\red$, if $\sum d_i\le 2g$. Both potentials~$F^\DR$ and~$F^\red$ satisfy the assumptions of Lemma~\ref{lemma:technical lemma for reconstruction} (see~\cite{BDGR16a}). Therefore, the lemma implies that $F^\DR=F^\red$. So the strong DR/DZ equivalence conjecture is true.
\end{proof}

In the appendix we will prove that relations~\eqref{eq:main relations} are true, when $g=2$, $d_i\ge 1$ and $\sum d_i\le 4$. Therefore, the strong DR/DZ equivalence conjecture is true for all semisimple cohomological field theories at the approximation up to genus~$2$.


{\appendix

\section{Proof of the restricted genus $2$ relations}\label{section:genus 2}

Here we prove relations~\eqref{eq:main relations}, when $g=2$, $d_i\ge 1$ and $\sum d_i\le 4$.

\subsection{Relation $A^2_d=B^2_d$}

As we know from Section~\ref{subsection:1-point}, in order to prove that $A^2_d=B^2_d$ for any $d\ge 3$, it is sufficient to prove that $A^2_3=B^2_3$. We have
$$
B^2_3=\psi_1^3-\tikz[baseline=-1mm]{\draw (A)--(B);\leg{B}{0};\gg{1}{A};\gg{1}{B};\lab{B}{25}{4mm}{\psi^2};},
$$
and $A^2_3=\Coef_{a^4}\left(\lambda_2\DR_2(\widetilde{-a},a)\right)$. The group $H^2(\oM_{2,1},\mbQ)$ has dimension $3$ and a basis is given by~(see e.g.~\cite{Get98})
$$
\psi_1\qquad 
\delta_0:=\tikz[baseline=-1mm]{\lp{A}{180};\leg{A}{0};\gg{1}{A};}\qquad
\delta_1:=\tikz[baseline=-1mm]{\draw (A)--(B);\leg{B}{0};\gg{1}{A};\gg{1}{B};}.
$$ 
So it is sufficient to check that the intersection of the difference $A^2_3-B^2_3$ with these three classes is zero. We compute
\begin{align*}
&\int_{\DR_2(\widetilde{-a},a)}\lambda_2\psi_1=\frac{a^4}{1152}&&\Rightarrow&&\int_{\oM_{2,1}}A^2_3\psi_1=\frac{1}{1152},\\
&\int_{\DR_2(\widetilde{-a},a)}\lambda_2\delta_0=0&&\Rightarrow&&\int_{\oM_{2,1}}A^2_3\delta_0=0,\\
&\int_{\DR_2(\widetilde{-a},a)}\lambda_2\delta_1=\frac{a^4}{576}&&\Rightarrow&&\int_{\oM_{2,1}}A^2_3\delta_1=\frac{1}{576},
\end{align*}
and
\begin{align*}
&\int_{\oM_{2,1}}B^2_3\psi_1=\int_{\oM_{2,1}}\psi_1^4=\frac{1}{1152},\\
&\int_{\oM_{2,1}}B^2_3\delta_0=
\tikz[baseline=-1mm]{\lp{A}{180};\leg{A}{0};\gg{1}{A};\lab{A}{25}{4mm}{\psi^3};}
-\tikz[baseline=-1mm]{\draw (A)--(B);\lp{A}{180};\leg{B}{0};\gg{0}{A};\gg{1}{B};\lab{B}{25}{4mm}{\psi^2};}=0,\\
&\int_{\oM_{2,1}}B^2_3\delta_1=
\tikz[baseline=-1mm]{\draw (A)--(B);\leg{B}{0};\gg{1}{A};\gg{1}{B};\lab{B}{25}{4mm}{\psi^2};\lab{A}{28}{3.5mm}{\psi};}=\frac{1}{576}.
\end{align*}
Thus, $A^2_3=B^2_3$.

\subsection{Relation $A^2_{2,1}=B^2_{2,1}$}

We have
\begin{gather}\label{eq:B21}
B^2_{2,1}=\psi_1^2\psi_2
-3\,\tikz[baseline=-1mm]{\draw (A)--(B);\legm{B}{30}{1};\legm{B}{-30}{2};\gg{2}{A};\gg{0}{B};\lab{A}{25}{4mm}{\psi^2}}
-\tikz[baseline=-1mm]{\draw (A)--(B);\legm{B}{30}{1};\legm{B}{-30}{2};\gg{1}{A};\gg{1}{B};\lab{B}{53}{3.5mm}{\psi};\lab{B}{-56}{3.5mm}{\psi}}
-\tikz[baseline=-1mm]{\draw (A)--(B);\legm{B}{30}{1};\legm{B}{-30}{2};\gg{1}{A};\gg{1}{B};\lab{B}{48}{4.1mm}{\psi^2}}
+3\,\tikz[baseline=-1mm]{\draw (A)--(B)--(C);\legm{C}{30}{1};\legm{C}{-30}{2};\gg{1}{A};\gg{1}{B};\gg{0}{C};\lab{B}{28}{3.5mm}{\psi}}.
\end{gather}
In~\cite{Get98} E.~Getzler proved that $H^*(\oM_{2,2},\mbQ)=R^*(\oM_{2,2})$. Moreover, he proved that the group $R^2(\oM_{2,2})$ has dimension~$14$ with a basis given by
\begin{align}
&\delta_{22}:=\tikz[baseline=-1mm]{\draw (A)--(B);\legm{B}{30}{1};\legm{B}{-30}{2};\gg{2}{A};\gg{0}{B};\lab{A}{28}{3.5mm}{\psi};}&&
\delta_{11|}:=\tikz[baseline=-1mm]{\draw (A)--(B)--(C);\legm{B}{-135}{1};\legm{B}{-45}{2};\gg{1}{A};\gg{0}{B};\gg{1}{C};}&&
\delta_{11|1}:=\tikz[baseline=-1mm]{\draw (A)--(B)--(C);\legm{A}{-90}{1};\legm{B}{-90}{2};\gg{1}{A};\gg{0}{B};\gg{1}{C};}\label{eq:basis in M22}\\
&\delta_{11|2}:=\tikz[baseline=-1mm]{\draw (A)--(B)--(C);\legm{A}{-90}{2};\legm{B}{-90}{1};\gg{1}{A};\gg{0}{B};\gg{1}{C};}&&
\delta_{11|12}:=\tikz[baseline=-1mm]{\draw (A)--(B)--(C);\legm{C}{30}{1};\legm{C}{-30}{2};\gg{1}{A};\gg{1}{B};\gg{0}{C};}&&
\delta_{01|}:=\tikz[baseline=-1mm]{\draw (A)--(B);\legm{A}{-135}{1};\legm{A}{-45}{2};\lp{A}{90};\gg{0}{A};\gg{1}{B}}\notag\\
&\delta_{01|1}:=\tikz[baseline=-1mm]{\draw (A)--(B);\legm{A}{-90}{2};\legm{B}{-90}{1};\lp{A}{90};\gg{0}{A};\gg{1}{B}}&&
\delta_{01|2}:=\tikz[baseline=-1mm]{\draw (A)--(B);\legm{A}{-90}{1};\legm{B}{-90}{2};\lp{A}{90};\gg{0}{A};\gg{1}{B}}&&
\delta_{01|12}:=\tikz[baseline=-1mm]{\draw (A)--(B);\legm{B}{30}{1};\legm{B}{-30}{2};\lp{A}{180};\gg{0}{A};\gg{1}{B}}\notag\\
&\delta_{0|}:=\tikz[baseline=-1mm]{\doubleedge;\legm{B}{30}{1};\legm{B}{-30}{2};\gg{1}{A};\gg{0}{B}}&&
\delta_{0|1}:=\tikz[baseline=-1mm]{\doubleedge;\legm{A}{180}{1};\legm{B}{0}{2};\gg{1}{A};\gg{0}{B}}&&
\delta_{0|2}:=\tikz[baseline=-1mm]{\doubleedge;\legm{A}{180}{2};\legm{B}{0}{1};\gg{1}{A};\gg{0}{B}}\notag\\
&\delta_{0|12}:=\tikz[baseline=-1mm]{\draw (A)--(B);\lp{A}{180};\legm{B}{30}{1};\legm{B}{-30}{2};\gg{1}{A};\gg{0}{B}}&&
\delta_{00}:=\tikz[baseline=-1mm]{\lp{A}{90};\lp{A}{-90};\legm{A}{0}{1};\legm{A}{180}{2};\gg{0}{A}}&&\notag
\end{align}
We compute
\begin{align*}
&\int_{\DR_2(\widetilde{-a_1-a_2},a_1,a_2)}\lambda_2\delta_{22}=&&&&\\
=&\int_{\DR_2(\widetilde{-a_1-a_2},a_1+a_2)}\lambda_2\psi_1=\frac{(a_1+a_2)^4}{1152}&&\Rightarrow&&\int_{\oM_{2,2}}A^2_{2,1}\delta_{22}=\frac{1}{384},\\
&\int_{\DR_2(\widetilde{-a_1-a_2},a_1,a_2)}\lambda_2\delta_{11|}=0&&\Rightarrow&&\int_{\oM_{2,2}}A^2_{2,1}\delta_{11|}=0,\\
&\int_{\DR_2(\widetilde{-a_1-a_2},a_1,a_2)}\lambda_2\delta_{11|1}=\frac{a_1^2(a_1+a_2)^2}{576}&&\Rightarrow&&\int_{\oM_{2,2}}A^2_{2,1}\delta_{11|1}=\frac{1}{576},\\
&\int_{\DR_2(\widetilde{-a_1-a_2},a_1,a_2)}\lambda_2\delta_{11|2}=\frac{a_2^2(a_1+a_2)^2}{576}&&\Rightarrow&&\int_{\oM_{2,2}}A^2_{2,1}\delta_{11|2}=0,\\
&\int_{\DR_2(\widetilde{-a_1-a_2},a_1,a_2)}\lambda_2\delta_{11|12}=\frac{(a_1+a_2)^4}{576}&&\Rightarrow&&\int_{\oM_{2,2}}A^2_{2,1}\delta_{11|12}=\frac{1}{192}.
\end{align*}
Since $\left.\lambda_g\right|_{\oM_{g,n}\backslash\cM^{\ct}_{g,n}}=0$, the intersections of all remaining $9$ classes from~\eqref{eq:basis in M22} with $A^2_{2,1}$ are equal to zero. It is not hard to compute the intersections of the class $B^2_{2,1}$ with the classes from~\eqref{eq:basis in M22} ans see that they agree with what we have just computed for $A^2_{2,1}$. Thus, $A^2_{2,1}=B^2_{2,1}$.

\subsection{Relation $A^2_{1,1,1}=B^2_{1,1,1}$}

We have
\begin{align}
B^2_{1,1,1}=&\psi_1\psi_2\psi_3
-2\,\tikz[baseline=-1mm]{\draw (A)--(B);\leg{A}{-30};\leg{B}{30};\leg{B}{-30};\gg{2}{A};\gg{0}{B};\lab{A}{28}{3.5mm}{\psi};\lab{A}{-60}{3.5mm}{\psi}}
-6\,\tikz[baseline=-1mm]{\draw (A)--(B);\leg{B}{30};\leg{B}{-30};\leg{B}{0};\gg{2}{A};\gg{0}{B};\lab{A}{25}{4mm}{\psi^2}}
-2\,\tikz[baseline=-1mm]{\draw (A)--(B);\leg{B}{30};\leg{B}{0};\leg{B}{-30};\gg{2}{A};\gg{0}{B};\lab{A}{28}{3.5mm}{\psi};\lab{B}{53}{3.5mm}{\psi};}
-\tikz[baseline=-1mm]{\draw (A)--(B);\leg{B}{30};\leg{B}{0};\leg{B}{-30};\gg{1}{A};\gg{1}{B};\lab{B}{53}{3.5mm}{\psi};\lab{B}{-60}{3.5mm}{\psi};}\label{eq:B2111}\\
&+4\,\tikz[baseline=-1mm]{\draw (A)--(B)--(C);\leg{C}{-30};\leg{C}{30};\leg{B}{-30};\gg{2}{A};\gg{0}{B};\gg{0}{C};\lab{A}{28}{3.5mm}{\psi};}
+2\,\tikz[baseline=-1mm]{\draw (A)--(B)--(C);\leg{B}{-30};\leg{C}{-30};\leg{C}{30};\gg{1}{A};\gg{1}{B};\gg{0}{C};\lab{B}{-60}{3.5mm}{\psi}}
+2\,\tikz[baseline=-1mm]{\draw (A)--(B)--(C);\leg{C}{-30};\leg{C}{30};\leg{B}{-30};\gg{1}{A};\gg{1}{B};\gg{0}{C};\lab{B}{28}{3.5mm}{\psi};}\notag\\
&+6\,\tikz[baseline=-1mm]{\draw (A)--(B)--(C);\leg{C}{-30};\leg{C}{30};\leg{C}{0};\gg{1}{A};\gg{1}{B};\gg{0}{C};\lab{B}{28}{3.5mm}{\psi};}
+2\,\tikz[baseline=-1mm]{\draw (A)--(B)--(C);\leg{C}{-30};\leg{C}{30};\leg{C}{0};\gg{1}{A};\gg{1}{B};\gg{0}{C};\lab{C}{53}{3.5mm}{\psi};}
-4\,\tikz[baseline=-1mm]{\draw (A)--(B)--(C)--(D);\leg{C}{-30};\leg{D}{30};\leg{D}{-30};\gg{1}{A};\gg{1}{B};\gg{0}{C};\gg{0}{D};}.\notag
\end{align}
Introduce the following notations:
\begin{align*}
&\alpha_1:=\tikz[baseline=-1mm]{\draw (A)--(B)--(C)--(D);\leg{B}{-60};\leg{B}{-120};\leg{C}{-90};\gg{1}{A};\gg{0}{B};\gg{0}{C};\gg{1}{D}}&&
\alpha_2:=\tikz[baseline=-1mm]{\draw (E)--(A)--(B);\draw (F)--(A);\leg{A}{-30};\leg{B}{30};\leg{B}{-30};\gg{1}{E};\gg{1}{F};\gg{0}{A};\gg{0}{B}}&&
\alpha_3:=\tikz[baseline=-1mm]{\draw (E)--(A)--(B);\draw (F)--(A);\leg{F}{-30};\leg{B}{30};\leg{B}{-30};\gg{1}{E};\gg{1}{F};\gg{0}{A};\gg{0}{B}}\\
&\alpha_4:=\tikz[baseline=-1mm]{\draw (A)--(B)--(C)--(D);\leg{C}{-30};\leg{D}{30};\leg{D}{-30};\gg{1}{A};\gg{1}{B};\gg{0}{C};\gg{0}{D}}&&
\alpha_5:=\tikz[baseline=-1mm]{\draw (E)--(A)--(B);\draw (F)--(A);\leg{B}{0};\leg{B}{30};\leg{B}{-30};\gg{1}{E};\gg{1}{F};\gg{0}{A};\gg{0}{B}}&&
\alpha_6:=\tikz[baseline=-1mm]{\draw (A)--(B)--(C)--(D);\leg{B}{-30};\leg{D}{30};\leg{D}{-30};\gg{1}{A};\gg{0}{B};\gg{1}{C};\gg{0}{D}}\\
&a_{3,2,1}^{i,j,k}:=\tikz[baseline=-1mm]{\draw (A)--(B)--(C);\leg{A}{-60};\leg{A}{-120};\leg{B}{-60};\leg{B}{-120};\leg{C}{-60};\leg{C}{-120};\lp{A}{90};\gg{1}{A};\gg{0}{B};\gg{0}{C};\lab{A}{-90}{5.9mm}{...};\lab{A}{-90}{7mm}{\underbrace{\phantom{aaa}}_{\text{$i$ legs}}};\lab{B}{-90}{5.9mm}{...};\lab{B}{-90}{7mm}{\underbrace{\phantom{aaa}}_{\text{$j$ legs}}};\lab{C}{-90}{5.9mm}{...};\lab{C}{-90}{7mm}{\underbrace{\phantom{aaa}}_{\text{$k$ legs}}}}&&
a^{i,j,k}_{1,4,1}:=\tikz[baseline=-1mm]{\draw (A)--(B)--(C);\leg{A}{-60};\leg{A}{-120};\leg{B}{-60};\leg{B}{-120};\leg{C}{-60};\leg{C}{-120};\lp{B}{90};\gg{1}{A};\gg{0}{B};\gg{0}{C};\lab{A}{-90}{5.9mm}{...};\lab{A}{-90}{7mm}{\underbrace{\phantom{aaa}}_{\text{$i$ legs}}};\lab{B}{-90}{5.9mm}{...};\lab{B}{-90}{7mm}{\underbrace{\phantom{aaa}}_{\text{$j$ legs}}};\lab{C}{-90}{5.9mm}{...};\lab{C}{-90}{7mm}{\underbrace{\phantom{aaa}}_{\text{$k$ legs}}}}&&
a^{i,j,k}_{1,2,3}:=\tikz[baseline=-1mm]{\draw (A)--(B)--(C);\leg{A}{-60};\leg{A}{-120};\leg{B}{-60};\leg{B}{-120};\leg{C}{-60};\leg{C}{-120};\lp{C}{90};\gg{1}{A};\gg{0}{B};\gg{0}{C};\lab{A}{-90}{5.9mm}{...};\lab{A}{-90}{7mm}{\underbrace{\phantom{aaa}}_{\text{$i$ legs}}};\lab{B}{-90}{5.9mm}{...};\lab{B}{-90}{7mm}{\underbrace{\phantom{aaa}}_{\text{$j$ legs}}};\lab{C}{-90}{5.9mm}{...};\lab{C}{-90}{7mm}{\underbrace{\phantom{aaa}}_{\text{$k$ legs}}}}\\
&a^{i,j,k}_{2,3,1}:=\tikz[baseline=-1mm]{\draw (A)--(B)--(C);\leg{A}{-60};\leg{A}{-120};\leg{B}{-60};\leg{B}{-120};\leg{C}{-60};\leg{C}{-120};\lp{A}{90};\gg{0}{A};\gg{1}{B};\gg{0}{C};\lab{A}{-90}{5.9mm}{...};\lab{A}{-90}{7mm}{\underbrace{\phantom{aaa}}_{\text{$j$ legs}}};\lab{B}{-90}{5.9mm}{...};\lab{B}{-90}{7mm}{\underbrace{\phantom{aaa}}_{\text{$i$ legs}}};\lab{C}{-90}{5.9mm}{...};\lab{C}{-90}{7mm}{\underbrace{\phantom{aaa}}_{\text{$k$ legs}}}}&&
b_{2,3,1}^{i,j,k}:=\tikz[baseline=-1mm]{\doubleedge;\draw (B)--(C);\leg{A}{-60};\leg{A}{-120};\leg{B}{-60};\leg{B}{-120};\leg{C}{-60};\leg{C}{-120};\gg{1}{A};\gg{0}{B};\gg{0}{C};\lab{A}{-90}{5.9mm}{...};\lab{A}{-90}{7mm}{\underbrace{\phantom{aaa}}_{\text{$i$ legs}}};\lab{B}{-90}{5.9mm}{...};\lab{B}{-90}{7mm}{\underbrace{\phantom{aaa}}_{\text{$j$ legs}}};\lab{C}{-90}{5.9mm}{...};\lab{C}{-90}{7mm}{\underbrace{\phantom{aaa}}_{\text{$k$ legs}}}}&&
b_{1,3,2}^{i,j,k}:=\tikz[baseline=-1mm]{\doubleedgeB;\draw (A)--(B);\leg{A}{-60};\leg{A}{-120};\leg{B}{-60};\leg{B}{-120};\leg{C}{-60};\leg{C}{-120};\gg{1}{A};\gg{0}{B};\gg{0}{C};\lab{A}{-90}{5.9mm}{...};\lab{A}{-90}{7mm}{\underbrace{\phantom{aaa}}_{\text{$i$ legs}}};\lab{B}{-90}{5.9mm}{...};\lab{B}{-90}{7mm}{\underbrace{\phantom{aaa}}_{\text{$j$ legs}}};\lab{C}{-90}{5.9mm}{...};\lab{C}{-90}{7mm}{\underbrace{\phantom{aaa}}_{\text{$k$ legs}}}}\\
&b_{3,2,1}^{i,j,k}:=\tikz[baseline=-1mm]{\doubleedgeB;\draw (A)--(B);\leg{A}{-60};\leg{A}{-120};\leg{B}{-60};\leg{B}{-120};\leg{C}{-60};\leg{C}{-120};\gg{0}{A};\gg{1}{B};\gg{0}{C};\lab{A}{-90}{5.9mm}{...};\lab{A}{-90}{7mm}{\underbrace{\phantom{aaa}}_{\text{$k$ legs}}};\lab{B}{-90}{5.9mm}{...};\lab{B}{-90}{7mm}{\underbrace{\phantom{aaa}}_{\text{$i$ legs}}};\lab{C}{-90}{5.9mm}{...};\lab{C}{-90}{7mm}{\underbrace{\phantom{aaa}}_{\text{$j$ legs}}}}&&
c_{2,2,2}^{1,1,1}:=\tikz[baseline=-1mm]{\draw (G)--(E)--(F)--(G);\leg{G}{180};\leg{E}{0};\leg{F}{0};\gg{1}{G};\gg{0}{E};\gg{0}{F};}&&
c_{2,2,2}^{0,2,1}:=\tikz[baseline=-1mm]{\draw (G)--(E)--(F)--(G);\leg{E}{30};\leg{E}{-30};\leg{F}{0};\gg{1}{G};\gg{0}{E};\gg{0}{F};}
\end{align*}
Denote by $L'$ the subspace of $R^3(\oM_{2,3})$ spanned by boundary strata $\xi_{\Gamma*}(1)$, where the first Betti number of a stable graph $\Gamma$ is equal to~$2$. The symmetric group $S_3$ acts on $\oM_{2,3}$ by permutations of marked points. This action induces an action on $R^*(\oM_{2,3})$. Define a map $\mathrm{Sym}\colon R^*(\oM_{2,3})\to R^*(\oM_{2,3})$ by
$$
\mathrm{Sym}(\alpha):=\frac{1}{3!}\sum_{\sigma\in S_3}\sigma\alpha,\quad\alpha\in R^*(\oM_{2,3}).
$$
Let $L:=\mathrm{Sym}(L')\subset R^3(\oM_{2,3})^{S_3}$. For two classes $\alpha,\beta\in R^3(\oM_{2,3})$ we will write $\alpha\stackrel{\mathrm{mod}\,L}{=}\beta$, if $\alpha-\beta\in L$. Using the formulas for $\psi_1^2\in R^2(\oM_{2,1})$ and $\psi_1\psi_2\in R^2(\oM_{2,2})$ from~\cite{Get98} and also the topological recursion relations in genus~$0$ and~$1$, after a long computation we obtain
\begin{align*}
B^2_{1,1,1}\stackrel{\mathrm{mod}\,L}{=}&-\frac{3 \alpha _1}{5}+\frac{3 \alpha _2}{10}+\frac{4\alpha _3}{5}+\frac{4 \alpha _4}{15}+\frac{\alpha _5}{10}-\frac{4 \alpha _6}{5}+\frac{1}{360} a_{1,2,3}^{0,1,2}-\frac{1}{90} a_{1,2,3}^{0,2,1}\\
&-\frac{1}{24} a_{1,2,3}^{0,3,0}+\frac{1}{40}a_{1,2,3}^{1,1,1}+\frac{1}{180} a_{1,2,3}^{1,2,0}-\frac{1}{120} a_{1,2,3}^{2,1,0}-\frac{5}{48} a_{1,4,1}^{0,0,3}-\frac{1}{80}a_{1,4,1}^{0,1,2}\\
&-\frac{1}{144} a_{1,4,1}^{1,0,2}+\frac{1}{80} a_{2,3,1}^{0,0,3}-\frac{7}{240} a_{2,3,1}^{0,1,2}+\frac{1}{240}a_{2,3,1}^{1,0,2}+\frac{1}{90} a_{3,2,1}^{0,1,2}+\frac{7}{90} b_{1,3,2}^{0,1,2}\\
&+\frac{1}{45} b_{1,3,2}^{0,2,1}+\frac{1}{90}b_{1,3,2}^{1,1,1}-\frac{11}{30} b_{2,3,1}^{0,0,3}-\frac{1}{10} b_{2,3,1}^{0,1,2}-\frac{1}{90} b_{2,3,1}^{1,0,2}-\frac{1}{30}b_{3,2,1}^{0,1,2}\\
&+\frac{1}{5} c_{2,2,2}^{0,2,1}+\frac{2}{45} c_{2,2,2}^{1,1,1}.
\end{align*}
On the other hand, a direct computation using Hain's formula gives
\begin{align*}
&\left.\Coef_{a_1a_2a_3}\left(\frac{1}{a_1+a_2+a_3}\DR_2(\widetilde{-a_1-a_2-a_3},a_1,a_2,a_3)\right)\right|_{\cM_{2,3}^\ct}=\\
=&(\psi_1+\psi_2+\psi_3)
-6\,\tikz[baseline=-1mm]{\draw (A)--(B);\leg{B}{30};\leg{B}{0};\leg{B}{-30};\gg{2}{A};\gg{0}{B};}
-3\,\tikz[baseline=-1mm]{\draw (A)--(B);\leg{A}{-30};\leg{B}{30};\leg{B}{-30};\gg{2}{A};\gg{0}{B};}
-\frac{6}{5}\,\tikz[baseline=-1mm]{\draw (A)--(B);\leg{B}{30};\leg{B}{0};\leg{B}{-30};\gg{1}{A};\gg{1}{B};}
-\frac{1}{5}\,\tikz[baseline=-1mm]{\draw (A)--(B);\leg{A}{-30};\leg{B}{30};\leg{B}{-30};\gg{1}{A};\gg{1}{B};}.
\end{align*}
In this computation one should use that (see e.g.~\cite{Get98})
$$
\kappa_1=
\frac{7}{10}\,\tikz[baseline=-1mm]{\draw (A)--(B);\gg{1}{A};\gg{1}{B};}
+\frac{1}{10}\,\tikz[baseline=-1mm]{\lp{A}{90}\gg{1}{A};}\in R^1(\oM_2).
$$
Using the formula (\cite{JPPZ16})
$$
\lambda_2=
\frac{1}{960}\,\tikz[baseline=-1mm]{\lp{A}{0};\lp{A}{180};\gg{0}{A};}
+\frac{1}{240}\,\tikz[baseline=-1mm]{\draw (A)--(B);\lp{B}{0};\gg{1}{A};\gg{0}{B}}\in R^2(\oM_2),
$$
we obtain the following formula for the class $A^2_{1,1,1}$:
\begin{align*}
A^2_{1,1,1}\stackrel{\mathrm{mod}\,L}{=}&\frac{1}{120} a_{1,2,3}^{0,1,2}+\frac{1}{60} a_{1,2,3}^{0,2,1}+\frac{1}{40} a_{1,2,3}^{0,3,0}+\frac{1}{120}a_{1,2,3}^{1,1,1}+\frac{1}{60} a_{1,2,3}^{1,2,0}+\frac{1}{120} a_{1,2,3}^{2,1,0}\\
&-\frac{1}{80} a_{1,4,1}^{0,0,3}-\frac{1}{240}a_{1,4,1}^{0,1,2}-\frac{1}{240} a_{1,4,1}^{1,0,2}-\frac{1}{80} a_{2,3,1}^{0,0,3}-\frac{1}{240} a_{2,3,1}^{0,1,2}-\frac{1}{240}a_{2,3,1}^{1,0,2}.
\end{align*}
Thus,
\begin{align*}
B^2_{1,1,1}-A^2_{1,1,1}\stackrel{\mathrm{mod}\,L}{=}&-\frac{3 \alpha _1}{5}+\frac{3 \alpha _2}{10}+\frac{4
   \alpha _3}{5}+\frac{4 \alpha _4}{15}+\frac{\alpha _5}{10}-\frac{4 \alpha _6}{5}-\frac{1}{180} a_{1,2,3}^{0,1,2}\\
   &-\frac{1}{36} a_{1,2,3}^{0,2,1}-\frac{1}{15} a_{1,2,3}^{0,3,0}+\frac{1}{60}a_{1,2,3}^{1,1,1}-\frac{1}{90} a_{1,2,3}^{1,2,0}-\frac{1}{60} a_{1,2,3}^{2,1,0}\\
   &-\frac{11}{120} a_{1,4,1}^{0,0,3}-\frac{1}{120}a_{1,4,1}^{0,1,2}-\frac{1}{360} a_{1,4,1}^{1,0,2}+\frac{1}{40} a_{2,3,1}^{0,0,3}-\frac{1}{40} a_{2,3,1}^{0,1,2}\\
   &+\frac{1}{120}a_{2,3,1}^{1,0,2}+\frac{1}{90} a_{3,2,1}^{0,1,2}+\frac{7}{90} b_{1,3,2}^{0,1,2}+\frac{1}{45} b_{1,3,2}^{0,2,1}+\frac{1}{90}b_{1,3,2}^{1,1,1}\\
   &-\frac{11}{30} b_{2,3,1}^{0,0,3}-\frac{1}{10}b_{2,3,1}^{0,1,2}-\frac{1}{90} b_{2,3,1}^{1,0,2}-\frac{1}{30}b_{3,2,1}^{0,1,2}+\frac{1}{5} c_{2,2,2}^{0,2,1}+\frac{2}{45} c_{2,2,2}^{1,1,1}.
\end{align*}

The famous Getzler relation~\cite{Get97} says that 
\begin{align}
\gamma:=&\tikz[baseline=-1mm]{\draw (A)--(B)--(C);\leg{A}{-60};\leg{A}{-120};\leg{C}{-60};\leg{C}{-120};\gg{0}{A};\gg{1}{B};\gg{0}{C}}
-\frac{1}{3}\,\tikz[baseline=-1mm]{\draw (A)--(B)--(C);\leg{A}{-90};\leg{B}{-90};\leg{C}{-60};\leg{C}{-120};\gg{1}{A};\gg{0}{B};\gg{0}{C}}
-\frac{1}{6}\,\tikz[baseline=-1mm]{\draw (A)--(B)--(C);\leg{B}{-60};\leg{B}{-120};\leg{C}{-60};\leg{C}{-120};\gg{1}{A};\gg{0}{B};\gg{0}{C}}
+\frac{1}{2}\,\tikz[baseline=-1mm]{\draw (A)--(B)--(C);\leg{B}{-90};\leg{C}{-90};\leg{C}{-60};\leg{C}{-120};\gg{1}{A};\gg{0}{B};\gg{0}{C}}+\label{eq:Getzler relation}\\
&+\frac{1}{24}\,\tikz[baseline=-1mm]{\draw (A)--(B);\lp{A}{90};\leg{A}{-90};\leg{B}{-120};\leg{B}{-90};\leg{B}{-60};\gg{0}{A};\gg{0}{B}}
+\frac{1}{24}\,\tikz[baseline=-1mm]{\draw (A)--(B);\lp{A}{90};\leg{B}{-60};\leg{B}{-80};\leg{B}{-100};\leg{B}{-120};\gg{0}{A};\gg{0}{B}}
-\frac{1}{12}\,\tikz[baseline=-1mm]{\doubleedge;\leg{A}{-60};\leg{A}{-120};\leg{B}{-60};\leg{B}{-120};\gg{0}{A};\gg{0}{B}}=\notag\\
=&0\in R^2(\oM_{1,4}).\notag
\end{align}
We will adopt the following notation. Suppose $g_1,g_2\ge 0$ and let $i_1,\ldots,i_k$ and $j_1,\ldots,j_l$ be two lists of integers such that $\{i_1,\ldots,i_k,j_1,\ldots,j_l\}=\{1,2,\ldots,k+l\}$. Consider the moduli spaces~$\oM_{g_1,k+1}$ and~$\oM_{g_2,l+1}$, but let us label the marked points on curves from~$\oM_{g_1,k+1}$ and~$\oM_{g_2,l+1}$ by the numbers~$i_1,\ldots,i_k,k+l+1$ and $j_1,\ldots,j_l,k+l+2$, respectively. Denote by
$$
\gl^{g_1|g_2}_{i_1,\ldots,i_k|j_1,\ldots,j_l}\colon\oM_{g_1,k+1}\times\oM_{g_2,l+1}\to\oM_{g_1+g_2,k+l}
$$
the gluing map that glues the marked points labeled by~$k+l+1$ and~$k+l+2$. From Getzler's relation we obtain
\begin{align}
(\gl^{1|1}_{|1,2,3,4})_*([\oM_{1,1}]\times\gamma)=&\frac{\alpha _1}{3}-\frac{\alpha _2}{6}-\alpha _3-\frac{\alpha _4}{3}+\frac{\alpha _5}{2}+\alpha _6+\frac{1}{24} a_{1,2,3}^{0,2,1}\label{eq:get1}\\
&+\frac{1}{24} a_{1,2,3}^{0,3,0}+\frac{1}{24} a_{1,4,1}^{0,0,3}-\frac{1}{12}b_{1,3,2}^{0,1,2}=0\in R^3(\oM_{2,3}).\notag
\end{align}
Notice that the WDVV relation on $\oM_{0,5}$ implies that $-\frac{\alpha _1}{3}+\frac{\alpha _2}{6}+\frac{\alpha_5}{2}=0$. Using this observation and expressing the class $\alpha_6$ via formula~\eqref{eq:get1}, we get
\begin{align}
B^2_{1,1,1}-A^2_{1,1,1}\stackrel{\mathrm{mod}\,L}{=}&-\frac{1}{180} a_{1,2,3}^{0,1,2}+\frac{1}{180} a_{1,2,3}^{0,2,1}-\frac{1}{30} a_{1,2,3}^{0,3,0}+\frac{1}{60}a_{1,2,3}^{1,1,1}-\frac{1}{90} a_{1,2,3}^{1,2,0}\label{eq:a111b111,1}\\
&-\frac{1}{60} a_{1,2,3}^{2,1,0}-\frac{7}{120} a_{1,4,1}^{0,0,3}-\frac{1}{120}a_{1,4,1}^{0,1,2}-\frac{1}{360} a_{1,4,1}^{1,0,2}+\frac{1}{40} a_{2,3,1}^{0,0,3}\notag\\
&-\frac{1}{40} a_{2,3,1}^{0,1,2}+\frac{1}{120}a_{2,3,1}^{1,0,2}+\frac{1}{90} a_{3,2,1}^{0,1,2}+\frac{1}{90} b_{1,3,2}^{0,1,2}+\frac{1}{45} b_{1,3,2}^{0,2,1}\notag\\
&+\frac{1}{90}b_{1,3,2}^{1,1,1}-\frac{11}{30} b_{2,3,1}^{0,0,3}-\frac{1}{10} b_{2,3,1}^{0,1,2}-\frac{1}{90} b_{2,3,1}^{1,0,2}-\frac{1}{30}b_{3,2,1}^{0,1,2}\notag\\
&+\frac{1}{5} c_{2,2,2}^{0,2,1}+\frac{2}{45} c_{2,2,2}^{1,1,1}.\notag
\end{align}
Let $\pi\colon\oM_{1,5}\to\oM_{1,4}$ be the forgetful map that forgets the fifth marked point and $\gl_{1,5}\colon\oM_{1,5}\to\oM_{2,3}$ be the gluing map that glues the first and the fifth marked points. Then from Getzler's relation~\eqref{eq:Getzler relation} we obtain
\begin{align*}
\gl_{1,5*}(\pi^*\gamma)\stackrel{\mathrm{mod}\,L}{=}&\frac{1}{2} a_{1,2,3}^{0,1,2}-\frac{1}{6} a_{1,2,3}^{0,2,1}-\frac{1}{3} a_{1,2,3}^{1,1,1}+\frac{1}{2} a_{1,4,1}^{0,0,3}-\frac{1}{6}a_{1,4,1}^{0,1,2}-\frac{1}{3} a_{1,4,1}^{1,0,2}+a_{2,3,1}^{0,1,2}\\
&-\frac{1}{3} a_{3,2,1}^{0,1,2}+\frac{1}{2}b_{1,3,2}^{0,0,3}+\frac{1}{3} b_{1,3,2}^{0,1,2}-\frac{1}{6} b_{1,3,2}^{0,2,1}-\frac{1}{3} b_{1,3,2}^{1,0,2}-\frac{1}{3}b_{1,3,2}^{1,1,1}\\
&+\frac{1}{2} b_{2,3,1}^{0,0,3}-\frac{1}{2} b_{2,3,1}^{0,1,2}-\frac{1}{3}b_{2,3,1}^{1,0,2}+b_{3,2,1}^{0,1,2}+c_{2,2,2}^{0,2,1}-\frac{2}{3} c_{2,2,2}^{1,1,1}\in L.
\end{align*}
We can obtain another consequence from~\eqref{eq:Getzler relation}. Let $\gl_{1,2}\colon\oM_{1,5}\to\oM_{2,3}$ the the gluing map that glues the first two marked points. Then Getzler's relation implies that
\begin{align*}
\mathrm{Sym}(\gl_{1,2*}(\pi^*\gamma))\stackrel{\mathrm{mod}\,L}{=}&\frac{1}{3} a_{1,2,3}^{0,1,2}+\frac{5}{18} a_{1,2,3}^{0,2,1}-\frac{1}{6} a_{1,2,3}^{0,3,0}+\frac{1}{18}a_{1,2,3}^{1,1,1}-\frac{5}{18} a_{1,2,3}^{1,2,0}-\frac{2}{9} a_{1,2,3}^{2,1,0}\\
&-\frac{1}{6} a_{1,4,1}^{0,0,3}-\frac{1}{18}a_{1,4,1}^{0,1,2}-\frac{1}{18} a_{1,4,1}^{1,0,2}+a_{2,3,1}^{0,0,3}+\frac{1}{3} a_{2,3,1}^{0,1,2}+\frac{1}{3}
   a_{2,3,1}^{1,0,2}\\
   &+b_{1,3,2}^{0,0,3}+\frac{1}{9} b_{1,3,2}^{0,1,2}-\frac{2}{9} b_{1,3,2}^{0,2,1}-\frac{1}{9}b_{1,3,2}^{1,0,2}-\frac{1}{3} b_{1,3,2}^{1,1,1}-\frac{4}{9} b_{1,3,2}^{2,0,1}\\
   &-\frac{2}{3} b_{2,3,1}^{0,0,3}-\frac{2}{9}b_{2,3,1}^{0,1,2}-\frac{2}{9} b_{2,3,1}^{1,0,2}+\frac{4}{9} c_{2,2,2}^{0,2,1}+\frac{2}{9} c_{2,2,2}^{1,1,1}\in L.
\end{align*}
Adding $\frac{1}{30}\gl_{1,5*}(\pi^*\gamma)-\frac{1}{40}\mathrm{Sym}(\gl_{1,2*}(\pi^*\gamma))$ to the right-hand side of~\eqref{eq:a111b111,1}, we get
\begin{align}
B^2_{1,1,1}-A^2_{1,1,1}\stackrel{\mathrm{mod}\,L}{=}&\frac{1}{360} a_{1,2,3}^{0,1,2}-\frac{1}{144} a_{1,2,3}^{0,2,1}-\frac{7}{240} a_{1,2,3}^{0,3,0}+\frac{1}{240}a_{1,2,3}^{1,1,1}-\frac{1}{240} a_{1,2,3}^{1,2,0}\label{eq:a111b111,2}\\
&-\frac{1}{90} a_{1,2,3}^{2,1,0}-\frac{3}{80} a_{1,4,1}^{0,0,3}-\frac{1}{80}a_{1,4,1}^{0,1,2}-\frac{1}{80} a_{1,4,1}^{1,0,2}-\frac{1}{120} b_{1,3,2}^{0,0,3}\notag\\
&+\frac{7}{360} b_{1,3,2}^{0,1,2}+\frac{1}{45}b_{1,3,2}^{0,2,1}-\frac{1}{120} b_{1,3,2}^{1,0,2}+\frac{1}{120} b_{1,3,2}^{1,1,1}+\frac{1}{90} b_{1,3,2}^{2,0,1}\notag\\
&-\frac{1}{3}b_{2,3,1}^{0,0,3}-\frac{1}{9} b_{2,3,1}^{0,1,2}-\frac{1}{60} b_{2,3,1}^{1,0,2}+\frac{2}{9} c_{2,2,2}^{0,2,1}+\frac{1}{60}c_{2,2,2}^{1,1,1}.\notag
\end{align}
The WDVV relations on $\oM_{0,4}$, $\oM_{0,5}$ and $\oM_{0,6}$ imply that
\begin{align*}
c_{2,2,2}^{1,1,1}=&b_{2,3,1}^{1,0,2},\\
c_{2,2,2}^{0,2,1}=&\frac{3}{2} b_{2,3,1}^{0,0,3}+\frac{1}{2} b_{2,3,1}^{0,1,2},\\
a_{1,2,3}^{2,1,0}=&b_{1,3,2}^{2,0,1},\\
b_{1,3,2}^{1,0,2}=&\frac{1}{2} a_{1,2,3}^{1,1,1}+a_{1,2,3}^{1,2,0}-\frac{1}{2} b_{1,3,2}^{1,1,1},
\end{align*}
\begin{align*}
b_{1,3,2}^{1,1,1}=&a_{1,2,3}^{1,2,0}+a_{1,4,1}^{1,0,2},\\
b_{1,3,2}^{0,0,3}=&\frac{1}{3} a_{1,2,3}^{0,1,2}+\frac{2}{3} a_{1,2,3}^{0,2,1}+a_{1,2,3}^{0,3,0}-\frac{2}{3}b_{1,3,2}^{0,1,2}-\frac{1}{3} b_{1,3,2}^{0,2,1},\\
b_{1,3,2}^{0,2,1}=&\frac{1}{2} a_{1,2,3}^{0,2,1}+\frac{3}{2} a_{1,2,3}^{0,3,0}+\frac{3}{2} a_{1,4,1}^{0,0,3}+\frac{1}{2}a_{1,4,1}^{0,1,2}-b_{1,3,2}^{0,1,2}.
\end{align*}
Using these relations, one can easily check that the right-hand side of~\eqref{eq:a111b111,2} is zero. We conclude that $B^2_{1,1,1}-A^2_{1,1,1}\in L$.

It is easy to see that the space $L$ is spanned by the following classes:
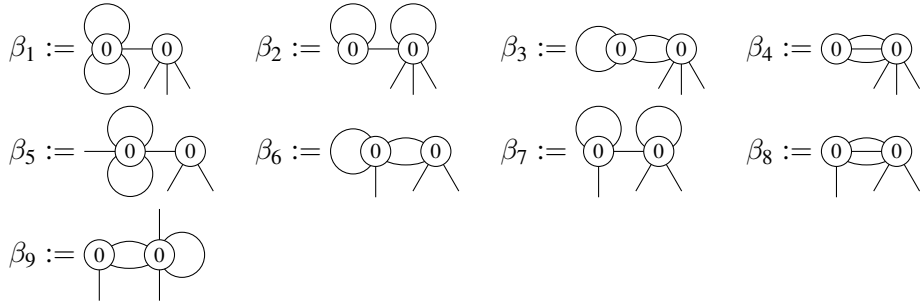
\begin{align*}
&\beta_1:=\tikz[baseline=-1mm]{\draw (A)--(B);\lp{A}{90};\lp{A}{-90};\leg{B}{-60};\leg{B}{-120};\leg{B}{-90};\gg{0}{A};\gg{0}{B}}&&
\beta_2:=\tikz[baseline=-1mm]{\draw (A)--(B);\lp{A}{90};\lp{B}{90};\leg{B}{-60};\leg{B}{-120};\leg{B}{-90};\gg{0}{A};\gg{0}{B}}&&
\beta_3:=\tikz[baseline=-1mm]{\doubleedge;\lp{A}{180};\leg{B}{-60};\leg{B}{-120};\leg{B}{-90};\gg{0}{A};\gg{0}{B}}&&
\beta_4:=\tikz[baseline=-1mm]{\doubleedge;\draw (A)--(B);\leg{B}{-60};\leg{B}{-120};\leg{B}{-90};\gg{0}{A};\gg{0}{B}}\\
&\beta_5:=\tikz[baseline=-1mm]{\draw (A)--(B);\lp{A}{90};\lp{A}{-90};\leg{A}{180};\leg{B}{-120};\leg{B}{-60};\gg{0}{A};\gg{0}{B}}&&
\beta_6:=\tikz[baseline=-1mm]{\doubleedge;\lp{A}{180};\leg{A}{-90};\leg{B}{-120};\leg{B}{-60};\gg{0}{A};\gg{0}{B}}&&
\beta_7:=\tikz[baseline=-1mm]{\draw (A)--(B);\lp{A}{90};\lp{B}{90};\leg{A}{-90};\leg{B}{-120};\leg{B}{-60};\gg{0}{A};\gg{0}{B}}&&
\beta_8:=\tikz[baseline=-1mm]{\doubleedge;\draw (A)--(B);\leg{A}{-90};\leg{B}{-120};\leg{B}{-60};\gg{0}{A};\gg{0}{B}}\\
&\beta_9:=\tikz[baseline=-1mm]{\doubleedge;\lp{B}{0};\leg{A}{-90};\leg{B}{90};\leg{B}{-90};\gg{0}{A};\gg{0}{B}}&&&&&&
\end{align*}
The WDVV relations on $\oM_{0,7}$ give the following relations:
\begin{align*}
\beta_2-\beta_4+\beta_7-\beta_8=&0,\\
\beta_1+\beta_2+2\beta_3+\frac{1}{3}\beta_5+\frac{1}{3}\beta_7-\frac{4}{3}\beta_8-\frac{2}{3}\beta_9=&0.
\end{align*}
Therefore, $\dim L\le 7$. On the other hand, in Figure~\ref{fig:intersection matrix} we compute the intersection matrix of the classes $\beta_1,\ldots,\beta_9$ with the following seven classes: $\psi_1^3,\psi_1^2\psi_2,\psi_1\psi_2\psi_3,\kappa_3$, $\kappa_1\kappa_2,\psi_1\kappa_2,\psi_1^2\delta$; where
$$
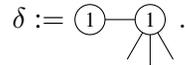
\delta:=\tikz[baseline=-1mm]{\draw (A)--(B);\leg{B}{-120};\leg{B}{-90};\leg{B}{-60};\gg{1}{A};\gg{1}{B}}\,.
$$
\begin{figure}[t]
\begin{tabular}{c|c|c|c|c|c|c|c|c|c}
& $\beta_1$ & $\beta_2$ & $\beta_3$ & $\beta_4$ & $\beta_5$ & $\beta_6$ & $\beta_7$ & $\beta_8$ & $\beta_9$ \\
\hline
$\psi_1^3$ & $0$ & $1$ & $0$ & $1$ & $1$ & $0$ & $0$ & $0$ & $2$ \\
\hline
$\psi_1^2\psi_2$ & $0$ & $3$ & $0$ & $3$ & $0$ & $1$ & $1$ & $1$ & $3$ \\
\hline
$\psi_1\psi_2\psi_3$ & $0$ & $6$ & $0$ & $6$ & $0$ & $0$ & $6$ & $6$ & $0$ \\
\hline
$\kappa_3$ & $0$ & $1$ & $0$ & $1$ & $3$ & $0$ & $0$ & $0$ & $3$ \\
\hline
$\kappa_1\kappa_2$ & $1$ & $9$ & $1$ & $9$ & $27$ & $3$ & $3$ & $3$ & $27$ \\
\hline
$\psi_1\kappa_2$ & $1$ & $4$ & $0$ & $4$ & $4$ & $2$ & $1$ & $1$ & $8$ \\
\hline
$\psi_1^2\delta$ & $0$ & $-2$ & $1$ & $0$ & $2$ & $0$ & $2$ & $0$ & $2$
\end{tabular}
\caption{Intersection matrix of $\beta_1,\ldots,\beta_9$ with $\psi_1^3,\psi_1^2\psi_2,\psi_1\psi_2\psi_3,\kappa_3,\kappa_1\kappa_2,\psi_1\kappa_2,\psi_1^2\delta$}
\label{fig:intersection matrix}
\end{figure}
This matrix is non-degenerate, so $\dim L=7$. Thus, is order to prove that $A^2_{1,1,1}=B^2_{1,1,1}$ it is sufficient to check that the intersections of $A^2_{1,1,1}-B^2_{1,1,1}$ with the classes $\psi_1^3,\psi_1^2\psi_2,\psi_1\psi_2\psi_3$, $\kappa_3,\kappa_1\kappa_2,\psi_1\kappa_2,\psi_1^2\delta$ are zero. This is a simple direct computation. The relation $A^2_{1,1,1}=B^2_{1,1,1}$ is proved.

\subsection{Relations $A^2_{3,1}=B^2_{3,1}$ and $A^2_{2,2}=B^2_{2,2}$}

Suppose $g,n\ge 1$ and $a_1,\ldots,a_n\in\mbZ$. Let $a:=\sum a_i$. The following formula is the particular case of Corollary~\ref{corollary:psi times Apol} when $m=2$:
\begin{align}
&A^{g,2}(a_1,\ldots,a_n)-a_1\psi_1 A^{g,1}(a_1,\ldots,a_n)=\label{eq:psi and A-class}\\
=&\lambda_g\DR_g\left(a_1-a,a_2,\ldots,a_n\right)\notag\\
&+\lambda_g\sum_{\substack{g_1\ge 1,\,g_2\ge 0\\g_1+g_2=g}}\sum_{\substack{I\sqcup J=\{1,\ldots,n\}\\1\in I\\2g_2-1+|J|>0}}\frac{a_J}{a}\DR_{g_1}\left(\widetilde{-a},A_I,a_J\right)\boxtimes_1\DR_{g_2}\left(A_J,-a_J\right).\notag
\end{align}

Let us prove now that $A^2_{d_1,d_2}=B^2_{d_1,d_2}$, where $(d_1,d_2)=(3,1)$ or $(d_1,d_2)=(2,2)$. By equation~\eqref{eq:psi and A-class}, we have
\begin{align*}
&A^2_{d_1,d_2}-\psi_1 A^2_{d_1-1,d_2}=\\
=&\Coef_{a_1^{d_1} a_2^{d_2}}\left(\frac{1}{a_1+a_2}\lambda_2\DR_1(\widetilde{-a_1-a_2},a_1,a_2)\boxtimes_1\DR_1(a_2,-a_2)\right)=\\
=&\Coef_{a_1^{d_1}a_2^{d_2}}\left(\frac{(a_1+a_2)a_2^3}{576}\tikz[baseline=-1mm]{\draw (A)--(B)--(C);\lp{A}{180};\lp{C}{0};\legm{A}{-30}{1};\legm{B}{-30}{2};\gg{0}{A};\gg{0}{B};\gg{0}{C};}\right)=0.
\end{align*}
On the other hand, it is easy to compute that
\begin{align*}
B^2_{3,1}=&\psi_1^3\psi_2
-\tikz[baseline=-1mm]{\draw (A)--(B);\legm{B}{30}{1};\legm{B}{-30}{2};\gg{1}{A};\gg{1}{B};\lab{B}{48}{4.1mm}{\psi^2};\lab{B}{-60}{3.5mm}{\psi};}
-\tikz[baseline=-1mm]{\draw (A)--(B);\legm{B}{30}{1};\legm{B}{-30}{2};\gg{1}{A};\gg{1}{B};\lab{B}{48}{4.1mm}{\psi^3};},\\
B^2_{2,2}=&\psi_1^2\psi_2^2
-\tikz[baseline=-1mm]{\draw (A)--(B);\legm{B}{30}{1};\legm{B}{-30}{2};\gg{1}{A};\gg{1}{B};\lab{B}{48}{4.1mm}{\psi^2};\lab{B}{-60}{3.5mm}{\psi};}
-\tikz[baseline=-1mm]{\draw (A)--(B);\legm{B}{30}{1};\legm{B}{-30}{2};\gg{1}{A};\gg{1}{B};\lab{B}{53}{3.5mm}{\psi};\lab{B}{-67}{3.8mm}{\psi^2};}.
\end{align*}
Comparing these expressions with formula~\eqref{eq:B21}, we can easily see that $B^2_{3,1}=\psi_1 B^2_{2,1}$ and $B^2_{2,2}=\psi_1 B^2_{1,2}$. Since the relation $A^2_{2,1}=B^2_{2,1}$ is already checked, the relations $A^2_{3,1}=B^2_{3,1}$ and $A^2_{2,2}=B^2_{2,2}$ are now also proved.

\subsection{Relation $A^2_{2,1,1}=B^2_{2,1,1}$}

Using equation~\eqref{eq:psi and A-class}, we compute
\begin{align}
&A^2_{2,1,1}-\psi_1 A^2_{1,1,1}=\notag\\
=&\sum_{\substack{I\sqcup J=\{1,2,3\}\\1\in I,\,|J|\ge 1}}\Coef_{a_1^2a_2a_3}\left(\frac{a_J}{\sum a_i}\lambda_2\DR_1\left(\widetilde{-\sum a_i},A_I,a_J\right)\boxtimes_1\DR_1(A_J,-a_J)\right)+\label{eq:A211B211,1}\\
&+\Coef_{a_1^2a_2a_3}\left(\frac{a_2+a_3}{\sum a_i}\lambda_2\DR_2\left(\widetilde{-\sum a_i},a_1,a_2+a_3\right)\boxtimes_1\DR_0(a_2,a_3,-a_2-a_3)\right).\label{eq:A211B211,2}
\end{align}
Let us look at a term in the sum in line~\eqref{eq:A211B211,1}. The class $\lambda_1\DR_1(A_J,-a_J)$ is a polynomial in the variables $a_j, j\in J$, and it doesn't depend on $a_1$. We have
$$
\frac{1}{a_1+a_2+a_3}\lambda_1\DR_1(\widetilde{-a_1-a_2-a_3},A_I,a_J)=(a_1+a_2+a_3)\lambda_1\in R^1(\oM_{1,|I|+1}).
$$
So, the polynomial class in the brackets in line~\eqref{eq:A211B211,1} depends on $a_1$ at most linearly. Therefore, the expression in line~\eqref{eq:A211B211,1} is equal to zero. Let us look at the expression in line~\eqref{eq:A211B211,2}. We can easily see that it is equal to
\begin{multline*}
2\cdot\Coef_{a^2b^2}(\gl^{2,0}_{1|2,3})_*\left(\frac{b}{a+b}\lambda_2\DR_2(\widetilde{-a-b},a,b)\times[\oM_{0,3}]\right)=\\
=2\cdot(\gl^{2|0}_{1|2,3})_*\left(A^2_{2,1}\times[\oM_{0,3}]\right).
\end{multline*}
As a result, we obtain 
$$
A^2_{2,1,1}=\psi_1 A^2_{1,1,1}+2\cdot(\gl^{2|0}_{1|2,3})_*\left(A^2_{2,1}\times[\oM_{0,3}]\right).
$$

On the other hand, we have
\begin{align*}
B^2_{2,1,1}=&\psi_1^2\psi_2\psi_3
-6\,\tikz[baseline=-1mm]{\draw (A)--(B);\legm{B}{30}{1};\leg{B}{-30};\leg{B}{0};\gg{2}{A};\gg{0}{B};\lab{A}{25}{4mm}{\psi^2};\lab{B}{53}{3.5mm}{\psi};}
-3\,\tikz[baseline=-1mm]{\draw (A)--(B);\legm{B}{30}{1};\leg{B}{0};\leg{B}{-30};\gg{2}{A};\gg{0}{B};\lab{B}{-60}{3.5mm}{\psi};\lab{A}{25}{4mm}{\psi^2};}
-\tikz[baseline=-1mm]{\draw (A)--(B);\leg{B}{45};\leg{B}{0};\leg{B}{-45};\gg{1}{A};\gg{1}{B};\lab{B}{64}{3.9mm}{\psi};\lab{B}{19}{4.5mm}{\psi};\lab{B}{-71}{3.7mm}{\psi};}
-\tikz[baseline=-1mm]{\draw (A)--(B);\legm{B}{30}{1};\leg{B}{0};\leg{B}{-30};\gg{1}{A};\gg{1}{B};\lab{B}{-56}{3.5mm}{\psi};\lab{B}{48}{4.1mm}{\psi^2}}\\
&+6\,\tikz[baseline=-1mm]{\draw (A)--(B)--(C);\leg{C}{-30};\legm{C}{30}{1};\leg{C}{0};\gg{1}{A};\gg{1}{B};\gg{0}{C};\lab{B}{28}{3.5mm}{\psi};\lab{C}{53}{3.5mm}{\psi};}
+3\,\tikz[baseline=-1mm]{\draw (A)--(B)--(C);\leg{C}{-30};\legm{C}{30}{1};\leg{C}{0};\gg{1}{A};\gg{1}{B};\gg{0}{C};\lab{B}{28}{3.5mm}{\psi};\lab{C}{-60}{3.5mm}{\psi};}.
\end{align*}
Using also formula~\eqref{eq:B2111}, we compute
\begin{align*}
B^2_{2,1,1}-\psi_1 B^2_{1,1,1}=&
-3\,\tikz[baseline=-1mm]{\draw (A)--(B);\legm{B}{30}{1};\leg{B}{0};\leg{B}{-30};\gg{2}{A};\gg{0}{B};\lab{B}{-60}{3.5mm}{\psi};\lab{A}{25}{4mm}{\psi^2};}
+3\,\tikz[baseline=-1mm]{\draw (A)--(B)--(C);\leg{C}{-30};\legm{C}{30}{1};\leg{C}{0};\gg{1}{A};\gg{1}{B};\gg{0}{C};\lab{B}{28}{3.5mm}{\psi};\lab{C}{-60}{3.5mm}{\psi};}
+2\,\tikz[baseline=-1mm]{\draw (A)--(B);\legm{A}{-40}{1};\leg{B}{30};\leg{B}{-30};\gg{2}{A};\gg{0}{B};\lab{A}{28}{3.5mm}{\psi};\lab{A}{-78}{3.7mm}{\psi^2}}\\
&-2\,\tikz[baseline=-1mm]{\draw (A)--(B)--(C);\legm{B}{-40}{1};\leg{C}{-30};\leg{C}{30};\gg{1}{A};\gg{1}{B};\gg{0}{C};\lab{B}{-78}{3.7mm}{\psi^2}}
-2\,\tikz[baseline=-1mm]{\draw (A)--(B)--(C);\leg{C}{-30};\leg{C}{30};\legm{B}{-40}{1};\gg{1}{A};\gg{1}{B};\gg{0}{C};\lab{B}{28}{3.5mm}{\psi};\lab{B}{-70}{3.5mm}{\psi};}=\\
=&2\,\tikz[baseline=-1mm]{\draw (A)--(B);\legm{A}{-40}{1};\leg{B}{30};\leg{B}{-30};\gg{2}{A};\gg{0}{B};\lab{A}{28}{3.5mm}{\psi};\lab{A}{-78}{3.7mm}{\psi^2}}
-6\,\tikz[baseline=-1mm]{\draw (A)--(B)--(C);\leg{C}{30};\leg{C}{-30};\legm{B}{-40}{1};\gg{2}{A};\gg{0}{B};\gg{0}{C};\lab{A}{25}{4mm}{\psi^2};}
-2\,\tikz[baseline=-1mm]{\draw (A)--(B)--(C);\leg{C}{-30};\leg{C}{30};\legm{B}{-40}{1};\gg{1}{A};\gg{1}{B};\gg{0}{C};\lab{B}{28}{3.5mm}{\psi};\lab{B}{-70}{3.5mm}{\psi};}\\
&-2\,\tikz[baseline=-1mm]{\draw (A)--(B)--(C);\legm{B}{-40}{1};\leg{C}{-30};\leg{C}{30};\gg{1}{A};\gg{1}{B};\gg{0}{C};\lab{B}{-78}{3.7mm}{\psi^2}}
+6\,\tikz[baseline=-1mm]{\draw (A)--(B)--(C)--(D);\leg{D}{30};\legm{C}{-40}{1};\leg{D}{-30};\gg{1}{A};\gg{1}{B};\gg{0}{C};\gg{0}{D};\lab{B}{28}{3.5mm}{\psi};}.
\end{align*}
Using~\eqref{eq:B21} we see that the last expression is equal to $2\cdot(\gl^{2|0}_{1|2,3})_*\left(B^2_{2,1}\times [\oM_{0,3}]\right)$ and we get
$$
B^2_{2,1,1}=\psi_1 B^2_{1,1,1}+2\cdot(\gl^{2|0}_{1|2,3})_*\left(B^2_{2,1}\times[\oM_{0,3}]\right).
$$
Since the relations $A^2_{2,1}=B^2_{2,1}$ and $A^2_{1,1,1}=B^2_{1,1,1}$ are proved, we conclude that relation $A^2_{2,1,1}=B^2_{2,1,1}$ is true. 

\subsection{Relation $A^2_{1,1,1,1}=B^2_{1,1,1,1}$}
We follow the same strategy, as in the previous section. Using equation~\eqref{eq:psi and A-class}, we compute
\begin{align}
&A^2_{1,1,1,1}-\psi_1 A^2_{0,1,1,1}=\notag\\
=&\sum_{\substack{I\sqcup J=\{1,2,3,4\}\\I\ni 1,\,|J|\ge 1}}\Coef_{a_1a_2a_3a_4}\left(\frac{a_J}{a}\lambda_2\DR_1(\widetilde{-a},A_I,a_J)\boxtimes_1\DR_1(A_J,-a_J)\right)+\label{eq:A1111B1111,1}\\
&+\sum_{\substack{I\sqcup J=\{1,2,3,4\}\\I\ni 1,\,|J|\ge 2}}\Coef_{a_1a_2a_3a_4}\left(\frac{a_J}{a}\lambda_2\DR_2(\widetilde{-a},A_I,a_J)\boxtimes_1\DR_0(A_J,-a_J)\right),\label{eq:A1111B1111,2}
\end{align}
where $a:=\sum_{i=1}^4 a_i$. Let us look at a term in the sum in line~\eqref{eq:A1111B1111,1}. We have $\frac{1}{a}\lambda_1\DR_1(\widetilde{-a},A_I,a_J)=a\lambda_1\in R^1(\oM_{1,|I|+1})$ and the class $\lambda_1\DR_1(A_J,-a_J)$ doesn't depend on the variables $a_i$, $i\in I$. Therefore, the coefficient of $a_1a_2a_3a_4$ can be non-zero only if $I=\{1\}$. So the expression in line~\eqref{eq:A1111B1111,1} is equal to
\begin{align}
&(\gl_{1|2,3,4}^{1|1})_*\left(\lambda_1\times\Coef_{a_2a_3a_4}((a_2+a_3+a_4)\lambda_1\DR_1(a_2,a_3,a_4,-a_2-a_3-a_4))\right)=\label{eq:A1111B1111,4}\\
=&-\tikz[baseline=-1mm]{\draw (A)--(B)--(C);\legm{A}{-90}{1};\leg{B}{-90};\leg{C}{-60};\leg{C}{-120};\gg{1}{A};\gg{1}{B};\gg{0}{C};\lab{A}{90}{3.5mm}{\lambda_1};\lab{B}{90}{3.5mm}{\lambda_1};}
-3\,\tikz[baseline=-1mm]{\draw (A)--(B)--(C);\legm{A}{-90}{1};\leg{C}{-90};\leg{C}{-60};\leg{C}{-120};\gg{1}{A};\gg{1}{B};\gg{0}{C};\lab{A}{90}{3.5mm}{\lambda_1};\lab{B}{90}{3.5mm}{\lambda_1};}
+2\,\tikz[baseline=-1mm]{\draw (A)--(B)--(C);\legm{A}{-90}{1};\leg{B}{-90};\leg{C}{-60};\leg{C}{-120};\gg{1}{A};\gg{0}{B};\gg{1}{C};\lab{A}{90}{3.5mm}{\lambda_1};\lab{C}{90}{3.5mm}{\lambda_1};}
+3\,\tikz[baseline=-1mm]{\draw (A)--(B)--(C);\legm{A}{-90}{1};\leg{B}{-60};\leg{B}{-120};\leg{C}{-90};\gg{1}{A};\gg{0}{B};\gg{1}{C};\lab{A}{90}{3.5mm}{\lambda_1};\lab{C}{90}{3.5mm}{\lambda_1};}\notag\\
&+3\,\tikz[baseline=-1mm]{\draw (A)--(B)--(C);\legm{A}{-90}{1};\leg{B}{-60};\leg{B}{-120};\leg{B}{-90};\gg{1}{A};\gg{0}{B};\gg{1}{C};\lab{A}{90}{3.5mm}{\lambda_1};\lab{C}{90}{3.5mm}{\lambda_1};}.\notag
\end{align}
The expression in line~\eqref{eq:A1111B1111,2} is equal to
$$
6(\gl_{1|2,3,4}^{2|0})_*\left(A^2_{1,2}\times [\oM_{0,4}]\right)+2\sum_{\substack{\{i,j,k\}=\{2,3,4\}\\j<k}}(\gl_{1,i|j,k}^{2|0})_*\left(A^2_{1,1,1}\times [\oM_{0,3}]\right).
$$
On the other hand, we have
\begin{align*}
B^2_{1,1,1,1}=&\psi_1\psi_2\psi_3\psi_4
-6\,\tikz[baseline=-1mm]{\draw (A)--(B);\leg{B}{30};\leg{B}{10};\leg{B}{-10};\leg{B}{-30};\gg{2}{A};\gg{0}{B};\lab{A}{25}{4mm}{\psi^2};\lab{B}{53}{3.5mm}{\psi};}
-2\,\tikz[baseline=-1mm]{\draw (A)--(B);\leg{A}{-30};\leg{B}{30};\leg{B}{0};\leg{B}{-30};\gg{2}{A};\gg{0}{B};\lab{A}{28}{3.5mm}{\psi};\lab{A}{-60}{3.5mm}{\psi};\lab{B}{53}{3.5mm}{\psi};}
-2\,\tikz[baseline=-1mm]{\draw (A)--(B);\leg{B}{-30};\leg{B}{-10};\leg{B}{10};\leg{B}{30};\gg{2}{A};\gg{0}{B};\lab{A}{28}{3.5mm}{\psi};\lab{B}{-60}{3.5mm}{\psi};\lab{B}{53}{3.5mm}{\psi};}
-\tikz[baseline=-1mm]{\draw (A)--(B);\leg{B}{-60};\leg{B}{-20};\leg{B}{20};\leg{B}{60};\gg{1}{A};\gg{1}{B};\lab{B}{75}{4.5mm}{\psi};\lab{B}{37}{4.5mm}{\psi};\lab{B}{-2}{4.5mm}{\psi};}\\
&-4\,\tikz[baseline=-1mm]{\draw (G)--(E);\draw (G)--(F);\leg{E}{-30};\leg{E}{30};\leg{F}{-30};\leg{E}{30};\gg{2}{G};\gg{0}{E};\gg{0}{F};\lab{G}{53}{3.5mm}{\psi};\lab{G}{-60}{3.5mm}{\psi};}
+4\,\tikz[baseline=-1mm]{\draw (A)--(B)--(C);\leg{B}{-30};\leg{C}{30};\leg{C}{0};\leg{C}{-30};\gg{2}{A};\gg{0}{B};\gg{0}{C};\lab{A}{28}{3.5mm}{\psi};\lab{C}{53}{3.5mm}{\psi};}
+6\,\tikz[baseline=-1mm]{\draw (A)--(B)--(C);\leg{C}{-30};\leg{C}{-10};\leg{C}{10};\leg{C}{30};\gg{1}{A};\gg{1}{B};\gg{0}{C};\lab{B}{28}{3.5mm}{\psi};\lab{C}{53}{3.5mm}{\psi};}
+2\,\tikz[baseline=-1mm]{\draw (A)--(B)--(C);\leg{C}{-30};\leg{C}{-10};\leg{C}{10};\leg{C}{30};\gg{1}{A};\gg{1}{B};\gg{0}{C};\lab{C}{-60}{3.5mm}{\psi};\lab{C}{53}{3.5mm}{\psi};}\\
&+2\,\tikz[baseline=-1mm]{\draw (A)--(B)--(C);\leg{B}{-30};\leg{C}{-30};\leg{C}{0};\leg{C}{30};\gg{1}{A};\gg{1}{B};\gg{0}{C};\lab{B}{-60}{3.5mm}{\psi};\lab{C}{53}{3.5mm}{\psi};}
+2\,\tikz[baseline=-1mm]{\draw (A)--(B)--(C);\leg{B}{-30};\leg{C}{-30};\leg{C}{0};\leg{C}{30};\gg{1}{A};\gg{1}{B};\gg{0}{C};\lab{B}{28}{3.5mm}{\psi};\lab{C}{53}{3.5mm}{\psi};}
+8\,\tikz[baseline=-1mm]{\draw (H)--(G)--(E);\draw (G)--(F);\leg{E}{-30};\leg{E}{30};\leg{F}{-30};\leg{F}{30};\gg{0}{G};\gg{0}{E};\gg{0}{F};\gg{2}{H};\lab{H}{28}{3.5mm}{\psi};}\\
&-4\,\tikz[baseline=-1mm]{\draw (A)--(B)--(C)--(D);\leg{C}{-30};\leg{D}{30};\leg{D}{0};\leg{D}{-30};\gg{1}{A};\gg{1}{B};\gg{0}{C};\gg{0}{D};\lab{D}{53}{3.5mm}{\psi};}
+4\,\tikz[baseline=-1mm]{\draw (H)--(G)--(E);\draw (G)--(F);\leg{E}{-30};\leg{E}{30};\leg{F}{-30};\leg{F}{30};\gg{1}{G};\gg{0}{E};\gg{0}{F};\gg{1}{H};\lab{G}{53}{3.5mm}{\psi};}
-8\,\tikz[baseline=-1mm]{\draw (I)--(H)--(G)--(E);\draw (G)--(F);\leg{E}{-30};\leg{E}{30};\leg{F}{-30};\leg{F}{30};\gg{1}{I};\gg{1}{H};\gg{0}{G};\gg{0}{E};\gg{0}{F};}\,.
\end{align*}
After a long direct computation, that uses only the genus $0$ topological recursion relation, we obtain
\begin{align}
B^2_{1,1,1,1}-\psi_1 B^2_{0,1,1,1}=&6(\gl_{1|2,3,4}^{2|0})_*\left(B^2_{1,2}\times [\oM_{0,4}]\right)\label{eq:A1111B1111,3}\\
&+2\sum_{\substack{\{i,j,k\}=\{2,3,4\}\\j<k}}(\gl_{1,i|j,k}^{2|0})_*\left(B^2_{1,1,1}\times [\oM_{0,3}]\right)\notag
\end{align}
\begin{align}
&+6\,\tikz[baseline=-1mm]{\draw (A)--(B)--(C);\legm{B}{-90}{1};\leg{C}{-120};\leg{C}{-90};\leg{C}{-60};\gg{1}{A};\gg{1}{B};\gg{0}{C};\lab{B}{25}{4mm}{\psi^2};}
+2\,\tikz[baseline=-1mm]{\draw (A)--(B)--(C);\legm{B}{-120}{1};\leg{B}{-60};\leg{C}{-120};\leg{C}{-60};\gg{1}{A};\gg{1}{B};\gg{0}{C};\lab{B}{28}{3.5mm}{\psi};\lab{B}{-38}{3.5mm}{\psi};}
-2\,\tikz[baseline=-1mm]{\draw (A)--(B)--(C)--(D);\legm{B}{-90}{1};\leg{C}{-90};\leg{D}{-120};\leg{D}{-60};\gg{1}{A};\gg{1}{B};\gg{0}{C};\gg{0}{D};\lab{B}{28}{3.5mm}{\psi};}\notag\\
&-\tikz[baseline=-1mm]{\draw (A)--(B);\legm{B}{-150}{1};\leg{B}{-110};\leg{B}{-70};\leg{B}{-30};\gg{1}{A};\gg{1}{B};\lab{B}{-12}{4.5mm}{\psi};\lab{B}{-55}{4.5mm}{\psi};\lab{B}{-94}{4.5mm}{\psi};}
+\left(\tikz[baseline=-1mm]{\draw (A)--(B);\legm{A}{-90}{1};\leg{B}{-120};\leg{B}{-90};\leg{B}{-60};\gg{1}{A};\gg{1}{B};\lab{B}{-38}{3.5mm}{\psi};\lab{B}{-138}{3.5mm}{\psi};\lab{A}{-70}{3.5mm}{\psi};}
-2\,\tikz[baseline=-1mm]{\draw (A)--(B)--(C);\legm{A}{-90}{1};\leg{B}{-90};\leg{C}{-120};\leg{C}{-60};\gg{1}{A};\gg{1}{B};\gg{0}{C};\lab{B}{-70}{3.5mm}{\psi};\lab{A}{-70}{3.5mm}{\psi};}
-2\,\tikz[baseline=-1mm]{\draw (A)--(B)--(C);\legm{A}{-90}{1};\leg{B}{-90};\leg{C}{-120};\leg{C}{-60};\gg{1}{A};\gg{1}{B};\gg{0}{C};\lab{B}{28}{3.5mm}{\psi};\lab{A}{-70}{3.5mm}{\psi};}\right.\notag\\
&\left.-6\,\tikz[baseline=-1mm]{\draw (A)--(B)--(C);\legm{A}{-90}{1};\leg{C}{-120};\leg{C}{-90};\leg{C}{-60};\gg{1}{A};\gg{1}{B};\gg{0}{C};\lab{B}{28}{3.5mm}{\psi};\lab{A}{-70}{3.5mm}{\psi};}
+2\,\tikz[baseline=-1mm]{\draw (A)--(B)--(C)--(D);\legm{A}{-90}{1};\leg{C}{-90};\leg{D}{-120};\leg{D}{-60};\gg{1}{A};\gg{1}{B};\gg{0}{C};\gg{0}{D};\lab{A}{-70}{3.5mm}{\psi};}\right).\notag
\end{align}
Using the formula
$$
\psi_1=\lambda_1+\tikz[baseline=-1mm]{\draw (A)--(B);\legm{B}{30}{1};\legm{B}{-30}{2};\gg{1}{A};\gg{0}{B};}\in R^1(\oM_{1,2}),
$$
we can rewrite the expression in brackets on the right-hand side of equation~\eqref{eq:A1111B1111,3} as
\begin{align}
&\label{eq:A1111B1111,5}\left(\tikz[baseline=-1mm]{\draw (A)--(B);\legm{A}{-90}{1};\leg{B}{-120};\leg{B}{-90};\leg{B}{-60};\gg{1}{A};\gg{1}{B};\lab{B}{-38}{3.5mm}{\psi};\lab{B}{-138}{3.5mm}{\psi};\lab{A}{90}{3.5mm}{\lambda_1};}
-2\,\tikz[baseline=-1mm]{\draw (A)--(B)--(C);\legm{A}{-90}{1};\leg{B}{-90};\leg{C}{-120};\leg{C}{-60};\gg{1}{A};\gg{1}{B};\gg{0}{C};\lab{B}{-70}{3.5mm}{\psi};\lab{A}{90}{3.5mm}{\lambda_1};}
-2\,\tikz[baseline=-1mm]{\draw (A)--(B)--(C);\legm{A}{-90}{1};\leg{B}{-90};\leg{C}{-120};\leg{C}{-60};\gg{1}{A};\gg{1}{B};\gg{0}{C};\lab{B}{28}{3.5mm}{\psi};\lab{A}{90}{3.5mm}{\lambda_1};}
-6\,\tikz[baseline=-1mm]{\draw (A)--(B)--(C);\legm{A}{-90}{1};\leg{C}{-120};\leg{C}{-90};\leg{C}{-60};\gg{1}{A};\gg{1}{B};\gg{0}{C};\lab{B}{28}{3.5mm}{\psi};\lab{A}{90}{3.5mm}{\lambda_1};}\right.\\
&+\left.2\,\tikz[baseline=-1mm]{\draw (A)--(B)--(C)--(D);\legm{A}{-90}{1};\leg{C}{-90};\leg{D}{-120};\leg{D}{-60};\gg{1}{A};\gg{1}{B};\gg{0}{C};\gg{0}{D};\lab{A}{90}{3.5mm}{\lambda_1};}\right)
+\tikz[baseline=-1mm]{\draw (AA)--(A)--(B);\legm{A}{-90}{1};\leg{B}{-120};\leg{B}{-90};\leg{B}{-60};\gg{1}{AA};\gg{0}{A};\gg{1}{B};\lab{B}{-38}{3.5mm}{\psi};\lab{B}{-138}{3.5mm}{\psi};}
-2\,\tikz[baseline=-1mm]{\draw (AA)--(A)--(B)--(C);\legm{A}{-90}{1};\leg{B}{-90};\leg{C}{-120};\leg{C}{-60};\gg{1}{AA};\gg{0}{A};\gg{1}{B};\gg{0}{C};\lab{B}{-70}{3.5mm}{\psi};}
-2\,\tikz[baseline=-1mm]{\draw (AA)--(A)--(B)--(C);\legm{A}{-90}{1};\leg{B}{-90};\leg{C}{-120};\leg{C}{-60};\gg{1}{AA};\gg{0}{A};\gg{1}{B};\gg{0}{C};\lab{B}{28}{3.5mm}{\psi};}\notag\\
&-6\,\tikz[baseline=-1mm]{\draw (AA)--(A)--(B)--(C);\legm{A}{-90}{1};\leg{C}{-120};\leg{C}{-90};\leg{C}{-60};\gg{1}{AA};\gg{0}{A};\gg{1}{B};\gg{0}{C};\lab{B}{28}{3.5mm}{\psi};}
+2\,\tikz[baseline=-1mm]{\draw (AA)--(A)--(B)--(C)--(D);\legm{A}{-90}{1};\leg{C}{-90};\leg{D}{-120};\leg{D}{-60};\gg{1}{AA};\gg{0}{A};\gg{1}{B};\gg{0}{C};\gg{0}{D};}.\notag
\end{align}
The expression on the right-hand side of equation~\eqref{eq:A1111B1111,4} has the form $(\gl^{1|1}_{1|2,3,4})_*(\lambda_1\times\alpha)$, where 
\begin{align*}
\alpha=&-\tikz[baseline=-1mm]{\draw (B)--(C);\legm{B}{180}{4};\leg{B}{-90};\leg{C}{-60};\leg{C}{-120};\gg{1}{B};\gg{0}{C};\lab{B}{90}{3.5mm}{\lambda_1};}
-3\,\tikz[baseline=-1mm]{\draw (B)--(C);\legm{B}{180}{4};\leg{C}{-90};\leg{C}{-60};\leg{C}{-120};\gg{1}{B};\gg{0}{C};\lab{B}{90}{3.5mm}{\lambda_1};}
+2\,\tikz[baseline=-1mm]{\draw (B)--(C);\legm{B}{180}{4};\leg{B}{-90};\leg{C}{-60};\leg{C}{-120};\gg{0}{B};\gg{1}{C};\lab{C}{90}{3.5mm}{\lambda_1};}
+3\,\tikz[baseline=-1mm]{\draw (B)--(C);\legm{B}{180}{4};\leg{B}{-60};\leg{B}{-120};\leg{C}{-90};\gg{0}{B};\gg{1}{C};\lab{C}{90}{3.5mm}{\lambda_1};}\\
&+3\,\tikz[baseline=-1mm]{\draw (B)--(C);\legm{B}{180}{4};\leg{B}{-60};\leg{B}{-120};\leg{B}{-90};\gg{0}{B};\gg{1}{C};\lab{C}{90}{3.5mm}{\lambda_1};}.
\end{align*}
The part in brackets in expression~\eqref{eq:A1111B1111,5} has the form $(\gl^{1|1}_{1|2,3,4})_*(\lambda_1\times\beta)$, where 
\begin{align*}
\beta=&\tikz[baseline=-1mm]{\legm{B}{180}{4};\leg{B}{-120};\leg{B}{-90};\leg{B}{-60};\gg{1}{B};\lab{B}{-38}{3.5mm}{\psi};\lab{B}{-138}{3.5mm}{\psi};}
-2\,\tikz[baseline=-1mm]{\draw (B)--(C);\legm{B}{180}{4};\leg{B}{-90};\leg{C}{-120};\leg{C}{-60};\gg{1}{B};\gg{0}{C};\lab{B}{-70}{3.5mm}{\psi};}
-2\,\tikz[baseline=-1mm]{\draw (B)--(C);\legm{B}{180}{4};\leg{B}{-90};\leg{C}{-120};\leg{C}{-60};\gg{1}{B};\gg{0}{C};\lab{B}{28}{3.5mm}{\psi};}
-6\,\tikz[baseline=-1mm]{\draw (B)--(C);\legm{B}{180}{4};\leg{C}{-120};\leg{C}{-90};\leg{C}{-60};\gg{1}{B};\gg{0}{C};\lab{B}{28}{3.5mm}{\psi};}\\
&+2\,\tikz[baseline=-1mm]{\draw (B)--(C)--(D);\legm{B}{180}{4};\leg{C}{-90};\leg{D}{-120};\leg{D}{-60};\gg{1}{B};\gg{0}{C};\gg{0}{D};}.
\end{align*}
Expressing all psi classes using the genus $1$ topological recursion relation and also using the WDVV relation, it is easy to show that $\alpha=\beta$. Since $A^2_{2,1}=B^2_{2,1}$ and $A^2_{1,1,1}=B^2_{1,1,1}$, we obtain
\begin{align}
B^2_{1,1,1,1}-A^2_{1,1,1,1}=&6\,\tikz[baseline=-1mm]{\draw (A)--(B)--(C);\legm{B}{-90}{1};\leg{C}{-120};\leg{C}{-90};\leg{C}{-60};\gg{1}{A};\gg{1}{B};\gg{0}{C};\lab{B}{25}{4mm}{\psi^2};}
+2\,\tikz[baseline=-1mm]{\draw (A)--(B)--(C);\legm{B}{-120}{1};\leg{B}{-60};\leg{C}{-120};\leg{C}{-60};\gg{1}{A};\gg{1}{B};\gg{0}{C};\lab{B}{28}{3.5mm}{\psi};\lab{B}{-38}{3.5mm}{\psi};}
-2\,\tikz[baseline=-1mm]{\draw (A)--(B)--(C)--(D);\legm{B}{-90}{1};\leg{C}{-90};\leg{D}{-120};\leg{D}{-60};\gg{1}{A};\gg{1}{B};\gg{0}{C};\gg{0}{D};\lab{B}{28}{3.5mm}{\psi};}\label{eq:A1111B1111,6}\\
&-\tikz[baseline=-1mm]{\draw (A)--(B);\legm{B}{-150}{1};\leg{B}{-110};\leg{B}{-70};\leg{B}{-30};\gg{1}{A};\gg{1}{B};\lab{B}{-12}{4.5mm}{\psi};\lab{B}{-55}{4.5mm}{\psi};\lab{B}{-94}{4.5mm}{\psi};}
+\tikz[baseline=-1mm]{\draw (AAA)--(A)--(B);\legm{A}{-90}{1};\leg{B}{-120};\leg{B}{-90};\leg{B}{-60};\gg{1}{AAA};\gg{0}{A};\gg{1}{B};\lab{B}{-38}{3.5mm}{\psi};\lab{B}{-138}{3.5mm}{\psi};}
-2\,\tikz[baseline=-1mm]{\draw (AAA)--(A)--(B)--(C);\legm{A}{-90}{1};\leg{B}{-90};\leg{C}{-120};\leg{C}{-60};\gg{1}{AAA};\gg{0}{A};\gg{1}{B};\gg{0}{C};\lab{B}{-70}{3.5mm}{\psi};}\notag\\
&-2\,\tikz[baseline=-1mm]{\draw (AAA)--(A)--(B)--(C);\legm{A}{-90}{1};\leg{B}{-90};\leg{C}{-120};\leg{C}{-60};\gg{1}{AAA};\gg{0}{A};\gg{1}{B};\gg{0}{C};\lab{B}{28}{3.5mm}{\psi};}
-6\,\tikz[baseline=-1mm]{\draw (AAA)--(A)--(B)--(C);\legm{A}{-90}{1};\leg{C}{-120};\leg{C}{-90};\leg{C}{-60};\gg{1}{AAA};\gg{0}{A};\gg{1}{B};\gg{0}{C};\lab{B}{28}{3.5mm}{\psi};}\notag\\
&+2\,\tikz[baseline=-1mm]{\draw (AAA)--(A)--(B)--(C)--(D);\legm{A}{-90}{1};\leg{C}{-90};\leg{D}{-120};\leg{D}{-60};\gg{1}{AAA};\gg{0}{A};\gg{1}{B};\gg{0}{C};\gg{0}{D};}.\notag
\end{align}

Define an operator $\mathrm{Sym}\colon R^*(\oM_{2,4})\to R^*(\oM_{2,4})$ by
$$
\mathrm{Sym}(\alpha):=\frac{1}{4!}\sum_{\sigma\in S_4}\sigma\alpha,\quad\alpha\in R^*(\oM_{2,4}),
$$
where the symmetric group $S_4$ acts on $\oM_{2,4}$ by permutations of marked points. Applying the operator $\mathrm{Sym}$ to both sides of equation~\eqref{eq:A1111B1111,6} we obtain
\begin{align*}
B^2_{1,1,1,1}-A^2_{1,1,1,1}=&\frac{3}{2}\,\tikz[baseline=-1mm]{\draw (A)--(B)--(C);\leg{B}{-90};\leg{C}{-120};\leg{C}{-90};\leg{C}{-60};\gg{1}{A};\gg{1}{B};\gg{0}{C};\lab{B}{25}{4mm}{\psi^2};}
+\frac{1}{2}\,\tikz[baseline=-1mm]{\draw (A)--(B)--(C);\leg{B}{-120};\leg{B}{-60};\leg{C}{-120};\leg{C}{-60};\gg{1}{A};\gg{1}{B};\gg{0}{C};\lab{B}{28}{3.5mm}{\psi};\lab{B}{-38}{3.5mm}{\psi};}
-\frac{1}{2}\,\tikz[baseline=-1mm]{\draw (A)--(B)--(C)--(D);\leg{B}{-90};\leg{C}{-90};\leg{D}{-120};\leg{D}{-60};\gg{1}{A};\gg{1}{B};\gg{0}{C};\gg{0}{D};\lab{B}{28}{3.5mm}{\psi};}\\
&-\frac{1}{4}\tikz[baseline=-1mm]{\draw (A)--(B);\leg{B}{-150};\leg{B}{-110};\leg{B}{-70};\leg{B}{-30};\gg{1}{A};\gg{1}{B};\lab{B}{-12}{4.5mm}{\psi};\lab{B}{-55}{4.5mm}{\psi};\lab{B}{-94}{4.5mm}{\psi};}
+\frac{1}{4}\tikz[baseline=-1mm]{\draw (AAA)--(A)--(B);\leg{A}{-90};\leg{B}{-120};\leg{B}{-90};\leg{B}{-60};\gg{1}{AAA};\gg{0}{A};\gg{1}{B};\lab{B}{-38}{3.5mm}{\psi};\lab{B}{-138}{3.5mm}{\psi};}
-\frac{1}{2}\,\tikz[baseline=-1mm]{\draw (AAA)--(A)--(B)--(C);\leg{A}{-90};\leg{B}{-90};\leg{C}{-120};\leg{C}{-60};\gg{1}{AAA};\gg{0}{A};\gg{1}{B};\gg{0}{C};\lab{B}{-70}{3.5mm}{\psi};}\\
&-\frac{1}{2}\,\tikz[baseline=-1mm]{\draw (AAA)--(A)--(B)--(C);\leg{A}{-90};\leg{B}{-90};\leg{C}{-120};\leg{C}{-60};\gg{1}{AAA};\gg{0}{A};\gg{1}{B};\gg{0}{C};\lab{B}{28}{3.5mm}{\psi};}
-\frac{3}{2}\,\tikz[baseline=-1mm]{\draw (AAA)--(A)--(B)--(C);\leg{A}{-90};\leg{C}{-120};\leg{C}{-90};\leg{C}{-60};\gg{1}{AAA};\gg{0}{A};\gg{1}{B};\gg{0}{C};\lab{B}{28}{3.5mm}{\psi};}\\
&+\frac{1}{2}\,\tikz[baseline=-1mm]{\draw (AAA)--(A)--(B)--(C)--(D);\leg{A}{-90};\leg{C}{-90};\leg{D}{-120};\leg{D}{-60};\gg{1}{AAA};\gg{0}{A};\gg{1}{B};\gg{0}{C};\gg{0}{D};}.
\end{align*}
We see that the expression on the right-hand side has the form $(\gl^{1|1}_{|1,2,3,4})_*\left(\oM_{1,1}\times\rho\right)$, where
\begin{align*}
\rho=&\frac{3}{2}\,\tikz[baseline=-1mm]{\draw (B)--(C);\legm{B}{180}{5};\leg{B}{-90};\leg{C}{-120};\leg{C}{-90};\leg{C}{-60};\gg{1}{B};\gg{0}{C};\lab{B}{25}{4mm}{\psi^2};}
+\frac{1}{2}\,\tikz[baseline=-1mm]{\draw (B)--(C);\legm{B}{180}{5};\leg{B}{-120};\leg{B}{-60};\leg{C}{-120};\leg{C}{-60};\gg{1}{B};\gg{0}{C};\lab{B}{28}{3.5mm}{\psi};\lab{B}{-38}{3.5mm}{\psi};}
-\frac{1}{2}\,\tikz[baseline=-1mm]{\draw (B)--(C)--(D);\legm{B}{180}{5};\leg{B}{-90};\leg{C}{-90};\leg{D}{-120};\leg{D}{-60};\gg{1}{B};\gg{0}{C};\gg{0}{D};\lab{B}{28}{3.5mm}{\psi};}
-\frac{1}{4}\tikz[baseline=-1mm]{\legm{B}{180}{5};\leg{B}{-150};\leg{B}{-110};\leg{B}{-70};\leg{B}{-30};\gg{1}{B};\lab{B}{-12}{4.5mm}{\psi};\lab{B}{-55}{4.5mm}{\psi};\lab{B}{-94}{4.5mm}{\psi};}\\
&+\frac{1}{4}\tikz[baseline=-1mm]{\draw (A)--(B);\legm{A}{180}{5};\leg{A}{-90};\leg{B}{-120};\leg{B}{-90};\leg{B}{-60};\gg{0}{A};\gg{1}{B};\lab{B}{-38}{3.5mm}{\psi};\lab{B}{-138}{3.5mm}{\psi};}
-\frac{1}{2}\,\tikz[baseline=-1mm]{\draw (A)--(B)--(C);\legm{A}{180}{5};\leg{A}{-90};\leg{B}{-90};\leg{C}{-120};\leg{C}{-60};\gg{0}{A};\gg{1}{B};\gg{0}{C};\lab{B}{-70}{3.5mm}{\psi};}
-\frac{1}{2}\,\tikz[baseline=-1mm]{\draw (A)--(B)--(C);\legm{A}{180}{5};\leg{A}{-90};\leg{B}{-90};\leg{C}{-120};\leg{C}{-60};\gg{0}{A};\gg{1}{B};\gg{0}{C};\lab{B}{28}{3.5mm}{\psi};}\\
&-\frac{3}{2}\,\tikz[baseline=-1mm]{\draw (A)--(B)--(C);\legm{A}{180}{5};\leg{A}{-90};\leg{C}{-120};\leg{C}{-90};\leg{C}{-60};\gg{0}{A};\gg{1}{B};\gg{0}{C};\lab{B}{28}{3.5mm}{\psi};}
+\frac{1}{2}\,\tikz[baseline=-1mm]{\draw (A)--(B)--(C)--(D);\legm{A}{180}{5};\leg{A}{-90};\leg{C}{-90};\leg{D}{-120};\leg{D}{-60};\gg{0}{A};\gg{1}{B};\gg{0}{C};\gg{0}{D};}.
\end{align*}
It is sufficient to prove that $\rho=0$. For this we express all the psi classes using the genus~$1$ topological recursion relation, and then prove that $\rho=0$ using the WDVV relation and Getzler's relation. This computation is straightforward, but quite long, so we present here only the most interesting parts of it. Expressing all the psi classes we obtain
$$
\rho=\lambda_1\theta_1+\theta_2,
$$
where $\theta_1\in R^2(\oM_{1,5})$ and $\theta_2\in R^3(\oM_{1,5})$ are sums of boundary strata. Using the WDVV relation it is not hard to prove that $\theta_2=0$. For the class $\lambda_1\theta_1$ we get the following expression:
\begin{align*}
\lambda_1\theta_1=&a_1^{0,1,3}+\frac{1}{6} a_1^{0,2,2}-\frac{1}{6} a_1^{1,1,2}-\frac{1}{4} a_2^{0,0,4}+\frac{7}{8} a_2^{0,1,3}+\frac{1}{8}a_2^{0,2,2}+\frac{5}{4} a_2^{1,0,3}+\frac{1}{4} a_2^{1,1,2}\\
&+a_2^{2,0,2}-\frac{7}{16} a_3^{0,1,3}-\frac{3}{4}a_3^{0,2,2}+\frac{1}{16} a_3^{0,3,1}-\frac{7}{12} a_3^{1,1,2}+\frac{1}{12} a_3^{1,2,1}+b_1^{0,2,2}\\
&-\frac{3}{2}b_2^{0,1,3}-\frac{1}{2} b_2^{1,1,2},
\end{align*}
where we use the following notations:
\begin{align*}
&a^{i,j,k}_1:=\tikz[baseline=-1mm]{\draw (A)--(B)--(C);\legm{A}{180}{5};\leg{A}{-60};\leg{A}{-120};\leg{B}{-60};\leg{B}{-120};\leg{C}{-60};\leg{C}{-120};\gg{1}{A};\gg{0}{B};\gg{0}{C};\lab{A}{-90}{5.9mm}{...};\lab{A}{-90}{7mm}{\underbrace{\phantom{aaa}}_{\text{$i$ legs}}};\lab{B}{-90}{5.9mm}{...};\lab{B}{-90}{7mm}{\underbrace{\phantom{aaa}}_{\text{$j$ legs}}};\lab{C}{-90}{5.9mm}{...};\lab{C}{-90}{7mm}{\underbrace{\phantom{aaa}}_{\text{$k$ legs}}};\lab{A}{90}{3.5mm}{\lambda_1};}&&
a^{i,j,k}_2:=\tikz[baseline=-1mm]{\draw (A)--(B)--(C);\legm{B}{90}{5};\leg{A}{-60};\leg{A}{-120};\leg{B}{-60};\leg{B}{-120};\leg{C}{-60};\leg{C}{-120};\gg{1}{A};\gg{0}{B};\gg{0}{C};\lab{A}{-90}{5.9mm}{...};\lab{A}{-90}{7mm}{\underbrace{\phantom{aaa}}_{\text{$i$ legs}}};\lab{B}{-90}{5.9mm}{...};\lab{B}{-90}{7mm}{\underbrace{\phantom{aaa}}_{\text{$j$ legs}}};\lab{C}{-90}{5.9mm}{...};\lab{C}{-90}{7mm}{\underbrace{\phantom{aaa}}_{\text{$k$ legs}}};\lab{A}{90}{3.5mm}{\lambda_1};}\\
&a^{i,j,k}_3:=\tikz[baseline=-1mm]{\draw (A)--(B)--(C);\legm{C}{0}{5};\leg{A}{-60};\leg{A}{-120};\leg{B}{-60};\leg{B}{-120};\leg{C}{-60};\leg{C}{-120};\gg{1}{A};\gg{0}{B};\gg{0}{C};\lab{A}{-90}{5.9mm}{...};\lab{A}{-90}{7mm}{\underbrace{\phantom{aaa}}_{\text{$i$ legs}}};\lab{B}{-90}{5.9mm}{...};\lab{B}{-90}{7mm}{\underbrace{\phantom{aaa}}_{\text{$j$ legs}}};\lab{C}{-90}{5.9mm}{...};\lab{C}{-90}{7mm}{\underbrace{\phantom{aaa}}_{\text{$k$ legs}}};\lab{A}{90}{3.5mm}{\lambda_1};}&&
b^{0,2,2}_1:=\tikz[baseline=-1mm]{\draw (A)--(B)--(C);\legm{B}{-90}{5};\leg{A}{-60};\leg{A}{-120};\leg{C}{-60};\leg{C}{-120};\gg{0}{A};\gg{1}{B};\gg{0}{C};\lab{B}{90}{3.5mm}{\lambda_1};}\\
&b^{i,j,k}_2:=\tikz[baseline=-1mm]{\draw (A)--(B)--(C);\legm{A}{180}{5};\leg{A}{-60};\leg{A}{-120};\leg{B}{-60};\leg{B}{-120};\leg{C}{-60};\leg{C}{-120};\gg{0}{A};\gg{1}{B};\gg{0}{C};\lab{A}{-90}{5.9mm}{...};\lab{A}{-90}{7mm}{\underbrace{\phantom{aaa}}_{\text{$j$ legs}}};\lab{B}{-90}{5.9mm}{...};\lab{B}{-90}{7mm}{\underbrace{\phantom{aaa}}_{\text{$i$ legs}}};\lab{C}{-90}{5.9mm}{...};\lab{C}{-90}{7mm}{\underbrace{\phantom{aaa}}_{\text{$k$ legs}}};\lab{B}{90}{3.5mm}{\lambda_1};}&&
\end{align*}

Consider Getzler's relation~\eqref{eq:Getzler relation}. Let $\pi\colon\oM_{1,5}\to\oM_{1,4}$ be the forgetful map that forgets the last marked point. We have
\begin{align*}
\lambda_1\pi^*\gamma=&\frac{1}{2} a_1^{0,1,3}-\frac{1}{6} a_1^{0,2,2}-\frac{1}{3} a_1^{1,1,2}+\frac{1}{2} a_2^{0,1,3}-\frac{1}{6}a_2^{0,2,2}-\frac{1}{3} a_2^{1,1,2}+\frac{1}{2} a_3^{0,1,3}\\
&-\frac{1}{6} a_3^{0,2,2}-\frac{1}{3}a_3^{1,1,2}+b_1^{0,2,2}+b_2^{0,2,2}=0\in R^3(\oM_{1,5}).
\end{align*} 
Let $\pi'\colon\oM_{1,5}\to\oM_{1,4}$ be the forgetful map that forgets the first marked point. We assume, that after forgetting the first marked point, a point labeled by $i$, $i\ge 2$, on a curve from $\oM_{1,5}$ becomes a point labeled by $i-1$ on a curve in $\oM_{1,4}$. The symmetric group $S_4$ acts on $\oM_{1,5}$ by permutations of the first four marked points. Define a map $\mathrm{Sym}'\colon R^*(\oM_{1,5})\to R^*(\oM_{1,5})$ by
$$
\mathrm{Sym}'(\alpha):=\frac{1}{4!}\sum_{\sigma\in S_4}\sigma\alpha,\quad\alpha\in R^*(\oM_{1,5}).
$$
We have
\begin{align*}
\lambda_1\mathrm{Sym}'((\pi')^*\gamma)=&-\frac{1}{4} a_1^{0,1,3}-\frac{1}{6} a_1^{0,2,2}-\frac{1}{12} a_1^{1,1,2}+\frac{1}{2} a_2^{0,0,4}-\frac{1}{12}a_2^{0,2,2}-\frac{1}{8} a_2^{1,0,3}\\
&-\frac{1}{8} a_2^{1,1,2}-\frac{1}{6} a_2^{2,0,2}+\frac{3}{8} a_3^{0,1,3}+\frac{1}{6}a_3^{0,2,2}-\frac{1}{8} a_3^{0,3,1}-\frac{1}{24} a_3^{1,1,2}\\
&-\frac{5}{24} a_3^{1,2,1}-\frac{1}{6} a_3^{2,1,1}+\frac{3}{4}b_2^{0,1,3}+\frac{1}{2} b_2^{0,2,2}+\frac{1}{4} b_2^{1,1,2}=0\in R^3(\oM_{1,5}).
\end{align*}
We compute
\begin{align*}
\lambda_1\theta_1=&\lambda_1\theta_1-\lambda_1\pi^*\gamma+2\lambda_1\mathrm{Sym}'((\pi')^*\gamma)=\\
=&\frac{3}{4} a_2^{0,0,4}+\frac{3}{8} a_2^{0,1,3}+\frac{1}{8} a_2^{0,2,2}+a_2^{1,0,3}+\frac{1}{3} a_2^{1,1,2}+\frac{2}{3}a_2^{2,0,2}-\frac{3}{16} a_3^{0,1,3}-\frac{1}{4} a_3^{0,2,2}\\
&-\frac{3}{16} a_3^{0,3,1}-\frac{1}{3} a_3^{1,1,2}-\frac{1}{3}a_3^{1,2,1}-\frac{1}{3} a_3^{2,1,1}.
\end{align*}
Finally, applying the WDVV relations
\begin{align*}
a_3^{2,1,1}=&2a_2^{2,0,2},\\
a_2^{1,0,3}=&-\frac{1}{3}a_2^{1,1,2}+\frac{1}{3}a_3^{1,1,2}+\frac{1}{3}a_3^{1,2,1},\\
a_2^{0,0,4}=&-\frac{1}{2}a_2^{0,1,3}-\frac{1}{6}a_2^{0,2,2}+\frac{1}{4}a_3^{0,1,3}+\frac{1}{3}a_3^{0,2,2}+\frac{1}{4}a_3^{0,3,1},
\end{align*}
it is easy to see that $\lambda_1\theta_1=0$. The relation $A^2_{1,1,1,1}=B^2_{1,1,1,1}$ is proved.
}

%
%
%
%


\begin{thebibliography}

\bibitem{Bur15} A. Buryak, {\it Double ramification cycles and integrable hierarchies}, Communications in Mathematical Physics~336 (2015), no.~3, 1085--1107.

\bibitem{Bur17} A. Buryak, {\it New approaches to the integrable hierarchies of topological type}, (Russian) Uspekhi Matematicheskikh Nauk 72 (2017), no. 5(437), 63--112. Translated in Russian Mathematical Surveys 72 (2017), no. 5, 841--887.

\bibitem{BDGR16a} A. Buryak, B. Dubrovin, J. Gu\'er\'e, P. Rossi, {\it Tau-structure for the double ramification hierarchies}, Communications in Mathematical Physics 363 (2018), no. 1, 191--260.

\bibitem{BDGR16b} A. Buryak, B. Dubrovin, J. Gu\'er\'e, P. Rossi, {\it Integrable systems of double ramification type}, arXiv:1609.04059. 

\bibitem{BG16} A. Buryak, J. Gu\'er\'e, {\it Towards a description of the double ramification hierarchy for Witten's $r$-spin class}, Journal de Math\'ematiques Pures et Appliqu\'ees. Neuvi\`eme S\'erie~(9)~106 (2016), no.~5, 837--865.

\bibitem{BPS12} A. Buryak, H. Posthuma, S. Shadrin, {\it A polynomial bracket for the Dubrovin-Zhang hierarchies}, Journal of Differential Geometry~92 (2012), no.~1, 153--185.

\bibitem{BR16a} A. Buryak, P. Rossi, {\it Recursion relations for double ramification hierarchies}, Communications in Mathematical Physics 342 (2016), no. 2, 533--568.

\bibitem{BR16b} A. Buryak, P. Rossi, {\it Double ramification cycles and quantum integrable systems}, Letters in Mathematical Physics 106 (2016), no. 3, 289--317.

\bibitem{BSSZ15} A. Buryak, S. Shadrin, L. Spitz, D. Zvonkine, {\it Integrals of $\psi$-classes over double ramification cycles}, American Journal of Mathematics~137 (2015), no.~3, 699--737.

\bibitem{DZ05} B. Dubrovin, Y. Zhang, {\it Normal forms of hierarchies of integrable PDEs, Frobenius manifolds and Gromov-Witten invariants}, a new 2005 version of arXiv:math/0108160, 295 pages.

\bibitem{Fab90} C. Faber, {\it Chow rings of moduli spaces of curves. I. The Chow ring of $\oM_3$}, Annals of Mathematics. Second Series~132 (1990), no.~2, 331--419. 

\bibitem{FP00} C.~Faber, R.~Pandharipande, {\it Hodge integrals and Gromov-Witten theory}, Inventiones Mathematicae~139 (2000), no. 1, 173--199.

\bibitem{FP05} C. Faber, R. Pandharipande, {\it Relative maps and tautological classes}, Journal of the European Mathematical Society~7 (2005), no. 1, 13--49.

\bibitem{Get97} E. Getzler, {\it Intersection theory on $\oM_{1,4}$ and elliptic Gromov-Witten invariants}, Journal of the American Mathematical Society~10 (1997), no.~4, 973--998.

\bibitem{Get98} E. Getzler, {\it Topological recursion relations in genus $2$}, Integrable systems and algebraic geometry (Kobe/Kyoto, 1997), 73--106, World Sci. Publ., River Edge, NJ, 1998.

\bibitem{Hai13} R. Hain, {\it Normal functions and the geometry of moduli spaces of curves}, Handbook of moduli, Vol.~I, 527--578, Adv. Lect. Math. (ALM), 24, Int. Press, Somerville, MA, 2013.

\bibitem{Ionel} E.-N. Ionel, {\it Topological recursive relations in $H^{2g}(\cM_{g,n})$}, Inventiones Mathematicae 148 (2002), no. 3, 627--658.

\bibitem{Jan15} F. Janda, {\it Relations on $\oM_{g,n}$ via equivariant Gromov-Witten theory of $\mbP^1$}, arXiv:1509.08421.

\bibitem{JPPZ16} F. Janda, R. Pandharipande, A. Pixton, D. Zvonkine, {\it Double ramification cycles on the moduli spaces of curves}, Publications math\'ematiques de l'IH\'ES 125 (2017), no. 1, 221--266.

\bibitem{KM94} M. Kontsevich, Yu. Manin, {\it Gromov-Witten classes, quantum cohomology, and enumerative geometry}, Communications in Mathematical Physics~164 (1994), no. 3, 525--562.

\bibitem{MW13} S. Marcus, J. Wise, {\it Stable maps to rational curves and the relative Jacobian}, arXiv:1310.5981.

\bibitem{PPZ15} R.~Pandharipande, A.~Pixton, D.~Zvonkine, {\it Relations on $\oM_{g,n}$ via $3$-spin structures}, Journal of the American Mathematical Society~28 (2015), no.~1, 279--309.

\bibitem{Ros17} P. Rossi, {\it Integrability, quantization and moduli spaces of curves}, Symmetry, Integrability and Geometry: Methods and Applications (SIGMA) 13 (2017), Paper no. 60, 29 pages.

\end{thebibliography}
\end{document}